\documentclass[a4paper,10pt]{amsart}

\usepackage[OT2,T1]{fontenc}
\usepackage[utf8]{inputenc}
\usepackage[russian,english]{babel}
\usepackage{
   amsmath
  ,amsfonts
  ,amssymb
  ,amsthm
  ,leftidx
  ,cite
  ,datetime
  ,enumitem
  ,stmaryrd
  ,todonotes
  ,xspace
  ,xparse
  ,hyphenat
  ,epigraph
  ,kodi
  ,mathtools
}
\usepackage[all,cmtip,2cell]{xy}\UseAllTwocells

\usepackage[cal=pxtx]{mathalfa}
\usepackage{hyperref}
\usepackage{cjhebrew}

\usepackage{tikz,tikz-cd}
\usetikzlibrary{arrows}
\usetikzlibrary{babel}

\tikzset{%
implies/.style={double,double equal sign distance,-implies},
shorten <>/.style={shorten >=#1,shorten <=#1}}

\newlength{\spacing}
\def\scaling{.2}
\newlength{\raising}
\def\drawboxvoid{
\begin{tikzpicture}[scale=\scaling]
\draw (0,0) rectangle (1,1);%
\end{tikzpicture}}

\def\boxvoid   {\raisebox{\raising}{\drawboxvoid}}

\newcommand{\htth}{{\scriptsize $\triangle$}\textsc{tth}\@\xspace}
\newcommand{\hfs}{{\scriptsize $\triangle$}\textsc{fs}\@\xspace}
\newcommand{\phfs}{{\scriptsize $\triangle$}\textsc{pfs}\@\xspace}

\def\hdisplayed{\mathbin{\ooalign{\hfil\raisebox{.255em}{$\shortmid$}\hfil\cr\raisebox{0em}{\scalebox{1}[.6]{$\approx$}}\cr}}}
\def\hnormaled{\hdisplayed}
\def\hscripted{\ooalign{\hfil\raisebox{.18em}{$\scriptstyle\shortmid$}\hfil\cr\raisebox{0em}{\scalebox{1}[.6]{$\scriptstyle\approx$}}\cr}}
\def\hscriptscripted{\ooalign{\hfil\raisebox{.12em}{$\scriptscriptstyle\shortmid$}\hfil\cr\raisebox{0em}{\scalebox{1}[.6]{$\scriptscriptstyle\approx$}}\cr}}
\def\cdisplayed{\mathbin{\ooalign{\hfil\raisebox{.225em}{$\shortmid$}\hfil\cr\raisebox{-.12em}{\scalebox{1}[1]{$=$}}\cr}}}
\def\cnormaled{\cdisplayed}
\def\cscripted{\ooalign{\hfil\raisebox{.14em}{$\scriptstyle\shortmid$}\hfil\cr\raisebox{-.12em}{\scalebox{1}[1]{$\scriptstyle =$}}\cr}}
\def\cscriptscripted{\ooalign{\hfil\raisebox{.05em}{$\scriptscriptstyle\shortmid$}\hfil\cr\raisebox{-.12em}{\scalebox{1}[1]{$\scriptscriptstyle =$}}\cr}}
\def\horth{\mathchoice{\hdisplayed}{\hnormaled}{\hscripted}{\hscriptscripted}}%
\def\corth{\mathchoice{\cdisplayed}{\cnormaled}{\cscripted}{\cscriptscripted}}%

\newcommand{\var}[2]{\left[ \begin{smallmatrix} #1 \\ \downarrow \\ #2 \end{smallmatrix}\right]}
\newcommand{\smat}[1]{\left[ \begin{smallmatrix} #1 \end{smallmatrix}\right]}

\def\lhorth#1{\leftidx{^{\horth}}{#1}{}}

\def\mho{\rotatebox[origin=c]{180}{$\Omega$}}

\def\ass{\text{\cjRL{M}}}
\def\uni{\text{\cjRL{h}}}
\def\yu{\text{\begin{otherlanguage*}{russian}ю\end{otherlanguage*}}}
\def\me{\text{\begin{otherlanguage*}{russian}м\end{otherlanguage*}}}

\def\tee{\mathfrak{t}}
\def\t{\tee}

\DeclareMathOperator{\id}{id}

\DeclareMathOperator{\hocolim}{hocolim}

\def\Nat{\textsf{Nat}}

\usepackage{CJKutf8}

\newcommand{\yon}{\text{\begin{CJK}{UTF8}{min}よ\end{CJK}}\!}

\DeclareMathAlphabet\EuScript{U}{eus}{m}{n}
\SetMathAlphabet\EuScript{bold}{U}{eus}{b}{n}

\def\cate#1{\textbf{#1}}

	\def\C  {\cate{C}}
	\def\iC{\EuScript{C}}
		\def\sD  {\cate{D}}
	\def\T  {\cate{T}}
	
	\def\cD {\cate{D}}
	\def\cat{\cate{Cat}}
    \def\Cat{\cate{CAT}}
	\def\CAT{\cate{CAT}}
	
	\def\fincat{\cate{fdCat}}
	\def\PDer{\cate{PDer}}
	\def\Der{\cate{Der}}
		\def\Fun{\mathrm{Fun}}
	\def\bDelta{\boldsymbol{\Delta}}

	\def\F {\mathfrak{F}}
	
\def\fs#1{\mathbb{#1}}
	\def\D {\fs{D}}

\def\class#1{\mathcal{#1}}
	\def\M{\class{M}}
	\def\E{\class{E}}
	\def\A{\cate{A}}

\newcommand{\xto}[1]{\xrightarrow{#1}}

\def\Dia{\cate{Dia}}
\def\dia{\text{dia}}
\def\opp{\text{op}}
\def\ho{\textsf{Ho}}

\def\dfs{\textsc{dfs}\@\xspace}
\def\pt{\text{pt}}
\def\tre{\textbf{3}}
\def\due{\textbf{2}}
\def\uno{\textbf{1}}

\def\numbers#1{\mathbb{#1}}
	\def\Z{\numbers{Z}}
	\def\N{\numbers{N}}

\DeclareFontEncoding{LS1}{}{}
\DeclareFontSubstitution{LS1}{stix}{m}{n}
\DeclareSymbolFont{symbols2}{LS1}{stixfrak}{m}{n}

\DeclareMathSymbol{\lmark}{\mathalpha}{symbols2}{"C0}
\DeclareMathSymbol{\rmark}{\mathalpha}{symbols2}{"C1}
\DeclareMathSymbol{\lrmark}{\mathalpha}{symbols2}{"C2}

\newcommand{\Nearrow}{\rotatebox[origin=c]{45}{$\Rightarrow$}}

\newcommand{\Searrow}{\rotatebox[origin=c]{-45}{$\Rightarrow$}}
\newcommand{\Swarrow}{\rotatebox[origin=c]{225}{$\Rightarrow$}}

\def\restriction{|}
\def\TT{\textsf{T}}

\usepackage{bm}
\def\circledast{\protect{\tiny \oast}}
\def\circledbang{\protect{\rotatebox[origin=c]{90}{\tiny $\ominus$}}}
\def\ee{e}
\def\mm{m}
\def\cpfs{\textsc{dpfs}\@\xspace}
\def\dpfs{\textsc{dpfs}\@\xspace}

\usepackage{turnstile}
	\newcommand{\adjunct}[2]{\nsststile{#2}{#1}}

\setlist[1]{itemsep=0pt}

\renewcommand{\textbf}[1]{\text{\fontseries{b}\selectfont{\upshape #1}}}

\newtheoremstyle{reference}%
   {}
   {}
   {}                      
   {}                      
   {\bfseries}             
   {:}                     
   {.2em}                  
   {\thmname{#1}           
    \thmnumber{#2}         
    \thmnote{{\sc [#3]}}}  

\numberwithin{equation}{section}

\theoremstyle{reference}
  \newtheorem{theorem}{Theorem}[section]
  \newtheorem{lemma}[theorem]{Lemma}
  \newtheorem{proposition}[theorem]{Proposition}

  \newtheorem{remark}[theorem]{Remark}
  \newtheorem{definition}[theorem]{Definition}
  \newtheorem{corollary}[theorem]{Corollary}
  \newtheorem{notat}[theorem]{Notation}

   \newtheorem{setting}[theorem]{Setting}
  \newtheorem*{theorem*}{Theorem}
  \newtheorem*{lemma*}{Lemma}
  \newtheorem*{proposition*}{Proposition}
  \newtheorem*{example*}{Example}
  \newtheorem*{exercise*}{Exercise}
  \newtheorem*{remark*}{Remark}
  \newtheorem*{definition*}{Definition}
  \newtheorem*{corollary*}{Corollary}
  \newtheorem*{notat*}{Notation}
  \newtheorem*{scholium*}{Scholium}
  \newtheorem*{counterex*}{Counterexample}
  \newtheorem*{conjec*}{Conjecture}
  \newtheorem*{quest*}{Question}

\hypersetup{%
  pdftoolbar=   true,
  pdfmenubar=   true,
  pdffitwindow= true,
  pdftitle=     {t-derivators},
  pdfauthor=    {Loregian - Virili},
  colorlinks=   true,
  linkcolor=    black,
  citecolor=    blue!40!black}

\providecommand{\refbf}[1]{\textbf{\ref{#1}}}
\makeatletter
  \def\@cite#1#2{[\textbf{#1}\if@tempswa , #2\fi]}
  \def\@biblabel#1{[\textsf{#1}]}
\makeatother

\providecommand{\abbrv}[1]{#1.\@\xspace}
  \providecommand{\ie}{\abbrv{i.e}}

  \providecommand{\achap}{\abbrv{Ch}}
  \providecommand{\adef}{\abbrv{Def}}
  \providecommand{\acor}{\abbrv{Cor}}
  \providecommand{\aprop}{\abbrv{Prop}}
  \providecommand{\athm}{\abbrv{Thm}}

\newlength{\seplen}
\newlength{\sepwid}
\def\firstblank{\,\rule{\seplen}{\sepwid}\,}
\def\secondblank{\firstblank\llap{\raisebox{2pt}{\firstblank}}}

\allowdisplaybreaks

\title{Factorization systems on (stable) derivators}
\author{Fosco Loregian}
\address{%
\textsf{Fosco Loregian}: \newline
Department of Mathematics and Statistics             \newline
Masaryk University, Faculty of Sciences              \newline
Kotl\'{a}\v{r}sk\'{a} 2, 611 37 Brno, Czech Republic \newline
\href{mailto:loregianf@math.muni.cz}
     {\sf loregianf@math.muni.cz}}
\thanks{The first-named author is supported by the Grant Agency of the 
        Czech Republic under the grant \textsc{P}201/12/\textsc{G}028.}

\author[Simone Virili]{Simone Virili}
\address{
\textsf{Simone Virili}: \newline
Facultad de Matem\'aticas, Universidad de Murcia \newline
Campus Espinardo, \oldstylenums{30100}, Murcia.  \newline
\href{mailto:virili.simone@gmail.com}
     {\sf virili.simone@gmail.com}}
\thanks{The second-named author is supported by the Ministerio de Econom\'ia y Competitividad of Spain via a grant `Juan de la Cierva-formaci\'on'. He is also supported by a research project from the Ministerio de Econom\'ia y Competitividad of Spain (MTM2014-53644-P) and from the Fundaci\'on `S\'eneca' of Murcia (19880/GERM/15), both with a part of FEDER funds.}
\subjclass[2010]{18E30, 18E35, 18C15 (primary), and 18C20, 13D09, 18G99(secondary)} 
\keywords{Factorization system, normal torsion theory, triangulated category, stable derivator, factorization algebra, squaring monad}

\begin{document}

\maketitle
\begin{abstract}
We define \emph{triangulated factorization systems} on triangulated categories, and prove that a suitable subclass thereof (the \emph{normal triangulated torsion theories}) corresponds bijectively to $t$-structures on the same category.
This result is then placed in the framework of derivators regarding a triangulated category as the base of a stable derivator. More generally, we define \emph{derivator factorization systems} in the 2-category $\PDer$, describing them as algebras for a suitable strict 2-monad (this result is of independent interest), and prove that a similar characterization still holds true: for a stable derivator $\D$, a suitable class of derivator factorization systems (the \emph{normal derivator torsion theories}) correspond bijectively with $t$-structures on the base $\D(\uno)$ of the derivator.
These two result can be regarded as the triangulated- and derivator- analogues, respectively, of the theorem that says that `$t$-structures are normal torsion theories' in the setting of stable $\infty$-categories, showing how the result remains true whatever the chosen model for stable homotopy theory is.
\end{abstract}

\tableofcontents
\section*{Introduction}
Factorization systems surely form a conspicuous part of modern category theory; this is especially because they provide the category where they live in with a rather rich structure, and they are commonly found (although very few of them can be easily built): for example, a trace of what we would today call a factorization system on the category of groups appears in the pioneering \cite{maclane1948groups}, published in 1948; more interestingly, as acknowledged by \cite{whitehead61elements}, any ``synthetic'' approach to homotopy theory inevitably relies on the notion of a --weak-- factorization system.

Soon after having reached a consensus on the definition for these gadgets \cite{FK}, category theorists wanted to make explicit the evident tight relation between (weak) factorization systems and (weakly) reflective subcategories on a same ambient category $\C$: this culminated with the proof, given in \cite{CHK}, that under mild assumptions the reflective subcategories of $\C$ are in bijection with the so-called \emph{reflective pre-factorization systems} on $\C$. 

Let us briefly recall this notion: a morphism $f$ in $\C$ is \emph{left orthogonal} to another morphism $g$ (or $g$ is \emph{right orthogonal} to $f$), in symbols $f\perp g$, if for any commutative square of solid arrows
\[
\xymatrix{
\cdot\ar[r]\ar[d]_f&\cdot\ar[d]^g\\
\cdot\ar[r]\ar@{.>}[ur]|{\ d\ }&\cdot
}
\]
there is a unique morphism $d$ that makes the two above triangles commute (this defines \emph{strong} orthogonality; in case \emph{at least one} such $d$ exist, we speak of weak orthogonality). Then,
\begin{itemize}
\item for a class $\mathcal X\subseteq \C^\due$ (where $\C^\due$ is the arrow category) we let ${}^{\perp}\mathcal X$ (resp., $\mathcal X^{\perp}$) be the class of morphisms which are left (resp., right) orthogonal to each element in $\mathcal X$;
\item a \emph{pre-factorization system} (\textsc{pfs} for short) on $\C$ is a pair $(\E,\M)$ of sub-classes of $\C^\due$ such that $\E={}^{\perp}\M$ and $\E^{\perp}=\M$;  
\item a pre-factorization system $\F=(\E,\M)$ on $\C$ such that every map $f\in \C^\due$ can be factored as a composition $f=m_f\circ e_f$, for $m_f\in \M$ and $e_f\in \E$ is called a \emph{factorization system} (\textsc{fs} for short; we informally call a morphism that can be factored by a \textsc{pfs} an $\F$-\emph{crumbled} arrow: then, a factorization system is such that every arrow is $\F$-crumbled);
\item a class $\mathcal X$ of morphisms of $\C$ is said to have the \emph{3-for-2 property} if, given two composable morphisms $\cdot\xto{f}\cdot\xto{g}\cdot$ in $\mathcal X$, if two elements of the set $\{f,g,g\circ f\}$ belong to $\mathcal X$, so does the third.
\end{itemize}
A \textsc{pfs} $\F=(\E,\M)$ is said to be \emph{reflective} if $\M$ has the 3-for-2 property and if any map of the form $\var{X}{0}$ is $\F$-crumbled. For such a \textsc{pfs}, the associated reflective subcategory of $\C$ is
\[
\M/0 \coloneqq \left\{X\in \C \mid \var{X}{0}\in \M\right\}\subseteq \C
\]
(uniqueness of lifts ensures that there is a functorial choice of an object in $\M/0$ for each $X\in \C$, precisely the object such that $X \xto{e_X} RX \xto{m_X} 0$). It is a remarkable result that \emph{all} the reflective subcategories of $\C$ arise in fact in this way: given such a subcategory $\cate{S}$, there is a reflective \textsc{pfs} \emph{generated} by all morphisms of $\cate{S}$.

The authors of \cite{CHK} then specialize this result attempting to describe the tight relation between factorization systems and \emph{torsion theories}, under similarly mild assumptions on $\C$. This approach has been extended sensibly in \cite{rosicky2007factorization}.

A factorization system $\F=(\E,\M)$ on $\C$ is said to be a \emph{torsion theory} (\textsc{tth} for short) if \emph{both} $\E$ and $\M$ have the 3-for-2 property. This gives (thanks to the above result and its dual) a \emph{pair} of subcategories $\M/0$ and $0/\E$ whose inclusions in $\C$ admit respectively a left and a right adjoint: these two subcategories form the classes of so\hyp{}called \emph{torsion} and \emph{torsion\hyp{}free} objects respectively, and relate to the classical notion of a \emph{torsion theory} given in \cite{dickson1966torsion}.

Suppose indeed that $\C$ is an abelian category. A \textsc{tth} $\F=(\E,\M)$ on $\C$ is said to be \emph{normal} if taking the $\F$-factorization
\[
X\xto{e}RX\xto{m}0
\]
of the final map $X\to 0$ for a given object $X\in \C$, and then taking the pullback
\begin{equation}\label{pb_diagram}
\begin{matrix}\xymatrix{
T\ar[r]\ar[d]\ar@{}[dr]|(.25)\lrcorner&X\ar[d]^{e}\\
0\ar[r]& RX
}\end{matrix}
\end{equation}
we have $\var{T}{0}\in \E$.

Applying the definitions, one can show that the pair $(0/\E,\M/0)$ is a \emph{classical} torsion theory (\ie a torsion theory as defined in \cite{dickson1966torsion}). In fact, it is also true that \emph{every} torsion theory arises this way (see \cite{rosicky2007factorization}); this gives a bijection between classical torsion theories and normal \textsc{tth}s. 

Switching to the triangulated context, the r\^ole played by classical \textsc{tth}s in abelian categories is now played by $t$-structures (\cite{BBD,keller2007derived}). The analogy between these two concepts was made completely formal by Beligiannis and Reiten \cite{beligiannis-reiten} where they introduced \emph{torsion pairs} in pre-triangulated categories. In fact, if the pre-triangulated structure is inherited from the abelian-ness of the ambient category, then torsion pairs correspond bijectively to classical \textsc{tth}s, while if the pre-triangulated structure is triangulated, then torsion pairs correspond bijectively to $t$-structures.

The strong analogies between classical \textsc{tth}s and $t$-structures suggests that there should be a way to describe them in terms of some kind of factorization systems, just like for \textsc{tth}s in abelian categories. In fact, pursuing a similar characterization in the non-abelian setting is acknowledged in \cite{rosicky2007factorization} as one of the most natural generalization of this technology. Nevertheless, the authors are not able to show a correspondence between $t$-structures on triangulated categories and factorization systems.

Somehow, this result has been prevented by a certain number of awkward properties of triangulated categories (see the introduction of \cite{maltsiniotis:k-theory} for a good account on this). In this respect, it is remarkable that such a theorem can be stated and proved quite naturally by getting rid of all these unwieldy features, ascending to the realm of stable $(\infty,1)$\hyp{}categories: the proof that $t$-structures on (the homotopy category of) a stable quasicategory correspond bijectively to normal torsion theories, regarded as particular $\infty$\hyp{}categorical factorization systems, has been the central result of the first author's PhD thesis \cite{tstructures}.\footnote{The fact that few triangulated categories generate an interesting poset of factorization systems is probably due to the fact that a nice factorization system on a category $\cate{A}$ interacts with co/limits on $\cate{A}$, and it is somehow generated by them: few triangulated categories have interesting co/limits, hence the fact that (for example) every \emph{proper} factorization system, where the left class is contained in the class of epimorphisms, although really natural in a generic category must be trivial in a triangulated one.}

Our first point in this paper is that the reason for the absence of this theorem from the setting of triangulated categories $\cD$ is that there is no notion of triangulated orthogonality $\horth$ for a pair of morphisms in $\cD$, with formal properties comparable to those of the orthogonality relation $\perp$ but \emph{mindful of the triangulated structure}.

The present work aims to fill this gap and solve the problem of finding a class of suitably defined \emph{triangulated factorization systems} on $\cD$ in bijection with the class of $t$-structures on $\cD$.

We start §\refbf{section_homo_FS} describing the homotopy orthogonality relation $f\horth g$ for two morphisms in a triangulated category $\cD$ (see \adef\refbf{wobbly}). After proving some natural properties, we mimic the classical theory showing that this definition is sound, in that it recovers basically all the formal properties enjoyed by the $\perp$-orthogonality relation (see \refbf{horth_coprod}--\refbf{closure_homo_ho2}). We introduce triangulated \textsc{pfs}s via triangulated orthogonality, triangulated \textsc{fs}s, triangulated \textsc{tth}s and, finally, normal triangulated \textsc{tth}s as the corresponding of each of the classical definitions.

We believe that this is the correct path to follow, as \adef\refbf{wobbly} is exactly an orthogonality condition that keeps track of the triangulated structure of $\cD$: as an example of this flexibility, normality for a triangulated \textsc{tth} can be introduced exactly as normality for a \textsc{tth} but taking a \emph{homotopy} cartesian square (see \refbf{recall_hocart} for the definition) in \eqref{pb_diagram} instead of a pullback square. So apparently the definition really captures the best of both worlds.

With the theory of triangulated \textsc{fs}s at hand, in \refbf{triang-rosetta} we prove the following
\begin{quote}
\textbf{Theorem I:} For a triangulated category $\cD$, the following map is bijective:
\[
\xymatrix@R=0pt{
\left\{
{\begin{smallmatrix}
\text{normal triangulated}\\
\text{\textsc{tth}s on }\cD
\end{smallmatrix}}
\right\}
\ar[rr]&&
\left\{
{\begin{smallmatrix}
\text{$t$-structures}\\
\text{ on }\cD
\end{smallmatrix}}
\right\}\\
(\E,\M) \ar@{|->}[rr]&& \Big(0/\E, \Sigma(\M/0)\Big) 
}
\]
\end{quote}
As mentioned above, \cite{Fiorenza2014} proved a $\infty$\hyp{}categorical version of \athm\textbf{I} in the setting of stable quasicategories. In fact, quasicategories support a fairly natural theory of \textsc{fs}s, as rich as the classical one; we refer to \cite{joyal2008notes} and \cite{HTT} (we briefly recall the relevant definitions in our §\refbf{infty_cat_fs} though). 

Once quasicategorical \textsc{fs}s are defined, one can mimic the definition of normal \textsc{tth} in this setting. The main results contained in \cite{tstructures} tells us that, for a stable quasicategory $\C$, the normal \textsc{tth}s on $\C$ are in bijection with $t$-structures on the triangulated category $\ho(\C)$.
An exercise in translation between models shows how the same result remains true
\begin{itemize}
\item in the setting of stable model categories, where one can speak about \emph{homotopy factorization systems} following \cite{bousfield1977constructions,Joy}; this leads to the definition of  \emph{homotopy $t$\hyp{}structures} on stable model categories $\cate{M}$ as suitable analogues of normal torsion theories in the set $\textsc{hfs}(\cate{M})$ of homotopy factorization systems on a model category $\cate{M}$;
\item in the setting of \textsc{dg}\hyp{}categories, where we speak about factorization systems (enriched in the sense of \cite{Day1974,enrichFS}); this leads to the definition of \emph{\textsc{dg}\hyp{}$t$\hyp{}structures} as enriched analogues of normal torsion theories in the set of enriched factorization systems on a \textsc{dg}\hyp{}category $\mathcal D$.
\end{itemize}
In both these settings, it is possible to recover a theorem that characterizes what, from time to time, you would like to call $t$-structures as a class in bijection with normal torsion theories defined in that specific model.

The second major result of the present paper is having established a similar result for again a different model of a stable homotopy theory, namely \emph{stable derivators}: this has to be regarded as the nontrivial step towards a model\hyp{}independence proof saying that $t$-structures \emph{are} indeed normal torsion theories whatever our preferred model for stable homotopy theory is.

The fact that the present claims are the less easy part of this plan is especially true because  it was the very definition of a factorization system on a derivator that had to be designed to perform this task, as this notion was absent from the general theory of the 2\hyp{}category $\PDer$. Building a flexible and expressive calculus of factorization systems on a (pre)derivator is then an important conceptual step \emph{per se}, in view of a deeper understanding of the 2\hyp{}categorical features of $\PDer$. A thorough, systematic approach to the subject of factorization systems in $\PDer$ will probably be the subject of subsequent investigations.

The theory of derivators was introduced by A\@. Grothendieck in an extremely long and famous manuscript \cite{tendieckderiv}, as an attempt to correct the above\hyp{}mentioned unwieldy features of triangulated categories: currently, a few people hope that they can provide an algebraic, purely 2-categorical model for the theory of (what we call today) $(\infty,1)$\hyp{}categories.

In modern terms, a pre-derivator $\D\colon \Dia^\opp\to\Cat$ is nothing more than a (strict) 2-functor, where $\Dia$ is a suitable sub\hyp{}2\hyp{}category of the 2-category $\cat$ of small categories, while $\Cat$ is the ``2-category'' of categories (see the introduction to \cite{Moritz} for all that regards set-theoretical issues in the basic theory of derivators).

We devote §\refbf{sec:squaring} to introduce and study the notion of \emph{derivator factorization system} (\textsc{dfs} for short) on a pre-derivator $\D$. Mimicking the classical theory, such a thing will be a pair of sub-functors $\mathbb E$ and $\mathbb M\colon \Dia^\opp\to\Cat$ of $\D^\due$ that are mutually ``orthogonal'' and that ``crumble all the morphisms in $\D$'' in a suitable sense (see \adef\refbf{def_c_ort}, \refbf{def_phfs} and \refbf{def_hfs}).

The precise definition of a \textsc{dfs} is fairly technical; let us just remark here that:
\begin{itemize}
\item if the pre-derivator $\D$ is \emph{representable}, \ie if there is a (large) category $\mathcal{A}$ such that $\D(I)=\mathcal{A}^I$, then a pair of sub pre-derivators $\F=(\mathbb E,\mathbb M)$ is a \textsc{dfs} if and only if $\F_I=(\mathbb E(I),\mathbb M(I))$ is a classical \textsc{fs} in the category $\D(I)$; this shows how the definition really generalizes the classical setting;
\item if $\D$ is a \emph{stable derivator} (which ensures that each $\D(I)$ is, canonically, a triangulated category), then a pair of sub pre-derivators $\F=(\mathbb E,\mathbb M)$ is a \textsc{dfs} if and only if $\F_I=(\mathbb E(I),\mathbb M(I))$ is a triangulated \textsc{fs} in $\D(I)$. This shows how the definition of a triangulated factorization system is nothing more than the ``shadow'' left by a derivator factorization system on the base $\D(\uno)$ of $\D$.
\end{itemize}
Of course, it would be possible to make a general statement out of this remark: a triangulated factorization system as defined in \refbf{the_def_of_hfs} is the shadow left by the $(\infty,1)$\hyp{}categorical definition by passing to the triangulated homotopy category of whatever model for our stable homotopy theory: it is worth to remark that the factorization systems arising in this way are seldom orthogonal (\ie there is no unique solution to lifting problems), even though they come from orthogonal ones (where uniqueness is specified up to a suitable notion of homotopy specific to the model in study).

In §\refbf{higher_rosetta_subs} we introduce the notion of \emph{normal} derivator \textsc{tth}. For a stable derivator $\D$, this corresponds to a \textsc{dfs} $\F=(\mathbb E,\mathbb M)$ for which each $\F_I=(\mathbb E(I),\mathbb M(I))$ is a normal triangulated \textsc{tth} in $\D(I)$. We then prove the following theorem, that summarizes all we said:
\begin{quote}
\textbf{Theorem II:} For a stable derivator $\D\colon \fincat^\opp\to \Cat$, the following map is bijective:
\[
\xymatrix@R=0pt{
\left\{
{\begin{smallmatrix}
\text{normal derivator}\\
\text{\textsc{tth}s on }\D
\end{smallmatrix}}
\right\}
\ar[rr] &&
\left\{
{\begin{smallmatrix}
\text{$t$-structures}\\
\text{on }\D(\uno)
\end{smallmatrix}}
\right\}\\
(\mathbb E,\mathbb M) \ar@{|->}[rr]&& \Big(0/\mathbb E(\uno), \Sigma(\mathbb M(\uno)/0)\Big) 
}
\]
\end{quote}
Notably, as a consequence of the above theorem we can recover the main result of \cite{Fiorenza2014} as a corollary.

In the last two sections of the paper we study some formal properties of \textsc{dfs}s. For this, we extend to our setting the two main results of \cite{Korostenski199357}. There, the authors start from the observation that any factorization systems is given by a so-called \emph{factorization pre-algebra}, that is, a functor $F_\F\colon \C^\due \to \C$ (defined as a section of the composition map $c\colon \C^\tre \to \C^\due\colon (g,f)\mapsto g\circ f$) such that $F_\F(\id_X)=X$ for any $X\in \C$. To any such functor, one associates two functors
\[
e_{-},\, m_{-}\colon \C^\due\to \C^\due
\]
that give us a functorial factorization of any given morphism $f\colon X\to Y$ in $\C$,  
\[
X\xto{e_f}F_\F(f)\xto{m_f}Y
\] 
with $e_f\in \E$ and $m_f\in \M$. On the other hand, given a factorization pre-algebra $F\colon \C^\due \to \C$, one defines
\[
\E_F \coloneqq \{h\in \C^\due \mid m_h\text{ is an iso}\}\qquad\text{and}\qquad
\M_F \coloneqq \{k\in \C^\due \mid e_k\text{ is an iso}\},
\]
and says that $F$ is an \emph{Eilenberg\hyp{}Moore factorization} provided $e_f\in \E_F$ and $m_f\in \M_F$ for any $f\in \C^\due$. The major result of \cite[\athm\textbf{A}]{Korostenski199357} is that, for an Eilenberg\hyp{}Moore factorization $F$, the pair $(\E_F,\M_F)$ is a \textsc{fs}. This beautiful piece of formal category theory  ignited a certain amount of research: related topics led to what we call today \emph{algebraic factorization systems} (see \cite{Gar,grandis2006natural}, also in connection with the definition of model category in \cite{riehl2011algebraic}).

For a pre-derivator $\D$, a \emph{factorization pre-algebra}  becomes a morphism  $F\colon \D^\due\to \D$ such that $F\circ \pt^\circledast=\id_\D$ (see \adef\refbf{def_factorization_alg}). To such an $F$ one associates two endo-1-cells 
\[
F_l,\, F_r\colon \D^\due\to \D^\due
\]
playing the same r\^ole of $e_{-}$ and $m_{-}$ above. Then one defines two sub pre-derivators $\mathbb E$ and $\mathbb M$ of $\D^\due$, where
\begin{gather*}
\mathbb E_F(I) \coloneqq \{X\in \D^\due(I) \mid F_lX\text{ is an iso}\}\\
\mathbb M_F(I) \coloneqq \{Y\in \D^\due(I) \mid F_rY\text{ is an iso}\}
\end{gather*}
for any $I\in \Dia$, and says that $F$ is an \emph{Eilenberg\hyp{}Moore factorization} provided $F_rX\in \mathbb E_F(I)$ and $F_lX\in \mathbb M_F(I)$ for any $X\in \D^\due(I)$. We are able to rephrase \cite[\athm\textbf{A}]{Korostenski199357} as follows:
\begin{quote}
\textbf{Theorem III:}
Let $\D$ be a pre-derivator and  $F\colon \D^\due\to \D$ an Eilenberg\hyp{}Moore factorization  (see \adef\refbf{def_EM_algebra}). Under very mild assumptions on $\D$ (see Setting \refbf{setting_sec_3})  the pair $(\mathbb E_F,\mathbb M_F)$ is a \textsc{dfs}. If $\D$ is represented or if it is a stable derivator, then any \textsc{dfs} on $\D$ arises this way from an Eilenberg\hyp{}Moore factorization.
\end{quote}
The inherently 2-categorical content of \cite{Korostenski199357} becomes clear as the authors move to the second main statement: \cite[\athm\textbf{B}]{Korostenski199357}
\begin{quote}
Orthogonal factorization systems can described as Eilenberg\hyp{}Moore algebras for the \emph{squaring monad} on $\Cat$, that sends a category $\A$ into its functor category $\Cat(\due,\A)$.
\end{quote}
The authors explicitly suggest how the reason why this second statement holds relies on purely formal computations that can in principle be carried on in a sufficiently well\hyp{}behaved 2-category other than $\Cat$.

Our aim here is to catch this hint and follow these steps quite faithfully, exploiting the intimate connection between $\Cat$ and $\PDer$; this allows us to reformulate quite easily those parts of \cite{Korostenski199357} that depend on the features of $\Cat$ only on the surface.

Spelled out more explicitly, \cite[\athm\textbf{B}]{Korostenski199357} regards orthogonal factorization systems as \emph{normal pseudo-algebras} for the squaring monad: this is the monad $T=((-)^\due,\mu,\eta)$, consisting of the strict $2$-functor 
\[
(\firstblank)^\due\colon \Cat\to \Cat
\]
such that $\C\mapsto \C^\due$, endowed with the natural transformations $\mu$ and $\eta$ (multiplication and unit, respectively), where $\mu_\C\colon \C^{\due\times \due}\to \C^\due$ is induced by the precomposition with the diagonal map $\Delta$ that we define in \refbf{comonoid_due_subs}: an object of $\C^{\due\times\due}$, \ie a commutative square $\left[\begin{smallmatrix} X_{00} &\to& X_{10} \\ \downarrow && \downarrow \\ X_{01} &\to& X_{11}\end{smallmatrix}\right]$, goes to its diagonal $\var{X_{00}}{X_{11}}$, while $\eta_\C\colon \C\to \C^\due$ maps an object to its identity morphism. An important property for us (see \cite{Korostenski199357,RW}) is that a factorization pre-algebra $F\colon \C^\due\to \C$ is forced to be an algebra by whichever isomorphism $FF^\due\cong F\mu_\C$, that is then forced to be an \emph{extended associator} interacting with the monad multiplication in the well-known way.

This can be regarded as a coherence result which is utterly specific to this particular monad, showing how the entire $(\firstblank)^\due$-algebra structure for $F$ is a little bit redundant: in this specific case, the unit alone is enough to uniquely determine an extended associator $\alpha_m$ (see \adef\refbf{two-monad}).

We reformulate these results in the setting of pre-derivators as follows, and prove it as the last statement in §\refbf{higher_coherence_sub}:
\begin{quote}
\textbf{Theorem \textbf{IV}:} Let $\D$ be either a represented pre-derivator or a stable derivator. The following are equivalent for a normal factorization pre-algebra $F\colon \D^\due\to \D$:
\begin{enumerate}
\item $F$ can be endowed  with the structure of an algebra over the squaring monad;
\item there exists an isomorphism $\alpha\colon FF^\due \xto{\sim} F\Delta^\circledast$;
\item $F$ is a \textsc{em} factorization (so that $(\mathbb E_F,\mathbb M_F)$ is a \textsc{dfs}).
\end{enumerate}
\end{quote}
\medskip
\paragraph{\bf Acknowledgements.} The first author thanks prof\@. J\@. Rosick\'y, because it was possible to finish the hardest part of the present paper mainly thanks to the pleasant environment of Masaryk University. 
Both authors would like to express their gratitude to F\@. Mattiello, because he surely is a moral third author, and A\@. Gagna, for his careful reading of §\textbf{3} and for having spotted an error in our initial argument linking derivator- and quasicategorical factorization systems.

\medskip
\paragraph*{\bf Notation and terminology.}
Among different foundational convention that one may adopt throughout the paper we assume that every set lies in a suitable Grothendieck universe. Throughout §\textbf{1-4} this choice can be safely replaced by the more popular foundation using sets and classes. In §\textbf{5-6} the need to consider the ``category'' of 2-functors $\PDer \to \PDer$ forces us to fix such a (hierarchy of) universe(s).

More in detail we implicitly fix an universe $\mho$, whose elements are termed \emph{sets}; \emph{small categories} have a \emph{set} of morphisms; \emph{locally small} categories are always considered to be small with respect to \emph{some} universe: in particular we choose to adopt, whenever necessary, the so\hyp{}called \emph{two\hyp{}universe convention}, where we postulate the existence of a universe $\mho^+\ni \mho$ in which all the classes of objects of non\hyp{}$\mho$\hyp{}small, locally small categories live. 

We denote $\cat = \mho\text{-}\Cat$ and $\CAT = \mho^+\text{-}\Cat$ for short, and we extend this notation somewhere without further mention: this means that, for example, $\cate{sSet} = [\bDelta^\opp,\mho]$ and $\cate{sSET} = [\bDelta^\opp, \mho^+]$.

Possibly large categories and higher categories will be usually denoted as boldface letters $\A,\cate{B},\dots$; generic classes of morphisms in a category are denoted as calligraphic letters $\E, \M, \mathcal{X},\mathcal{Y},\dots$; when they are considered as objects of the category $\cat$, small categories are usually denoted as capital Latin letters like $I,J,K,\dots$, but so is an object of a possibly large category $\C$; it is always possible to solve this slight abuse of notation.

Functors between \emph{small} categories are denoted as lowercase Latin letters like $u,v,w,\dots$ and suchlike (there are of course numerous deviations to this rule); the category of functors $\A\to \cate{B}$ between two categories is invariably denoted as $\Cat(\A,\cate{B})$, $\cate{B}^{\A}$, $[\A,\cate{B}]$ and suchlike; the canonical $\hom$\hyp{}bifunctor of a category $\A$ sending $(c,c')$ to the set of all arrows $\hom(c,c')\subseteq\hom(\A)$ is almost always denoted as $\A(\firstblank,\secondblank)\colon \A^\opp\times\A\to\cate{Sets}$, and the symbols $\firstblank$, $\secondblank$ are used as placeholders for the ``generic argument'' of a functor or bifunctor; morphisms in the category $\Cat(\A,\cate{B})$ (\ie natural transformations between functors) are often written in Greek, or Latin lowercase alphabet, and collected in the set $\Nat(F,G) = \cate{B}^{\A}(F,G)$. 

The simplex category $\bDelta$ is the \emph{topologist's delta} (opposed to the \emph{algebraist's delta} $\bDelta_+$ which has an additional initial object $[-1]$), having objects \emph{nonempty} finite ordinals $[n]=\{0<1\dots<n\}$; we denote $\Delta^n$ the representable presheaf on $[n]\in\bDelta$, \ie the image of $[n]$ under the Yoneda embedding of $\bDelta$ in the category $\cate{sSet}$ of simplicial sets; the notation $\Delta^J$ for $J\subseteq \{0,\dots,n\}$ denotes the sub\hyp{}simplex generated by the vertices in $J$ (so, for example, $\Delta^{\{0,2\}} \subseteq \Delta^2$ is the copy of $\Delta^1$ that sends $0$ to $0$ and $1$ to $2$). The notation $\Lambda^k[n]\subset \Delta^n$ denotes the \emph{$k^\text{th}$ horn inclusion}, \ie the sub\hyp{}simplicial set of $\Delta^n$ resulting from the union of all the images of the face maps $d_i$, for $i\neq k$ in $\{0,\dots,n\}$. More often the objects of $\Delta$ are considered as categories via the obvious embedding $\bDelta\subset\cat$: in this case, the object $[n-1]\in\bDelta$ is denoted $\cate{n}\in\cat$ (so for example all along §\refbf{sec:thmA} we write $\due = \{0 < 1\}$, and similarly $\tre = \{0<1<2\}$).

Apart from this, we indicate the Yoneda embedding of a category $\A$ into its presheaf category with $\yon_\A$ --or simply $\yon$--, i.e\@. with the hiragana symbol for ``yo''; this choice comes from \cite{Libland2015}. Whenever there is an adjunction $F\dashv G$ between functors, the arrow $Fa\to b$ in the codomain of $F$ and the corresponding arrow $a\to Gb$ in its domain are called \emph{mates} or \emph{adjuncts}; so, the notation ``the mate/adjunct of $f\colon Fa\to b$'' means ``the unique arrow $g\colon a\to Gb$ determined by $f$''. When there is an adjunction between two functors $F,G$ we adopt $F\adjunct{\eta}{\epsilon}G$ as a compact notation to denote at the same time that $F$ is left adjoint to $G$, with unit $\eta \colon 1 \to GF$ and counit $\epsilon\colon FG\to 1$. The \emph{whiskering} between a 1-cell $F$ and a 2-cell $\alpha$ is denoted $F * \alpha$ or $\alpha * F$.
\section{Triangulated factorization systems}\label{section_homo_FS}
Throughout this section we let $\cD$ be a (fixed but arbitrary) triangulated category, with shift functor $\Sigma\colon \cD \xto{\simeq} \cD$. For a general background and notation on triangulated categories we refer to \cite{Neeman} and \cite[Appendix \textbf{A}]{hps:axiomatic}. 

Even though this assumption is not requested, as we will state and prove our theorems for fully general triangulated categories (assuming only, from time to time, the existence of countable (co)products), the reader should keep in mind that in Section \refbf{sec:squaring} will be clear how our motivating example for $\cD$ is the underlying category $\D(\uno)$ of a stable derivator.
\subsection{Homotopy orthogonality of morphisms}\label{horth_subs}
Our first task is to build a notion of orthogonality of morphisms mindful of the triangulated structure on $\cD$.
\begin{definition}[homotopy orthogonality]\label{wobbly}
Let $E_0 \xto{e} E_1$ and $M_0 \xto{m} M_1$ be two maps in $\cD$, and complete them to triangles
\begin{equation}\label{triangles}
E_0 \xto{e}  E_1 \xto{\alpha_e} C_e \xto{\beta_e} \Sigma E_0
\qquad \text{and}\qquad 
M_0 \xto{m}  M_1 \xto{\alpha_m} C_m \xto{\beta_m} \Sigma M_0.
\end{equation}
We say that $e$ is \emph{left homotopy orthogonal} to $m$ (while $m$ is \emph{right homotopy orthogonal} to $e$), in symbols $e \horth m$, if the following conditions are satisfied:
\begin{enumerate}[label=\textsc{ho}\arabic*.]
\item \label{fst}the following map is trivial:
$$\xymatrix@R=0pt{\cD(C_e,\Sigma^{-1}C_m)\ar[r]& \cD(E_1,M_0)\\
(C_e\overset{\varphi}{\longrightarrow}\Sigma^{-1}C_m)\ar@{|->}[r]&(E_1\overset{\alpha_e}{\longrightarrow}C_e\overset{\varphi}{\longrightarrow}\Sigma^{-1}C_m\overset{\Sigma^{-1}\beta_m}{\longrightarrow}M_0)\,;}$$
\item \label{snd}the following map is injective:
$$\xymatrix@R=0pt{\cD(C_e,C_m)\ar[r]& \cD(E_1,\Sigma M_0)\\
(C_e\overset{\varphi}{\longrightarrow}C_m)\ar@{|->}[r]&(E_1\overset{\alpha_e}{\longrightarrow}C_e\overset{\varphi}{\longrightarrow}C_m\overset{\beta_m}{\longrightarrow}\Sigma M_0)\,.}$$
\end{enumerate}
\end{definition}

The concept of homotopy orthogonality seems quite artificial, but this notion arises naturally in the setting of stable derivators (see Section \refbf{sec:squaring}). Notice also that one can prove by standard arguments that homotopy orthogonality does not depend on the choice of triangles in \eqref{triangles}.
\begin{remark}
Condition \refbf{wobbly}\textbf{.}\textsc{ho2} can be substituted by the following one:
\begin{enumerate}
\item[\textsc{ho}2'.] The unique morphism $\varphi$ completing a morphism $(a,b)\colon e\to m$ in $\cD^\due$ to a morphism of triangles, as in the following diagram, is $\varphi=0$:
$$\xymatrix{
E_0\ar[r]^{e}\ar[d]_a&E_1\ar[r]^{\alpha_e}\ar[d]^b&C_e\ar[r]^{\beta_e}\ar@{.>}[d]|\varphi&\Sigma E_0\ar[d]\\
M_0\ar[r]^m&M_1\ar[r]^{\alpha_m}&C_m\ar[r]^{\beta_m}&\Sigma M_0.
}$$ 
\end{enumerate}
To see this equivalence, suppose that condition \refbf{wobbly}\textbf{.}\textsc{ho2} is satisfied. Then, the map $\cD(C_e,C_m)\to \cD(E_1,\Sigma M_0)$ sends $\varphi$ to $\beta_m\varphi\alpha_e=\beta_m\alpha_m b=0$; so by the injectivity of this map, we deduce that $\varphi=0$. On the other hand, suppose \textsc{ho}2' is satisfied and consider a morphism $\psi\in \cD(C_e,C_m)$ such that $\beta_m\psi\alpha_e=0$; we have to show that $\psi=0$. Indeed, since $\beta_m\psi\alpha_e=0$ we can construct a morphism of triangles as follows:
$$\xymatrix{
0\ar[r]\ar[d]&E_1\ar@{=}[r]\ar@{.>}[d]|{\exists b}&E_1\ar[r]^{}\ar[d]|{\psi\alpha_e}&0\ar[d]\\
M_0\ar[r]^m&M_1\ar[r]^{\alpha_m}&C_m\ar[r]^{\beta_m}&\Sigma M_0.
}$$ 
Now one can complete the central square in the following diagram to a morphism of triangles:
$$\xymatrix{
E_0\ar[r]^{e}\ar@{.>}[d]_{\exists a}&E_1\ar[r]^{\alpha_e}\ar[d]^b&C_e\ar[r]^{\beta_e}\ar[d]^\psi&\Sigma E_0\ar@{.>}[d]\\
M_0\ar[r]^m&M_1\ar[r]^{\alpha_m}&C_m\ar[r]^{\beta_m}&\Sigma M_0.
}$$ 
Then, by \textsc{ho}2', $\psi=0$ as desired.
\end{remark}
In what follows we verify some properties that one should expect from any well-behaved notion of orthogonality. 
Let us start with the following property, whose proof is an easy exercise:
\begin{lemma}\label{is_an_iso_if_horth_to_all}
The following are equivalent for $f\in\cD^\due$
\begin{enumerate}[label=(\roman*)]
	\item $f$ is an isomorphism;
	\item $f\horth \cD^\due$;
	\item $\cD^\due\horth f$;
	\item $f\horth f$.
\end{enumerate}
\end{lemma}
The above proposition adopted an harmless abuse of notation, that is, it denoted $\mathcal{H}\horth \mathcal{K}$ the fact that each $h\in\mathcal{H}$ is left $\horth$-orthogonal to every morphism of $\mathcal{K}$. To make this statement precise we introduce the following definitions.

\begin{notat}[$\horth$-orthogonal of a class]
We denote $\lhorth{(\firstblank)} \dashv (\firstblank)^{\horth}$ the (antitone) Galois connection induced by the relation $\horth$ on full subcategories of $\cD^\due$;
more explicitly, we denote
\begin{gather}
\mathcal{X}^{\horth} \coloneqq \{f\in \cD^\due \mid x \horth f, \ \forall x\in \mathcal{X}\}\notag\\
\lhorth{\mathcal{X}} \coloneqq \{f\in \cD^\due \mid f \horth x, \ \forall x\in \mathcal{X}\}.
\end{gather}
\end{notat}
\begin{remark}[$\horth$-locality]\label{object.ortho}
There is a related notion of orthogonality between an object $X$ and a morphism $f\in\cD^\due$, based on the fact that we can blur the distinction between objects and their initial or terminal arrows; given these data, we say that $X$ is \emph{right-orthogonal} to $f$ (or that $X$ is an \emph{$f$-local} object) if the hom functor $\cD(-,X)$ inverts $f$; in fact, the map $\cD(f,X)$ is injective if and only if the pair $(f,\var{X}{0})$ satisfies condition  \refbf{wobbly}\textbf{.}\textsc{ho1}, while it is surjective if and only if $(f,\var{X}{0})$ satisfies condition \refbf{wobbly}\textbf{.}\textsc{ho2}. (Obviously, there is a dual notion of left orthogonality between $f$ and $B\in\cD$, or a notion of a \emph{$f$-colocal} object $B$ which reduces to left orthogonality with respect to $0\to B$).
\end{remark}

By the above remark, it is natural to say that two objects $B$ and $X$ are homotopy orthogonal if $\var{0}{B}\horth\var{X}{0}$. In fact, it is not difficult to show that this happens if and only if $\cD(B,X)=0$, that is, $B\perp X$ in the usual sense.

The following lemma can be easily verified by hand:

\begin{lemma}\label{horth_coprod}
Let $\{f_i\}_{i\in I},\, g\in \cD^\due$. If $f_i\horth g$ for all $i\in I$,  then $\coprod_if_i\horth g$. On the other hand, if $g\horth f_i$ for all $i\in I$, then $g\horth \prod_i f_i$. 
\end{lemma}
\begin{lemma}
Let $f,\, g\in \cD^\due$ and let $f'$ be a retract of $f$, that is, there is a commutative diagram
\[
\xymatrix{F_0'\ar[r]|{i_0}\ar[d]_{f'} \ar@/^15pt/[rr]|{\id} & F_0\ar[r]|{p_0}\ar[d]^f & F_0'\ar[d]^{f'}\\ 
F_1'\ar[r]|{i_1}\ar@/_15pt/[rr]|{\id} &F_1 \ar[r]|{p_1} &F_1'}
\]
If $f\horth g$, then $f'\horth g$.
\end{lemma}
\begin{proof}
Let $(a,b)\colon f'\to g$ be a morphism in $\cD^\due$ and consider the following commutative diagram, whose columns are triangles:
$$
\xymatrix{
F'_0\ar[r]^{i_0}\ar[d]_{f'}&F_0\ar[r]^{p_0}\ar[d]_{f}&F'_0\ar[r]^a\ar[d]_{f'}&G_0\ar[d]^g\\
F'_1\ar[r]^{i_1}\ar[d]_{\alpha_{f'}}&F_1\ar[r]^{p_1}\ar[d]_{\alpha_{f}}&F'_1\ar[r]^b\ar[d]_{\alpha_{f'}}&G_1\ar[d]^{\alpha_g}\\
C_{f'}\ar[r]^i\ar[d]_{\beta_{f'}}&C_{f}\ar[r]^p\ar[d]_{\beta_{f}}&C_{f'}\ar[r]^{\varphi}\ar[d]_{\beta_{f'}}&C_{g}\ar[d]_{\beta_{g}}\\
\Sigma F'_0\ar[r]&\Sigma F_0\ar[r]&\Sigma F'_0\ar[r]^{\Sigma a}&\Sigma G_0
}
$$
and notice that the composition $p\circ i$ is  an isomorphism. To verify \refbf{wobbly}\textbf{.}\textsc{ho1} we should prove that $\varphi=0$, but in fact, $\varphi p=0$ for the same condition applied to the pair $(f,g)$, so that $\varphi\cong \varphi p\, i=0$. On the other hand, to verify \refbf{wobbly}\textbf{.}\textsc{ho2}, consider a morphism $\psi\colon C_{f'}\to \Sigma^{-1}C_g$, then $\Sigma^{-1}(\beta_g)\psi\alpha_{f'}\cong \Sigma^{-1}(\beta_g)\psi p\, i \alpha_{f'}=\Sigma^{-1}(\beta_g)\psi p\alpha_{f} i_{1}=0\circ i_1=0$, where $\Sigma^{-1}(\beta_g)\psi p\alpha_{f}=0$ by the same condition applied to the pair $(f,g)$. 
\end{proof}
\begin{remark}\label{recall_hocart}
To simplify the formulation of some of our forthcoming observations, let us recall that a \emph{homotopy cartesian square} in $\cD$ is a commutative diagram
\begin{equation}\label{cartesian_diagram}
\begin{matrix}\xymatrix{
X\ar@{}[dr]|\boxvoid\ar[r]^{\phi}\ar[d]_{\alpha}&Y\ar[d]^{\beta}\\
X'\ar[r]_{\phi'}&Y'
}\end{matrix}
\end{equation}
such that there exists a distinguished triangle $X\to X'\oplus Y\to Y'{\to}\Sigma X$, where the map $X\to X'\oplus Y$ is $\binom{\alpha}{-\phi}$, while the map $X'\oplus Y\to Y'$ is $(\phi', \beta)$. We call $\beta$ the {\em homotopy pushout} of $\alpha$, and $\alpha$ the {\em homotopy pullback} of $\beta$. We refer to \cite[\achap\textbf{1}]{Neeman} for more details on this construction.
\end{remark}
\begin{lemma}\label{closure_homo_ho2}
Let $(\psi\colon Y_0 \to  Y_1 )\in \cD^\due$ and consider a homotopy cartesian square:
$$\xymatrix{
 X_0 \ar@{}[dr]|\boxvoid\ar[r]^{s}\ar[d]_{\phi}& X_0' \ar[d]^{\phi'}\\
 X_1 \ar[r]_{t}& X_1 '
}$$
Then the following statements hold true:
\begin{enumerate}
\item[\rm (1)] if the pair $(\phi,\psi)$ satisfies \refbf{wobbly}\textbf{.}\textsc{ho2}, so does the pair $(\phi',\psi)$;
\item[\rm (2)] if the pair $(\phi,\psi)$ satisfies \refbf{wobbly}\textbf{.}\textsc{ho1} and $(\phi,\Sigma^{-1}\psi)$ satisfies \refbf{wobbly}\textbf{.}\textsc{ho2}, then $\cD(C_{\phi'},\Sigma^{-1}C_{\psi})=0$;
\item[\rm (3)] if $\phi\horth \psi$ and $\phi\horth\Sigma^{-1}\psi$, then $\phi'\horth \psi$.
\end{enumerate}
\end{lemma}
\begin{proof}
(1) Given a morphism $(a,b)\colon \phi'\to \psi$, we get a commutative diagram:
$$
\xymatrix{
 X_0 \ar@{}[rd]|\square\ar[r]^{s}\ar[d]_\phi& X_0' \ar[r]^a\ar[d]^{\phi'}& Y_0 \ar[d]^{\psi}\\
 X_1 \ar[r]^{t}\ar[d]_\alpha& X_1' \ar[r]^b\ar[d]^{\alpha'}& Y_1 \ar[d]^{\alpha_\psi}\\
C_\phi\ar[r]|\cong^{\varphi}\ar[d]_{\beta}&C_{\phi'}\ar[r]^{\psi}\ar[d]^{\beta'}&C_\psi\ar[d]^{\beta_\psi}\\
\Sigma  X_0 \ar[r]&\Sigma  X_0 '\ar[r]&\Sigma  Y_0 }
$$
we should prove that $\psi=0$. By \refbf{wobbly}\textbf{.}\textsc{ho2} applied to $(\phi, \psi)$ we get $\psi\varphi=0$, but since $\varphi$ is an isomorphism this allows us to conclude.

(2) Our two assumptions tell us that the  map $\cD(C_\phi,\Sigma^{-1}C_{\psi})\to \cD( X_1 , Y_0 )$ is both trivial and injective, so that $\cD(C_{\phi'},\Sigma^{-1}C_{\psi})\cong \cD(C_\phi,\Sigma^{-1}C_{\psi})=0$. 

(3) By part (1) and $\phi\horth \psi$, the pair $(\phi',\psi)$ satisfies \refbf{wobbly}\textbf{.}\textsc{ho2}. Furthermore, by part (2) and our assumptions, $\cD(C_{\phi'},\Sigma^{-1}C_{\psi})=0$, so the map $\cD(C_{\phi'},\Sigma^{-1}C_{\psi})\to\cD(\phi'_1, Y_0 )$ is clearly trivial.
\end{proof}

Let us recall from \cite{neeman:new-axioms} that a morphism of triangles
\begin{equation}\label{morphism_of_tria}\xymatrix{
A_0\ar[r]^{ \phi_0 }\ar[d]_a&B_0\ar[r]^{ \psi_0 }\ar[d]_b&C_0\ar[r]\ar[d]^c&\Sigma A_0\ar[d]^{\Sigma a}\\
A_1\ar[r]^{ \phi_1 }&B_1\ar[r]^{ \psi_1 }&C_1\ar[r]&\Sigma A_1}\end{equation}
is said to be {\em middling good} if it can be completed to a $3\times 3$ diagram whose rows and columns are triangles and where everything commutes but the lower right square, which anti-commutes:
\begin{equation}\label{3x3}
\xymatrix{
A_0\ar[r]^{ \phi_0 }\ar[d]_a&B_0\ar[r]^{ \psi_0 }\ar[d]_b&C_0\ar[r]\ar[d]^c&\Sigma A_0\ar[d]^{\Sigma a}\\
A_1\ar[d]_{\alpha_a}\ar[r]^{ \phi_1 }&B_1\ar[d]_{\alpha_b}\ar[r]^{ \psi_1 }&C_1\ar[r]\ar[d]^{\alpha_c}&\Sigma A_1\ar[d]\\
C_a\ar[r]^{\varphi_a}\ar[d]_{\beta_a}&C_b\ar[r]^{\varphi_b}\ar[d]_{\beta_b}&C_c\ar[r]\ar[d]^{\beta_c}&\Sigma C_a\ar[d]^{}\\
\Sigma A_0\ar[r]&\Sigma B_0\ar[r]_{}&\Sigma C_0\ar[r]&\Sigma^2 C_0
}
\end{equation}
Let us recall that, given a morphism $(a,b)\colon  \phi_0 \to  \phi_1 $ in $\cD^\due$, one can always choose a morphism $c\colon C_0\to C_1$ such that $(a,b,c)$ is a middling good morphism of triangles. 
\begin{lemma}\label{extension}
Let $(\chi\colon  Y_0 \to  Y_1 )\in\cD^\due$ and consider a middling good morphism of triangles as in \eqref{morphism_of_tria}. 
If $a,\, \Sigma a,\, c,\, \Sigma c, \Sigma^{-1}c\horth \chi$, then $b\horth \chi$.
\end{lemma}
\begin{proof}
By Lemma \refbf{closure_homo_ho2}, $a,\, \Sigma a\horth \chi$ implies $\cD(C_a,\Sigma^{-1}C_\chi)=0$, while  $c,\, \Sigma c\horth \chi$ implies $\cD(C_c,\Sigma^{-1}C_\chi)=0$. Hence,  $\cD(C_b,\Sigma^{-1}C_\chi)=0$. On the other hand, for a morphism $(d,e)\colon b\to \chi$, we get a commutative diagram whose columns are triangles:
$$\xymatrix{
A_0\ar[r]^{ \phi_0 }\ar[d]_a&B_0\ar[r]^d\ar[d]_b& Y_0 \ar[d]^\chi\\
A_1\ar[r]^{ \phi_1 }\ar[d]_{\alpha_a}&B_1\ar[r]^e\ar[d]_{\alpha_b}& Y_1 \ar[d]^{\alpha_\chi}\\
C_a\ar[r]^{\varphi_a}\ar[d]_{\beta_a}&C_b\ar[r]^{\varphi}\ar[d]_{\beta_b}&C_\chi\ar[d]^{\beta_\chi}\\
\Sigma A_0\ar[r]&\Sigma B_0\ar[r]&\Sigma  Y_0 
}$$ 
Since $a\horth \psi$, then $\varphi\varphi_a=0$, which implies that there exists $f\colon C_c\to C_\chi$ such that $f\circ \varphi_b=\varphi$. By $c, \Sigma^{-1}c\horth \chi$ we get $\cD(C_c,C_{\chi})=0$, so $f=0$, which implies $\varphi=0$.
\end{proof}

\begin{lemma}\label{horth_colimits}
Let $\psi\in\cD^\due$ and consider two countable chains of morphisms $A_\bullet = \{A_0\xto{j_0}A_1\xto{j_1}A_2\xto{j_2}\dots\}$ and $B_\bullet = \{B_0\xto{k_0}B_1\xto{k_1}B_2\xto{k_2}\dots\}$.  If there is a natural transformation $\alpha\colon A_\bullet \Rightarrow B_\bullet$ such that $\alpha_i,\, \Sigma \alpha_i,\, \Sigma^2 \alpha_i\horth \psi$ for all $i\in\N$, then any map $\varphi\colon  \text{hocolim } A_\bullet\to  \text{hocolim } B_\bullet$ completing the following diagram to a middling good map of triangles is such that $\varphi\horth \psi$
\[
\xymatrix{
	\coprod_{i\in\N} A_i \ar[d]\ar[r] & \coprod_{i\in\N} A_i \ar[d]\ar[r] & \text{hocolim } A_\bullet \ar[r]\ar@{.>}[d]^{\varphi} & +\\
	\coprod_{i\in\N} B_i \ar[r] & \coprod_{i\in\N} B_i \ar[r] & \text{hocolim } B_\bullet \ar[r] & +
}
\]
\end{lemma}
\begin{proof}
By Lemma \refbf{horth_coprod}, $\coprod_{i\in\N} \alpha_i,\, \Sigma\coprod_{i\in\N} \alpha_i,\, \Sigma^2\alpha_i\coprod_{i\in\N}\horth \psi$, so it is enough to apply Lemma \refbf{extension}.
\end{proof}
\subsection{Triangulated factorization systems}
Using the notion of homotopy orthogonality we can define triangulated factorization systems as follows:
\begin{definition}\label{the_def_of_hfs}
Let $\F=(\E,\M)$ be a pair of classes of morphisms in $\cD$. 
\begin{enumerate}
\item $\F$ is a \emph{triangulated pre-factorization system} (\phfs for short) if 
\begin{enumerate}
\item[\rm --] $\E^{ \horth }=\M$ and $\lhorth{\!\M}=\E$;
\item[\rm --] $\phi\in \E$ implies $\Sigma\phi\in \E$.
\end{enumerate} 
\item $\F$ is a \emph{triangulated factorization system} (\hfs for short) if it is a \phfs, and if any morphism in $\cD$ is \emph{$\F$-crumbled}, \ie it can be factored as a composition $\phi=m\circ e$ with $e\in\E, m\in\M$.
\end{enumerate}
\end{definition}
Notice that in the second condition defining a \phfs we could have equivalently asked that $\phi\in\M$ implies $\Sigma^{-1}\phi\in\M$.
\begin{remark}[left- and right\hyp{}generated $\horth$\hyp{}prefactorization]
It is evident 
that any class of morphism $\mathcal{X}\subseteq \cD^\due$ induces two {\phfs}s on $\cD$, obtained by sending $\mathcal{X}$ to $({}^{\horth} \mathcal{X}, ({}^{\horth}\mathcal{X})^{\horth})$ and $({}^{\horth}(\mathcal{X}^{\horth}),\mathcal{X}^{\horth})$. 
\end{remark}

By the properties proved in Section \refbf{horth_subs} we obtain the following closure properties for the classes composing a \phfs:

\begin{proposition}\label{closure_phfs}
Let $\F=(\E,\M)$ be a \phfs. Then 
\begin{enumerate}
\item $\E$ and $\M$ are closed under isomorphisms in $\cD^\due $;
\item $\E\cap \M$ is the class of all isomorphisms;
\item $\E$ is closed under arbitrary coproducts and $\M$ is closed under arbitrary products;
\item $\E$ and $\M$ are closed under retracts;
\item $\E$ is closed under homotopy pushouts and $\M$ is closed under homotopy pullbacks;
\item $\E$ is closed under homotopy colimits in the sense that, in the same setting of Lemma \refbf{horth_colimits}, if $\alpha_i\in \E$ for any $i\in \N$, then $\varphi\in \E$. A dual property regarding homotopy limits holds for $\M$.
\end{enumerate}
\end{proposition}
The following two definitions are of capital importance for us, as they determine the class of factorization systems we are interested in:
\begin{definition}[triangulated torsion theory]\label{hott}
A \hfs $\F=(\E,\M)$ is said to be a \emph{triangulated torsion theory} (for short, \htth
) if both $\E$ and $\cD$ are $3$-for-$2$ classes.
\end{definition}
\begin{definition}[normal triangulated fs]\label{hontt}
Let $\F=(\E,\M)$ be a \hfs in $\cD$. We say that $\F$ is \emph{normal} if, whenever we have a factorization of a final map $X\to 0$ as follows
\[
X \xto{e} T \xto{m}  0\ \ \text{ with }\ \ e\in \E,\ m\in \M\,,
\] and a triangle of the form $R\to X \xto{e} T\to \Sigma R$, the map $(R\to 0)$ belongs to $\E$.
\end{definition}
\section{The triangulated Rosetta stone}\label{t_structure_subs}

As in Section \refbf{section_homo_FS}, let us fix throughout this section a triangulated category $\cD$ with shift functor $\Sigma\colon \cD \xto{\simeq} \cD$.
\begin{definition}\label{def:teestru}
Recall that a \emph{$t$-structure} in $\cD$ is a pair $\t = (\cD^{\leq  0}, \cD^{\geq 0}  )$ of full sub-categories of $\cD$ that satisfy the following properties, where $\cD^{\leq n} \coloneqq \Sigma^{-n}\cD^{\leq 0}$ and $\cD^{\geq n} \coloneqq \Sigma^{-n} \cD^{\geq 0} $, for any $n\in\Z$:
\begin{enumerate}[label=t\arabic*)]
\item\label{tee:fst} $\cD(X,Y)=0$ for any $X \in \cD^{\leq  0}$ and $Y \in  \cD^{\geq  1} $;
\item\label{tee:snd} $\cD^{\leq-1}\subseteq \cD^{\leq  0}$ and $ \cD^{\geq 1} \subseteq  \cD^{\geq  0} $;
\item\label{tee:trd} for any $X \in \cD$ there is a distinguished triangle 
\[
X^{\leq0} \to X \to X^{\geq1}\to \Sigma X^{\leq0},
\] 
with $X^{\leq0} \in \cD^{\leq 0}$ and $X^{\geq1}\in  \cD^{\geq 1} $.
\end{enumerate}
\end{definition}
Given a $t$-structure $\t = (\cD^{\leq  0}, \cD^{\geq 0} )$ in $\cD$, one obtains two functors
\[
\tau^{\leq0}\colon \cD\to \cD^{\leq  0} \quad\text{and}\quad \tau^{\geq1}\colon \cD\to  \cD^{\geq  1} ,
\]
that are respectively the right adjoint to the inclusion $\cD^{\leq  0}\to \cD$ and the left adjoint to the inclusion $ \cD^{\geq  1} \to \cD$.

\begin{notat}
For an object $X\in\cD$ we will generally write $X^{\leq0}$ for $\tau^{\leq0}X$ and $X^{\geq1}$ for $\tau^{\geq1}X$. Furthermore, we will generally denote the unit of the co-reflection $\tau^{\leq0}$ and the co-unit of the reflection $\tau^{\geq1}$ by the following symbols: 
\[
X^{\leq0} \xto{\sigma_X} X \xto{\rho_X} X^{\geq1}.
\]
For any $n\in\Z$, we let $\tau^{\leq n} \coloneqq \Sigma^{-n}\tau^{\leq0}\Sigma^n$ and $\tau^{\geq n} \coloneqq \Sigma^{-n}\tau^{\geq0}\Sigma^n$. We adopt similar notational conventions for these shifted functors.
\end{notat}


\begin{remark}
We can equally define a $t$-structure as a single full additive subcategory  $\t\subseteq \cD$ such that 
\begin{itemize}
	\item $\Sigma\tee \subseteq \tee$;
	\item each object $X\in\cD$ fits into a distinguished triangle $X_{\tee}\to X\to X_{\tee^\perp}\to \Sigma X_{\tee}$ such that $X_\tee\in\tee$, $X_{\tee^\perp}\in\tee^\perp = \{Y\mid \cD(X,Y)=0,\; \forall X\in\tee\}$.
\end{itemize}
This equivalent description of $t$-structures calls $\tee$ an \emph{aisle} and $\tee^\perp$ a \emph{coaisle}.  We will usually blur the distinction between a $t$-structure and its aisle, since the correspondence between the two is obviously bijective under $\cD^{\leq 0} \leftrightarrows \text{aisle}$.
\end{remark}

\subsection{The induced \hfs of a $t$-structure}
Fix a $t$-structure $\t = (\cD^{\leq  0}, \cD^{\geq 0} )$ in $\cD$, and consider the following two classes of morphism 
\begin{align}
\E_\t &  \coloneqq \{\phi\in \cD^\due \mid\tau^{\geq1}\phi \text{ is an iso}\}\notag\\
\M_\t &  \coloneqq \{\psi\in \cD^\due \mid\tau^{\leq0}\psi \text{ is an iso}\}.
\end{align}
This subsection is devoted to the proof of the fact that $\F_\t \coloneqq (\E_\t,\M_\t)$ is a \hfs.

\begin{lemma}[cartesian characterization of $\F_\t$]\label{classes_via_cartesian}
In the above setting, a morphism $(\phi\colon X\to Y)\in\cD^\due$ belongs to $\E_\t$ if and only if the  square
\begin{equation}\label{cartesian?}
\xymatrix{
X^{\leq0}\ar@{}[dr]\ar[r]^{}\ar[d]_{\phi^{\leq0}}&X\ar[d]^{\phi}\\
Y^{\leq0}\ar[r]_{}&Y
}
\end{equation}
is homotopy cartesian. Thus, if $\phi\in\E_\t$, the cone of $\phi$  belongs to $ \cD^{\leq 0} $. Dually, $(\psi\colon X\to Y)\in\cD^\due$ belongs to $\M_\t$ if and only if the square
\begin{equation*}
\xymatrix{
X\ar@{}[dr]\ar[r]^{}\ar[d]_{\psi}&X^{\geq1}\ar[d]^{\psi^{\geq1}}\\
Y\ar[r]_{}&Y^{\geq1}
}
\end{equation*}
is homotopy cartesian. Thus, if $\psi\in\M_\t$, the cone of $\psi$ belongs to $ \cD^{\geq 0} $.
\end{lemma}
\begin{proof}
Suppose first that $\phi\in \E_\t$. 
By \cite[Remark \textbf{1.3.15}]{Neeman}, the square in \eqref{cartesian?} can be completed to a good morphism of triangles
\[
\xymatrix{
X^{\leq0}\ar@{}[dr]\ar[r]^{}\ar[d]_{\phi^{\leq0}}&X\ar[d]^{\phi}\ar[r]&X^{\geq1}\ar[r]\ar@{.>}[d]&\Sigma X\ar[d]\\
Y^{\leq0}\ar[r]_{}&Y\ar[r]&Y^{\geq1}\ar[r]&\Sigma Y
}
\]
while by \cite[\aprop\textbf{1.1.9}]{BBD}, the unique map completing the above square to a morphism of triangles is $\tau^{\geq1}\phi$. Thus, we get that the following candidate triangle is in fact a triangle
\[
X^{\leq0}\oplus\Sigma^{-1}Y^{\geq1}\to X\oplus Y^{\leq0}\to X^{\geq1}\oplus Y\to \Sigma X\oplus Y^{\geq1}.
\]
The above triangle is the direct sum of the following candidate triangles (see \cite[Lemma \textbf{1.2.4}]{Neeman})
\[
\Sigma^{-1}Y^{\geq1}\to 0\to X^{\geq1}\tilde\to Y^{\geq1}\ \ \text{ and }\ \ X^{\leq0}\to X\oplus Y^{\leq0}\to Y\to \Sigma X,
\]
showing that the candidate triangle on the right-hand-side is a distinguished triangle (as it is a summand of a distinguished triangle). The existence of such a triangle means exactly that the square in \eqref{cartesian?}  is homotopy cartesian.

On the other hand, suppose the square in \eqref{cartesian?} is homotopy cartesian. By \cite[Remark \textbf{1.4.5}]{Neeman}, this can be completed to a good morphism of triangles
\[
\xymatrix{
X^{\leq0}\ar@{}[dr]\ar[r]^{}\ar[d]_{\phi^{\leq0}}&X\ar[d]^{\phi}\ar[r]&X^{\geq1}\ar[r]\ar@{.>}[d]|\cong&\Sigma X\ar[d]\\
Y^{\leq0}\ar[r]_{}&Y\ar[r]&Y^{\geq1}\ar[r]&\Sigma Y
}
\]
Invoking again  \cite[\aprop\textbf{1.1.9}]{BBD}, we obtain that $\tau^{\geq1}\phi$ is an iso.
\end{proof}
\begin{lemma}\label{cart_clos}
Consider a homotopy cartesian square 
\[
\xymatrix{
X\ar@{}[dr]|\boxvoid\ar[r]\ar[d]_{\phi}&Y\ar[d]^{\psi}\\
X'\ar[r]&Y'
}
\]
If $\phi^{\geq 0}$ is an isomorphism, then $\psi^{\geq 0}$ is an isomorphism. Dually, if $\psi^{\leq 0}$ is an isomorphism, then $\phi^{\leq 0}$ is an isomorphism. In other words, $\E_\t$ is closed under homotopy pushouts and $\M_\t$ is closed under homotopy pullbacks.
\end{lemma}
\begin{proof}
Suppose first that $\phi^{\geq0}$ is an isomorphism. This means that $ \cD^{\geq  0} (\phi^{\geq 0},B)$ is an isomorphism for any $B\in  \cD^{\geq  0} $ or, equivalently, $\cD(\phi,B)$ is an isomorphism for any $B\in  \cD^{\geq 0} $. We have to show the same property holds for $\psi$. Consider the following morphism of triangles:
\[
\xymatrix{
Z\ar[r]\ar@{=}[d]&X\ar@{}[dr]|\boxvoid\ar[r]\ar[d]_{\phi}&Y\ar[d]^{\psi}\ar[r]&\Sigma Z\ar@{=}[d]\\
Z\ar[r]&X'\ar[r]&Y'\ar[r]&\Sigma Z
}
\]
For any given $B\in  \cD^{\geq 0} $, we obtain a morphism of long exact sequences:
\[
\xymatrix@C=15pt{
\cdots\ar[r]&\cD(\Sigma X',B)\ar[r]\ar[d]|{\cong}&\cD(\Sigma Z,B)\ar[r]\ar@{=}[d]&\cD(Y',B)\ar[r]\ar[d]&\cD(X',B)\ar[r]\ar[d]|{\cong}&\cD(Z,B)\ar[r]\ar@{=}[d]&\cdots\\
\cdots\ar[r]&\cD(\Sigma X,B)\ar[r]&\cD(\Sigma Z,B)\ar[r]&\cD(Y,B)\ar[r]&\cD(X,B)\ar[r]&\cD(Z,B)\ar[r]&\cdots
}
\]
where $\cD(\Sigma\phi,B)$ is an isomorphism because $\cD(\phi,\Sigma^{-1}B)$ is an isomorphism, since $\Sigma^{-1}B\in  \cD^{\geq 1} \subseteq  \cD^{\geq  0} $. Now, by the Five Lemma we obtain that $\cD(\psi,B)$ is an isomorphism for any $B\in  \cD^{\geq  0} $, that is, $\psi^{\geq0}$ is an isomorphism. The proof of the second part of the statement is dual.
\end{proof}
\begin{lemma}\label{all_maps_are_crumbled}
Any morphism in $\cD$ is $\F_\t$-crumbled. 
\end{lemma}
\begin{proof}
Take a map $\phi\colon X\to Y$ in $\cD$, and let us prove that $\phi$ is $\F_\t$-crumbled. Let us start taking a homotopy pullback of the maps $\phi^{\geq1}$ and $\rho_Y$:
\[
\xymatrix{
P\ar@{.>}[r]\ar@{.>}[d]_{\phi_m}\ar@{}[dr]|{\boxvoid}&X^{\geq1}\ar[d]^{\phi^{\geq1}}\\
Y\ar[r]_{\rho_Y}&Y^{\geq1}
}
\]
By Lemma \refbf{classes_via_cartesian}, $\phi_m\in \M_\t$. Consider also the following commutative solid diagram
\[
\xymatrix{
X\ar@{.>}[dr]|{\exists \phi_e}\ar@/_-10pt/[rrd]^{\rho_X}\ar@/_10pt/[rdd]_\phi\\
&P\ar[r]\ar[d]\ar@{}[dr]|{\boxvoid}&X^{\geq1}\ar[d]^{\phi^{\geq1}}\\
&Y\ar[r]_{\rho_Y}&Y^{\geq1}
}
\]
Then there exists a (non-unique, see \cite[p. \textbf{54}]{Neeman}) map $\phi_e\colon X\to p$ that makes the diagram commute. Finally consider the following diagram, where the dotted arrow is obtained completing to a good map of triangles:
\[
\xymatrix{
X^{\leq0}\ar@{.>}[d]\ar[r]^{\sigma_X}&X\ar[d]|{\phi_e}\ar[r]^{\rho_X}&X^{\geq1}\ar@{=}[d]\ar[r]&\Sigma X^{\leq0}\ar[d]\\
Y^{\leq0}\ar@{=}[d]\ar[r]&P\ar[r]\ar[d]|{\phi_m}\ar@{}[dr]|{\boxvoid}&X^{\geq1}\ar[d]^{\phi^{\geq1}}\ar[r]&\Sigma Y^{\leq0}\ar@{=}[d]\\
Y^{\leq0}\ar[r]_{\sigma_Y}&Y\ar[r]_{\rho_Y}&Y^{\geq1}\ar[r]&\Sigma Y^{\leq0}
}
\]
By construction $\phi=\phi_m\phi_e$. It remains to show that $\phi_e\in \E_\t$. By Lemma \refbf{classes_via_cartesian}, we have to verify that the top left square is homotopy cartesian. Indeed, take the following mapping cone, which is distinguished since we took a good morphism of triangles in our construction:
\[
X\oplus Y^{\leq0}\to P\oplus X^{\geq1}\to X^{\geq1}\oplus \Sigma X^{\leq0}\to\Sigma X\oplus \Sigma Y^{\leq 0}.
\]
This triangle is the direct sum of the following two candidate triangles (see \cite[Lemma \textbf{1.2.4}]{Neeman}):
\begin{gather*}
0\to X^{\geq1}\to X^{\geq1}{\to}0,\\
X\oplus Y^{\leq0}\to P\to  \Sigma X^{\leq0}{\to}\Sigma X\oplus\Sigma Y^{\leq0},
\end{gather*}
showing that $X^{\leq0}\to X\oplus Y^{\leq0}\to P \to \Sigma X^{\leq0}$ is distinguished.
\end{proof}

\begin{lemma}
Given $e\in \E_\t$ and $m\in\M_\t$, we have $e \horth  m$. 
\end{lemma}
\begin{proof}
Complete $e$ and $m$ to triangles as follows:
\[
E_0 \xto{e}  E_1 \xto{\alpha_e} C_e \xto{\beta_e} \Sigma E_0\ \ \ M_0 \xto{m}  M_1 \xto{\alpha_m} C_m \xto{\beta_m} \Sigma M_0,
\] 
By Lemma \refbf{classes_via_cartesian}, there are morphisms of triangles, with $\phi=e^{\leq0}$ and $\psi=m^{\geq1}$,
\[
\xymatrix{
X_0\ar@{}[dr]|{\square}\ar[d]_\phi\ar[r]&E_0\ar[d]^e&&M_0\ar@{}[dr]|{\square}\ar[d]_m\ar[r]&Y_0\ar[d]^\psi\\
X_1\ar[d]_{\alpha_e'}\ar[r]&E_1\ar[d]^{\alpha_e}&&M_1\ar[d]_{\alpha_m}\ar[r]&Y_1\ar[d]^{\alpha_m'}\\
C_e\ar[d]_{\beta_e'}\ar@{=}[r]&C_e\ar[d]^{\beta_e}&&C_m\ar[d]_{\beta_m}\ar@{=}[r]&C_m\ar[d]^{\beta_m'}\\
\Sigma X_0\ar[r]&\Sigma E_0&&\Sigma M_0\ar[r]&\Sigma Y_0}
\]
where $X_0, X_1\in  \cD^{\leq 0} $ and $Y_0, Y_1\in  \cD^{\geq 1} $. Using the closure properties of $ \cD^{\leq 0} $ and $ \cD^{\geq 1} $, one can show that $C_e\in  \cD^{\leq 0} $ and $\Sigma^{-1}C_m\in  \cD^{\geq  1} $. Thus, $\cD(C_e,\Sigma^{-1}C_m)=0$ by condition \refbf{def:teestru}.\refbf{tee:fst}, giving us  \refbf{wobbly}.\textsc{ho}1 It remains to verify condition \refbf{wobbly}.\textsc{ho}2, that is, suppose  we have a map $f\colon C_e\to C_m$ whose image in $\cD(E_1,\Sigma M_0)$ is trivial and let us prove that $f=0$. Indeed, we know that $\beta_mf\alpha_e=0$, so also $\beta'_mf\alpha'_e=0$ and thus we can find a morphism of triangles as follows
\[
\xymatrix{
0\ar[r]\ar[d]&X_1\ar@{=}[r]\ar@{.>}[d]|{f_1}&X_1\ar[r]\ar[d]|{f\alpha'_e}&0\ar[d]\\
Y_0\ar[r]&Y_1\ar[r]^{\alpha'_m}&C_m\ar[r]^{\beta'_m}&\Sigma Y_0
}
\]
showing that $f\alpha'_e=\alpha'_m f_1$ for some $f_1\colon X_1\to Y_1$. But $\cD(X_1,Y_1)=0$ by \refbf{def:teestru}.\refbf{tee:fst}, so $f_1=0$, showing that $f\alpha'_e=0$. Hence, we can find a morphism of triangles as follows
\[
\xymatrix{
X_0\ar[r]^{e}\ar[d]&X_1\ar[r]^{\alpha'_e}\ar[d]&C_e\ar[r]^{\beta_e'}\ar[d]|{f}&\Sigma X_0\ar@{.>}[d]^{f_2}\\
\Sigma^{-1}C_m\ar[r]&0\ar[r]&C_m\ar@{=}[r]&C_m
}
\]
showing that $f=f_2\beta_e'$, for some $f_2\colon \Sigma X_0\to C_m$. Now, since $\Sigma X_0\in  \cD^{\leq -1} $ and $C_m\in  \cD^{\geq 0} $, $f_2=0$ and so also $f=0$, as desired.
\end{proof}
\begin{proposition}
The pair of sub categories $\F_\t=(\E_\t,\M_\t)$ defines a \hfs.
\end{proposition}
\begin{proof}
We have already seen that any morphism is $\F_\t$-crumbled and that $\E_\t\subseteq \lhorth{\!\M_\t}$. Let us show the converse inclusion. Indeed, let $(\phi\colon X\to Y)\in \lhorth{\!\M_\t}$ and choose a factorization $\phi=\phi_m\phi_e$ with $\phi_e\in \E_\t$ and $\phi_m\in\M_\t$. By the usual $3\times 3$-lemma in triangulated categories, we can complete the commutative square
\[
\xymatrix{
X\ar[r]^{\phi_e}\ar[d]_{\phi}&p\ar[d]^{\phi_m}\\
Y\ar@{=}[r]&Y
}
\]
to a diagram where all the rows and columns are distinguished triangles, and where everything commutes but the top left square, that anti-commutes:
\[
\xymatrix{
\Sigma^{-1}C_e\ar@{=}[d]\ar[r]&\Sigma^{-1}C_\phi\ar[d]\ar[r]&\Sigma^{-1}C_m\ar[d]\ar[r]&C_e\ar@{=}[d]\\
\Sigma^{-1}C_e\ar[d]\ar[r]&X\ar[d]_\phi\ar[r]^{\phi_e}&P\ar[d]^{\phi_m}\ar[r]&C_e\ar[d]\\
0\ar[r]\ar[d]&Y\ar@{=}[r]\ar[d]&Y\ar[r]\ar[d]&0\ar[d]\\
C_e\ar[r]&C_\phi\ar[r]&C_m\ar[r]&\Sigma C_e}
\]
Now, since $\phi\in \lhorth{\!\M_\t}$, it follows by \refbf{wobbly}\textbf{.}\textsc{ho2}' that the map $C_\phi\to C_m$ in the above diagram is the trivial map. Thus, $\Sigma C_e\cong C_m\oplus \Sigma C_\phi$, in particular $C_m$ is a summand of $\Sigma C_e\in \Sigma  \cD^{\leq  0} = \cD^{\leq -1} $. Hence, $C_m\in  \cD^{\leq -1} \cap  \cD^{\geq 0} =0$, showing that $\phi_m$ is an isomorphism, so that $\phi\cong \phi_e\in \E_\t$. 
\end{proof}

\subsection{$t$-structures are  normal \htth}
We now concentrate on showing how each $t$-structure on $\cD$ naturally induces a \htth and vice-versa; the basic idea is to mimic the proof of \cite[\athm\textbf{3.1.1}]{tstructures} tailoring the argument to the triangulated setting.

\begin{lemma}
$\F_\t=(\E_\t,\M_\t)$ is a normal \htth.
\end{lemma}
\begin{proof}
We have already proved that $\F_\t$ is a \hfs, while the fact that $\E_\t$  and $\M_\t$ are $3$-for-$2$ classes is a trivial consequence of their definition, as they are the pre-image (under $\tau^{\geq1}$ and $\tau^{\leq0}$, respectively) of the class of all isomorphisms, which is a $3$-for-$2$ class. It remains to show that $\F_\t$ is normal. Consider a factorization of a final map $X\to 0$ as follows
\[
X \xto{e} T \xto{m}  0\ \ \text{ with }\ \ e\in \E_\t,\ m\in \M_\t,
\] 
and a triangle of the form $R\to X \xto{e} T\to \Sigma R$. We should prove that the map $(R\to 0)$ belongs to $\E_\t$, that is, that $R\in  \cD^{\leq 0} $. By Lemma \refbf{classes_via_cartesian}, $T\in  \cD^{\geq 1} $. Since $e\in \E_\t$ and using Lemma \refbf{classes_via_cartesian}, we can construct a commutative diagram as follows:
\[
\xymatrix{
X^{\leq0}\ar@{}[dr]|\square\ar[r]\ar[d]&X\ar[r]\ar[d]^{e}&X^{\geq1}\ar[r]\ar[d]|\cong&\Sigma X^{\leq0}\ar[d]\\
T^{\leq0}\ar[r]&T\ar[r]&T^{\geq1}\ar[r]&\Sigma T^{\leq0}}
\]
Since $T\in  \cD^{\geq 1} $, we get $T^{\leq0}=0$ and $T\cong T^{\geq 1}\cong X^{\geq1}$, so the fact that the square on the left-hand-side in the above diagram is homotopy cartesian provides us with a distinguished triangle of the form
\[
X^{\leq0}\to X\to T\to \Sigma X^{\leq0}.
\]
In particular, $R\cong X^{\leq0}\in  \cD^{\leq 0} $ as desired.
\end{proof}
\begin{lemma}\label{htt_induces_t_structure}
For a normal \htth $\F=(\E,\M)$ in $\cD$, $\t_\F \coloneqq (0/\E,\Sigma(\M/0))$ is a $t$-structure.
\end{lemma}
\begin{proof}
We verify the three axioms of a $t$-structure:
\begin{itemize}
\item Let $X\in 0/\E$ and $Y\in \M/0$, we have to show that $\cD (X,Y)=0$. Indeed, let $\varphi \colon X\to Y$ and consider the following diagram
\[
\xymatrix{
0\ar[r]\ar[d]&0\ar[d]\\
X\ar[r]^{\varphi}\ar@{=}[d]&Y\ar@{=}[d]\\
X\ar[r]^{\varphi}\ar[d]&Y\ar[d]\\
0\ar[r]&0
}
\]
Notice that $(0\to X)\in \E$. Furthermore, $0\to 0$ is an isomorphism so it belongs to $\M$, as well as $Y\to 0$; since $\M$ is a $2$-for-$3$ class, this means that also $0\to Y$  belongs to $\M$. By condition \refbf{wobbly}\textbf{.}\textsc{ho2}, we get $\varphi=0$.
\item Let $X\in 0/\E$. Reasoning as in verifying \refbf{def:teestru}.\refbf{tee:fst} above, one can show that the $2$-for-$3$  property of $\E$ implies that $X\to 0$ belongs to $\E$. Consider now the following homotopy cartesian square:
\[
\xymatrix{
X\ar@{}[rd]|\square\ar[r]\ar[d]&0\ar[d]\\
0\ar[r]&\Sigma X
}
\]
By Proposition \refbf{closure_phfs}, the map $0\to \Sigma X$ belongs to $\E$, that is $\Sigma(0/\E)\subseteq 0/\E$. One verifies similarly that $\M/0\subseteq \Sigma(\M/0)$.
\item Let $X\in \cD$, consider a factorization of the map $X\to 0$ as follows:
\[
X \xto{e} T \xto{m}  0\ \ \text{ with }\ \ e\in \E,\ m\in \M.
\] 
Now we can complete the map $e$ to a triangle to get
\[
R\to X \xto{e} T\to \Sigma R.
\]
By the normality of $\F$, $R\in 0/\E$ and $T\in \M/0$. \qedhere
\end{itemize}
\end{proof}

\begin{theorem}[the triangulated Rosetta stone]\label{triang-rosetta}
Let $\cD$ be a triangulated category, then there is a bijective correspondence
\[
\xymatrix@R=0pt{
\Phi:\left\{
{\begin{smallmatrix}
\text{normal triangulated}\\
\text{\textsc{tth}s on }\cD
\end{smallmatrix}}
\right\}
\ar@{<->}[rr]&&
\left\{
{\begin{smallmatrix}
\text{$t$-structures}\\
\text{ on }\cD
\end{smallmatrix}}
\right\}:\Psi\\
(\E,\M) \ar@{|->}[rr]&& \Big(0/\E, \Sigma(\M/0)\Big) \\
(\E_\t,\M_\t) \ar@{<-|}[rr]&& \t.
}
\]
%
%
\end{theorem}
\begin{proof}
We have already verified in the previous subsections that $\Phi$ and $\Psi$ are well-defined. Consider now a $t$-structure $\t$ and let us show that $\t=\Phi\Psi\t$, that is, we should verify that $ \cD^{\leq 0} =0/\E_\t$. But this is true since clearly $X\in  \cD^{\leq 0} $ if and only if $0\to X$ belongs to $\E_\t$, that is, $X\in 0/\E_\t$. 

On the other hand, let $\F=(\E,\M)$ and let us show that $\F=\Psi\Phi\F$. Let $\phi\in \E_{\t_\F}$, that is, $\phi^{\geq1}$ is an isomorphism and consider the following commutative square:
\[
\xymatrix{
X\ar[r]^{\rho_X}\ar[d]_\phi&X^{\geq1}\ar[d]^{\phi^{\geq1}}\\
Y\ar[r]^{\rho_Y}&Y^{\geq1}
}
\]
Notice that $\rho_X$ and $\rho_Y$ belong to $\E$ (in fact these reflections are constructed taking an $\F$-factorization of the final maps $X\to 0$ and $Y\to 0$, see the last part of the proof of Lemma \refbf{htt_induces_t_structure}). The composition $\rho_Y\phi=\phi^{\geq1}\rho_X$ belongs to $\E$ since $\phi^{\geq1}\in \E$ (as $\E$ contains any isomorphism) and we have already observed that $\rho_X\in \E$. For the $3$-for-$2$ property this means that $\phi\in \E$. This shows that $\E_{\t_\F}\subseteq \E$. One proves in the exact same way that $\M_{\t_\F}\subseteq \M$, but these two conditions together mean that $\F=\F_{\t_\F}$, as desired.\end{proof}

\section{Derivator factorization systems}\label{sec:squaring}

A {\em category of diagrams} is a full sub-2-category of the 2-category $\cat$ of small categories  that fulfills some closure properties, for which we refer to \cite[\adef\textbf{4.21}]{Moritz}. Let us just remark that every object $\mathbf{n}\in\cat$ belongs to any category $\Dia$ of diagrams.

For a given category of diagrams $\Dia$, a {\em pre-derivator} is  a 2-functor
\[
\D\colon \Dia^\opp\to \Cat
\]
A pre-derivator $\D$ is said to be {\em representable} if there is a category $\C$ such that $\D(I)=\C^{I}$ for any $I\in \Dia$, for any functor $u\colon J\to I$, the functor $\D(u)$ acts as 
\[
u^*\colon \C^I\to \C^J\qquad\text{such that}\qquad ( F\colon I\to \C)\mapsto (F\circ u\colon J\to I\to \C),
\]
and it acts on natural transformations in the obvious way. In the above situation we say that $\D$ is {\em represented by} $\C$, in symbols $\D=y(\C)$.

A pre-derivator $\D$ is a {\em derivator} if it satisfies a series of four axioms (Der1)--(Der4), for which we refer to \cite{Moritz}, as well as for the definitions of {\em pointed}, {\em strong}, and {\em stable} derivator. 

\medskip
Through this section let us fix the following minimal setting; from time to time we will to need work under stronger hypotheses (typically, we will assume that $\D$ is representable, or that it is a stable derivator):

\begin{setting}\label{setting_sec_3} 
Let $\Dia\subseteq \cat$ be a category of diagrams and we fix a pre-derivator
\[
\D\colon \Dia^\opp\longrightarrow \Cat
\]
that satisfies the following conditions for any $I\in \Dia$:
\begin{enumerate}
\item a morphism $\phi$ in $\D(I)$ is an isomorphism if and only if $\phi_i$ is an isomorphism in $\D(\uno)$ for any $i\in I$;
\item $\dia_\due(I)\colon \D^\due(I)\to \D(I)^\due$ is full and essentially surjective;
\item $\dia_\tre(I)\colon \D^\tre(I)\to \D(I)^\tre$ is full and essentially surjective.
\end{enumerate}
\end{setting}

If $\D=y(\C)$ for some category $\C$, then $\dia_I$ is an equivalence of categories for any $I\in \Dia$, so (1), (2) and (3) are always satisfied for this kind of pre-derivators. In fact, condition (1) is exactly (Der2), so in particular it is fulfilled by any derivator. In the language of \cite{Moritz}, condition (2) says that $\D$ is a strong pre-derivator. Finally, let us also remark that (1), (2) and (3) are satisfied by any stable derivator.

\subsection{The comonoid $\due$}\label{comonoid_due_subs}
Consider the \emph{point functor} $\pt\colon \due\to \uno$ that collapses the arrow category $\due$ to the point category $\uno$. This functor has both a right and a left adjoint choosing respectively the terminal and initial object of $\due$:
\[
\xymatrix@C=1.5cm{
**[l] 0\dashv \pt \dashv 1\colon \due\ar[r]|\pt & \ar@/_-1pc/@<3pt>[l]|{1}\ar@/_1pc/@<-3pt>[l]|{0} {\uno}
}
\]
where the left adjoint $0\colon \uno\to \due$ sends the unique object of $\uno$ to $0\in \due$, while the right adjoint $1\colon \uno\to \due$ sends the unique object of $\uno$ to $1\in \due$.

\begin{remark}\label{two-comonoid}
Like every object of $\cat$, the category $\due$ has the structure of a comonoid, where the co-multiplication is give by the diagonal map $\Delta\colon \due\to \due\times\due$, and the counit by the point functor $\pt\colon \due \to \uno$ above. It is in fact easy to check by hand the co-associativity and co-unitality relations:
\[\begin{cases}
(\pt\times \id_\due)\circ\Delta=\id_\due=(\id_\due\times \pt)\circ\Delta\notag\\
(\Delta\times\id_\due)\circ\Delta=(\id_\due\times \Delta)\circ\Delta.
\end{cases}\]
\end{remark}

Applying $\D$ to these functors we obtain the following adjunctions and isomorphisms thereof:
\[
\xymatrix@C=1.5cm{
**[l] \pt_! \cong 1^* \dashv \pt^* \dashv 0^* \cong \pt_* \colon
\D(\uno)\ar[r]|{\pt^*}&\ar@/_-1pc/@<3pt>[l]|{0^*}\ar@/_1pc/@<-3pt>[l]|{1^*} {\D(\due)}
}
\]
It is in fact easy to see that $\pt_!\cong 1^*$, $1_*\cong \pt^*$, $\pt_*\cong 0^*$ and $0_!\cong \pt^*$. Furthermore, the functor $0\colon \uno \to \due$ is a \emph{sieve} and $1\colon \uno\to \due$ is a \emph{co-sieve}; thus whenever $\D$ is a pointed derivator, by \cite[Corollary \textbf{3.8}]{Moritz}, $0_*$ has a right adjoint $0^!\colon \D(\due)\to \D(\uno)$, while $1_!$ has a left adjoint $1^?\colon \D(\due)\to \D(\uno)$. Hence, we end up with the string of adjoint functors
\[
\xymatrix@M=3mm@R=2cm{
     \ar@<4.5em>[d]|{0^!}
    \ar@<-4.5em>[d]|{1^?}
	\D(\due)
	\ar@<-1.5em>[d]|{\begin{smallmatrix} {}\\ \pt_!\\ 1^* \\{}\end{smallmatrix}}
	 \ar@<1.5em>[d]|{\begin{smallmatrix} {}\\ \pt_* \\ 0^* \\{}\end{smallmatrix}}\\
	\ar@{^{(}->}[u]|{\begin{smallmatrix} {}\\ 0_! \\ \pt^*  \\ 1_*\\{}\end{smallmatrix}}
	\D(\uno)
	 \ar@{^{(}->}@<3em>[u]|{1_!}
	\ar@{^{(}->}@<-3em>[u]|{0_*}
}
\]
(the functors are depicted from left to right respecting the adjointness relation, and functors on the same arrow are canonically isomorphic). Notice that all functors $1_!,$ $0_!\cong 1_*$, and $0_*$ are fully faithful. This is obvious since intuitively these three functors send an object into its initial, identity, and terminal arrow respectively.
\begin{notat}\label{la-kappa}
As a consequence of the fact that $0_!$ is fully faithful, the composition $0_!0^* \xto{\epsilon} \id \xto{\eta} 1_*1^*\cong 0_!1^*$ is of the form $0_!\kappa$ for a unique $\kappa\colon 0^* \to 1^*$.
\end{notat}
\begin{lemma}\protect{\cite[\aprop\textbf{3.24}]{Moritz}}
Suppose $\D$ is a pointed derivator, let $C\colon \D(\due)\to \D(\uno)$, $F\colon \D(\due)\to \D(\uno)$, $\Sigma\colon \D(\uno)\to \D(\uno)$, and $\Omega\colon \D(\uno)\to \D(\uno)$ be respectively the \emph{cone}, \emph{fiber}, \emph{suspension} and \emph{loop} functors, as defined in \cite[§\textbf{3.3}]{Moritz}. Then, $C\cong 1^?$, $\Sigma\cong 1^?0_*$, $F\cong 0^!$, and $\Omega\cong 0^!1_!$.
\end{lemma}
We conclude this subsection with two technical lemmas, which apply in case $\D$ is a stable derivator, that will make our life easier in the rest of the section. 
\begin{lemma}[The standard triangle of a coherent morphism]\label{simple_tria_for_orth}
Suppose $\D$ is a stable derivator. Given $X\in \D(\due)$, there is a triangle of the form
$$X\xto{\varphi_X} \pt^*\pt_!X\longrightarrow 0_*C(X)\longrightarrow \Sigma X,$$
where $\varphi_X\colon X\to \pt^*\pt_!X$ is the unit of the adjunction $(\pt_!,\pt^*)$.
\end{lemma}
\begin{proof}
Complete $\varphi_X$ to a triangle as follows
\[
X\xto{\varphi_X} \pt^*\pt_!X\longrightarrow K\longrightarrow \Sigma X
\]
The underlying diagram of the above triangle has the following form
\[
\xymatrix{
X_0\ar[r]\ar[d]&X_1\ar@{=}[d]\ar[r]&C(X)\ar[r]\ar[d]&\Sigma X_0\ar[d]\\
X_1\ar@{=}[r]&X_1\ar[r]&0\ar[r]&\Sigma X_1
}
\]
It is then clear that $K\cong 0_*C(X)$.
\end{proof}
\begin{notat}
For each $i<j$ in $\{0,1,2\}$ we denote by $(i,j)$ the functor 
\[
(i,j)\colon \Delta^{\{0,1\}}\hookrightarrow \Delta^{\{i,j\}}\subset\Delta^{2};
\] 
this slightly unusual notation for the co-face maps $\{\delta_2^i\colon \due\to \tre \mid i=0,1,2\}$ is motivated by the belief that $X_{\delta^2_i}$ or $\delta_i^{2,*}X$ to denote the image of $X$ under $\delta_i^{2,*} \colon \D(\tre)\to \D(\due)$ are unreadable clutters confronted with the simpler $X_{(i,j)}$.
\end{notat}
\begin{lemma}[The standard triangle of a coherent 2-simpleX]\label{factorization_triangle}
Suppose $\D$ is a stable derivator. Given $X\in \D(\tre)$, there is a triangle of the form
\[
(1,2)_!X_{(0,1)}\to 1_!X_1\oplus (0,2)_!X_{(0,2)}\to X\to \Sigma (1,2)_!X_{(0,1)}.
\]
\end{lemma}
\begin{proof}
We adopt a construction similar to one contained in \cite{porta2015universal}. 
Let $\epsilon_1\colon 1_!1^*X\to X$ and $\epsilon_{(0,2)}\colon (0,2)_!(0,2)^* X \to X$ be the co-units of the respective adjunctions. These obviously give a  map 
\[
1_!X_1\oplus (0,2)_!X_{(0,2)}\xto{\smat{\epsilon_1 & \epsilon_{(0,2)}}} X
\]
that can be completed to a triangle
\[
K\to 1_!X_1\oplus (0,2)_!X_{(0,2)}\to X\to \Sigma K.
\]
Since $(1_!X_1\oplus (0,2)_!X_{(0,2)})_0 \cong (1_!X_1)_0\oplus ((0,2)_!X_{(0,2)})_0 \cong X_0$, then $K_0=0$. Given how the functor $(1,2)_!$ acts on objects, $K \cong (1,2)_! Y$ for some $Y$; we now aim to prove that such a $Y$ is necessarily isomorphic to $X_{(0,1)}$. For this, apply the functor $(1,2)^*$ to the above triangle, 
 to obtain the following triangle in $\D(\due)$:
\begin{equation}\label{bad_triangle}
((1,2)_!Y)_{(1,2)}\to (1_! X_1\oplus (1,2)_!X_{(0,2)})_{(1,2)}\to X_{(1,2)}\to \Sigma ((1,2)_!Y)_{(1,2)}.
\end{equation}
Notice that the obvious natural transformation $\gamma\colon (0,1) \Rightarrow (1,2)$, can be viewed as a composition of natural transformations in the following two ways:
\[
\xymatrix{
\due \ar[rr]|{(0,2)}^*!/u5pt/{\labelstyle\ \Downarrow\alpha}_*!/d5pt/{\labelstyle\ \Downarrow\beta} \ar@/_-25pt/[rr]|{(0,1)}\ar@/_25pt/[rr]|{(1,2)} && \tre & = & \due \ar@/_-25pt/[rr]|{(0,1)}\ar@/_25pt/[rr]|{(1,2)}\ar@{}[rr]^*!/u6pt/{\labelstyle\;\;\; \Downarrow\beta'}_*!/d6pt/{\labelstyle\;\;\; \Downarrow\alpha'}
\ar[r]|{\pt}&\uno\ar[r]|{1}& \tre
}
\]
giving us the upper left square in the following commuting diagram in $\D[1]$:
	\[
		\xymatrix@C=2cm{
		X_{(0,1)} \ar[r]^{\alpha_X^*}\ar[d]_{(\beta')_X^*}& X_{(0,2)}\ar[d]^{\beta_X^*}&(0,2)^*(0,2)_!X_{(0,2)}\ar[d]|{\cong}^{\beta^*_{(0,2)_!X_{(0,2)}}}\ar[l]|{\cong}_(.6){(0,2)^*\epsilon_{(0,2),X}}\ar@{}[ld]|{(\bullet)}\\
		\pt^*X_1 \ar[r]^{(\alpha')_X^*}& X_{(1,2)}&(1,2)^*(0,2)_!X_{(0,2)}\ar[l]^(.6){(1,2)^*\epsilon_{(0,2),X}}\\
		\pt^*(1_!X_1)_1\ar[u]|\cong^{\pt^*1^*(\epsilon_{1,X})}\ar[r]|\cong_{(\alpha')_{1_!X_1}^*}&(1,2)^*1_!X_1\ar[u]_{(1,2)^*\epsilon_{1,X}}\ar@{}[lu]|{(\bullet\bullet)}
		}
	\]
The commutative squares marked by ($\bullet$) and ($\bullet\bullet$) tell us that the triangle in \eqref{bad_triangle} is isomorphic to a triangle of the form:
$$\xymatrix@C=2cm{
Y\ar[r]& \pt^* X_1\oplus X_{(0,2)}\ar[r]^(.6){((\alpha')^*,\beta^*)}& X_{(1,2)}\ar[r]& \Sigma Y,}$$
while the third commutative square shows that the following composition is trivial:	
\[
\xymatrix@C=2cm{ X_{(0,1)} \ar[r]^(.4){[(\beta')^*,-\alpha^*]^t} & 0_!X_1 \oplus X_{(0,2)} \ar[r]^(.6){[(\alpha')^*,\beta^*]} & X_{(1,2)}}
\]
We obtain a map $\varphi\colon X_{(0,1)}\to Y$ making the following diagram commutative:
\[
\xymatrix{
X_{(0,1)}\ar[dr]\ar@{.>}[d]|{\varphi}\\
Y\ar[r]& \pt^* X_1\oplus X_{(0,2)}\ar[r]& X_{(1,2)}\ar[r]& \Sigma Y
}
\]
To conclude our proof, it is enough to show that $\varphi$ is an isomorphism. For this, it is enough to show that $\varphi_0\colon X_0\to Y_0$ and $\varphi_1\colon X_1\to Y_1$ are isomorphisms in $\D(\uno)$. 
But in fact, applying $0^*$ and $1^*\colon \D(\due)\to \D(\uno)$ to the above diagram, we get the following diagrams in $\D(\uno)$:
\[
\xymatrix@C=20pt@R=20pt{
X_0\ar[d]\ar[dr]&&&       &           X_1\ar[d]\ar[rd]\\
Y_0\ar[r]& X_1\oplus X_0\ar[r]& X_1\ar[r]& \Sigma Y_0        &       Y_1\ar[r]& X_1\oplus X_2\ar[r]& X_2\ar[r]& \Sigma X_1
}
\]
respectively. This shows that $\varphi_0$ is a morphism that factors the kernel $X_0\to X_1\oplus X_0$ of the morphism $X_1\oplus X_0\to X_1$ through the kernel $Y_0\to X_1\oplus X_0$ of the same map. Hence, $\varphi_0$ is an isomorphism. A completely analogous argument shows that $\varphi_1$ is an isomorphism.
\end{proof}

\subsection{Coherent orthogonality}\label{Subs_Coh_Ort}
The objects of the category $\D(\due)$ can be thought of as ``coherent morphisms'' of $\D(\uno)$ (as opposed to the ``incoherent morphisms'', which are the objects of $\D(\uno)^{\due}$); in general the underlying diagram functor $\dia_\due \colon \D(\due)\to \D(\uno)^\due$ has no property whatsoever that ensures that a coherent diagram $X\in\D(\due)$ leaves a faithful image in its associated incoherent diagram $\dia_\due(X)$ (however, in our Setting \refbf{setting_sec_3}, $\dia_\due$ is at least full and essentially surjective).

In this subsection we are introducing a notion of \emph{coherent orthogonality} for a pair of objects in $\D(\due)$ that takes into account the richer structure of coherent diagrams. Indeed, let $X, Y\in\D(\due)$ and consider the unit $\varphi_X\colon X\to \pt^*\pt_!X$ of the adjunction $(\pt_!,\pt^*)$. Applying $\D(\due)(-,Y)$, and recalling that $\pt_!X\cong X_1$ and $\pt_*Y\cong Y_0$, we obtain a natural morphism
\[
\xymatrix@R=0pt{\D(\uno)(X_1,Y_0)\ar[r]^{\varphi_{X,Y}}&\D(\due)(X,Y)\\
(\pt_!X\xto{a} \pt_*Y)\ar@{|->}[r]&(X\xto{\varphi_X} \pt^*\pt_!X\xto{\pt^*a} \pt^*\pt_*Y\xto{\psi_Y} Y),}
\]
where $\psi_Y\colon\pt^*\pt_*Y\to Y$ is the counit of the adjunction $(\pt^*,\pt_*)$.

\begin{definition}[coherent orthogonality]\label{def_c_ort}
Given $X,Y\in \D(\due)$, we say that $X$ is \emph{left coherently orthogonal} to $Y$ (while $Y$ is \emph{right coherently orthogonal} to $X$), in symbols $X\corth Y$, if the map ${\varphi_{X,Y}}\colon\D(\uno)(X_1,Y_0)\to \D(\due)(X,Y)$ is an isomorphism.\\
Given $\mathcal{X}\subseteq \D(\due)$, we let
\begin{gather}
\mathcal{X}^{\corth } \coloneqq \{Y\in \D{(\due)}:X\corth Y, \ \forall X\in \mathcal{X}\}\notag\\
{}^{\corth }\mathcal{X} \coloneqq \{X\in \D{(\due)}:X\corth Y, \ \forall Y\in \mathcal{X}\}.\notag
\end{gather}
\end{definition}

Our first observation about coherent orthogonality is that, in case $\D$ is representable, we recover the classical notion of orthogonality of morphisms:
\begin{lemma}\label{corth=orth}
In Setting \refbf{setting_sec_3},  consider $X$ and $Y\in \D(\due)$. If $X\corth Y$ then for any commutative diagram 
\[
\xymatrix{
X_0\ar[d]_{\dia_\due X}\ar[r]^{\phi_0}&Y_0\ar[d]^{\dia_\due Y}\\
X_1\ar[r]_{\phi_1}\ar@{.>}[ur]|d&Y_1
}
\]
there is a $d\colon X_1\to Y_0$ such that $\phi_0=d\circ \dia_\due X $ and $\phi_1=\dia_\due Y \circ d$. Furthermore, if $\D=y(\C)$ is representable,  then $X\corth Y$ if and only if, in the above diagram, there is a unique arrow $d\colon X_1\to Y_0$ such that $\phi_0=d\circ X$ and $\phi_1=Y\circ d$.
\end{lemma}
\begin{proof}
By condition (2) in Setting \refbf{setting_sec_3}, a morphism $(\phi_0,\phi_1)\colon \dia_\due X \to \dia_\due Y $ can be lifted to a morphism $\phi\colon X\to Y$ such that $\dia_\due \phi =(\phi_0,\phi_1)$; if furthermore $\D$ is representable, than this lifting is unique. Now, $X\corth Y$ if and only if, given $\phi\colon X\to Y$ there is a unique morphism $\widetilde d\colon X_1\to Y_0$ such that $\varphi_{X,Y}(\widetilde d)=\phi$. This means that, letting $d \coloneqq \dia_\due(\widetilde d)$, we get a commutative diagram
\[
\xymatrix{
X_0\ar[d]_{\dia_\due X }\ar[r]^{\dia_\due X }&X_1\ar@{=}[d]\ar[r]^d&Y_0\ar@{=}[d]\ar@{=}[r]&Y_0\ar[d]^{\dia_\due Y }\\
X_1\ar@{=}[r]&X_1\ar[r]_d&Y_0\ar[r]_{\dia_\due  Y }&Y_1
}
\]
such that the composition of the top row is $\phi_0$ and that of the bottom row is $\phi_1$.
\end{proof}

The second thing we would like to point out is that, in case $\D$ is a stable derivator, so that $\D(\uno)$ is canonically a triangulated category,
then coherent orthogonality is equivalent to the homotopy orthogonality introduced in Definition \refbf{wobbly}:
\begin{lemma}\label{coherent_orth_is_wobbly}
Suppose $\D$ is a stable derivator and let $X,$ $Y\in \D(\due)$. Then, $X\corth Y$ if and only if $\dia_{\due}X\horth \dia_{\due}Y$.
\end{lemma}
\begin{proof}
Consider the triangle $X\to \pt^*\pt_!X\to 0_*C(X)\to \Sigma X$,  given by Lemma \refbf{simple_tria_for_orth}. Now apply $\D(\due)(-,Y)$ to this triangle to get the following long exact sequence:
\begin{align*}
\label{ort_ses}
\cdots &\to \D(\uno)(C(X),\Sigma^{-1}C(Y))\to \D(\uno)(X_1,Y_0)\to \D(\due)(X,Y)\to \\
\notag &\to\D(\uno)(C(X),C(Y))\to \D(\uno)(X_1,\Sigma Y_0)\to\cdots 
\end{align*}
By definition, $X\corth Y$ if and only if the map $\D(\uno)(X_1,Y_0)\to \D(\due)(X,Y)$ is bijective, but this map is injective if and only if  
\[
\D(\uno)(C(X),\Sigma^{-1}C(Y))\to \D(\uno)(X_1,Y_0)
\] 
is trivial (which is condition \refbf{wobbly}.\textsc{ho}1 for $(\dia_{\due}X,\dia_{\due}Y)$), while it is surjective if and only if the map 
\[
\D(\uno)(C(X),C(Y))\to \D(\uno)(X_1,\Sigma Y_0)
\] 
is injective (which is condition \refbf{wobbly}.\textsc{ho}2 for $(\dia_{\due}X,\dia_{\due}Y)$).
\end{proof}

Another characterization of coherent orthogonality, this time for classes, in a stable derivator $\D$, can be given using the composition functor $(0,2)^*\colon \D(\tre)\to \D(\due)$:

\begin{lemma}\label{level_wise_ff}
Suppose $\D$ is a stable derivator. The following are equivalent for a pair $\F=(\E,\M)$ of subclasses of $\D(\due)$:
\begin{enumerate}
\item $\E\corth \M$;
\item letting $\D_\F(\uno)\subseteq \D(\tre)$ be the full subcategory of those $X\in \D(\tre)$ such that $(0,1)^*X\in \E$ and $(1,2)^*X\in\M$, the restriction 
\[
\Psi:=(0,2)^*\restriction_{\D_\F(\uno)}\colon \D_\F(\uno)\to \D(\due)
\] 
is fully faithful. 
\end{enumerate}
\end{lemma}
\begin{proof}
Given $X,\, Y\in \D(\tre)$, by Lemma \refbf{factorization_triangle} there is a triangle
\[
(1,2)_!X_{(0,1)}\to 1_!X_1\oplus (0,2)_!X_{(0,2)}\to X\to \Sigma (1,2)_!X_{(0,1)}.
\]
Applying $\D(\tre)(-,Y)$ to this triangle we get a long exact sequence:
\begin{align*}
\cdots&\to \D(\due)(\Sigma X_{(0,2)},Y_{(0,2)})\oplus\D(\uno)(\Sigma X_1,Y_1)\xto{(*)} \D(\due)(\Sigma X_{(0,1)},Y_{(1,2)})\to\notag\\
&\to\D(\tre)(X,Y)\to\D(\due)(X_{(0,2)},Y_{(0,2)})\oplus\D(\uno)(X_1,Y_1)\to\notag\\
&\xto{(*)} \D(\due)(X_{(0,1)},Y_{(1,2)})\to \cdots
\end{align*}
If $X,\, Y\in\D_{\F}(\uno)$, then $X_{(0,1)},\, \Sigma X_{(0,1)}\in \E$ and $Y_{(1,2)}\in\M$. Thus, the following canonical maps are isomorphisms:
\begin{gather*}
\D(\uno)(X_1,Y_1)\xto{\cong}\D(\due)(X_{(0,1)},Y_{(1,2)})\\
\D(\uno)(\Sigma X_1,Y_1)\xto{\cong}\D(\due)(\Sigma X_{(0,1)},Y_{(1,2)})
\end{gather*}
showing that the two maps marked by $(*)$ in the above long exact sequence are (split) surjections. This shows that the natural map
\[
\D(\tre)(X,Y)\xto{\cong} \D(\due)(X_{(0,2)},Y_{(0,2)})
\]
is an isomorphisms, that is, $\Psi$ is fully faithful.
\end{proof}

We omit the proof of the following easy result

\begin{proposition}\label{ortho_to_self}
The following conditions are equivalent, for $X\in\D(\due)$.
\begin{enumerate}
	\item $X\corth X$;
	\item $X$ is an isomorphism (\ie $\dia_\due(X)$ is an isomorphism in $\D(\uno)$);
	\item $X\corth \D(\due)$;
	\item $\D(\due)\corth X$.
\end{enumerate}
\end{proposition}

\begin{definition}[derivator pre-factorization systems]\label{def_phfs}
Denote by $\D^{\due}$ the shifted pre-derivator $\D^{\due}(I) \coloneqq \D(\due\times I)$. 
Let $\mathbb E$ and $\mathbb M$ be two sub pre-derivators of $\D^{\due}$.
For any $I\in \Dia$, let $\E_I \coloneqq \mathbb E(I)$, $\M_I \coloneqq \mathbb M(I)$ and $\F_I \coloneqq (\E_I,\M_I)$.

The pair $\F \coloneqq (\mathbb E,\mathbb M)$ is a \emph{derivator pre-factorization system} (\cpfs for short) if $\E_I^{\corth}=\M_I$ and ${}^{\corth }\M_I=\E_I$,  for any $I\in\Dia$.
\end{definition}

The following lemma, whose proof is an easy consequence of Lemma \refbf{corth=orth}, describes the {\cpfs}s on a representable pre-derivator.

\begin{lemma}\label{dpfs=pfs_if_discrete}
Suppose that $\D=y(\C)$ is representable and let $\F=(\mathbb E,\mathbb M)$ be a pair of sub pre-derivators of $\D^\due$. Then, $\F$ is a \cpfs if and only if each $\F_I$ is an orthogonal pre-factorization system (see, for example, \cite[§\textbf{1}]{riehl2008factorization}).
\end{lemma}

Our next task is to describe the {\cpfs}s on a stable derivator $\D$ in terms of {\phfs}s on its images. Before that, we prove the following lemma giving some useful closure properties of {\cpfs}s.


\begin{lemma}\label{cfs_limit_colimit}
Suppose $\D$ is a derivator and let $\F=(\mathbb E,\mathbb M)$ be a \cpfs on $\D$. Given a functor $u\colon J\to I$ in $\Dia$, 
\begin{enumerate}
\item if $X\in \mathbb E(J)$, then $u_!X\in \mathbb E(I)$;
\item if $X\in \mathbb M(J)$, then $u_*X\in \mathbb M(I)$.
\end{enumerate}
\end{lemma}
\begin{proof}
Let $X\in \mathbb E(J)$ and $Y\in \mathbb M(I)$, then
\begin{align*}\D^\due(I)(u_!X,Y)&\cong \D^\due(J)(X,u^*Y)\\
&\cong \D(J)(X_1,(u^*Y)_0)\\
&\cong \D(J)(X_1,u^*(Y_0))\\
&\cong \D(I)(u_!(X_1),Y_0)\\
&\cong \D(I)((u_!X)_1,Y_0)
\end{align*}
where we used that $1^*$ commutes with Kan extensions, see \cite[\aprop\textbf{2.6}]{Moritz}. This shows that $u_!x{\corth }Y$ for any $Y\in\mathbb M(I)$, and so $u_!X\in \mathbb E(I)$. This proves (1), the proof of part (2) is completely analogous.
\end{proof}


\color{black}
Let $\D$ be a prederivator, $\F=(\mathbb E,\mathbb M)$ a pair of sub pre-derivators of $\D^\due$ and let 
$\D_\F\colon \Dia^\opp\to \Cat$
be a pre-derivator such that $\D_{\F}(I)\subseteq \D(I\times \tre)$ is the full subcategory spanned by those $X\in  \D(I\times \tre)$ such that $X_{(0,1)}\in \E_I$ and $X_{(1,2)}\in \M_I$. Denote by 
\begin{equation}\label{composition_morphism}
\Psi_\F\colon \D_\F\longrightarrow \D^{\due}
\end{equation}
the restriction of the morphism of derivators $(0,2)^\circledast\colon\D^{\tre}\to \D^{\due}$. In case $\D$ is representable, it is known (and not difficult to verify by hand) that the $\Psi_\F$ is fully faithful. 

\begin{definition}[Choric {\dpfs}]
In the above notation, $\F$ is said to be \emph{choric} if $\Psi_{\F}$ is fully faithful. 
\end{definition}

We do not known of any example of a non-choric \dpfs. In fact, for stable (and representable) derivators, any \dpfs is automatically choric. Furthermore, in the stable setting,  it is equivalent to specify a \dpfs and a ``compatible family'' of {\phfs}s: 
\color{black}
\begin{theorem}\label{stable_equv_orth_pre}
Suppose $\D$ is a stable derivator. Given $\class{X}\subseteq \D^\due(I)$ we denote $\overline{\class{X}}$ the isomorphism\hyp{}closure of the class $\dia_\due(\class{X})\subseteq \D(I)^\due$. The following are equivalent for a pair of sub pre-derivators $\F=(\mathbb E,\mathbb M)$  of $\D^{\due}$:
\begin{enumerate}
\item $\F$ is a \cpfs;
\item \color{black}$\F$ is a choric \cpfs;\color{black}
\item $\overline\F_I=(\overline\E_I,\overline\M_I)$ is a \phfs in $\D(I)$ for any $I\in\Dia$.
\end{enumerate}
\end{theorem}
\begin{proof}
The implication ``(2)$\Rightarrow$(1)'' is trivial while ``(1)$\Rightarrow$(2)'' is implied by Lemma \ref{level_wise_ff}. 
For the equivalence ``(3)$\Leftrightarrow$(1)'', we obtain by  Lemma \refbf{coherent_orth_is_wobbly}, that $(\overline\E_I)^{\horth}=\overline\M_I$ and ${}^{\horth}(\overline\M_I)=\overline\E_I$ if and only if $(\E_I)^{\corth}=\M_I$ and ${}^{\corth}(\M_I)=\E_I$. Then clearly (3) implies (1). For the converse, it is enough to show that $\overline\E_I$ is closed under suspensions, which is a consequence of Lemma \refbf{cfs_limit_colimit}.
\end{proof}

\subsection{Derivator factorization systems}

We are now going to give the definition of a derivator factorization system. Roughly speaking, this should be a \cpfs $\F$ such that any map is ``coherently $\F$-crumbled''; in the following definition we translate this idea in the language of derivators.

\begin{definition}[derivator factorization systems]\label{def_hfs}
Let $\F=(\mathbb E,\mathbb M)$ be a pair of sub pre-derivators of $\D^\due$ and let $\Psi_\F\colon \D_\F\longrightarrow \D^{\due}$ be the morphism defined in \eqref{composition_morphism}. We say that $\F$ is a \emph{derivator factorization system} (for short, \dfs) if it is a \cpfs and if $\Psi_\F(I)$ is essentially surjective for any $I\in \Dia$.
\end{definition}

Let us give an interpretation of the above definition in case $\D$ is representable:

\begin{lemma}\label{weak_and_cancellation}
In Setting \refbf{setting_sec_3}, let $\F=(\mathbb E,\mathbb M)$ be a pair of sub pre-derivators of $\D^\due$. If $\F$ is a \dfs, then $\overline \F_I \coloneqq (\overline\E_I,\overline \M_I)$ is a weak factorization system (see, for example, \cite[§\textbf{2}]{riehl2008factorization}) on $\D(I)$, for any $I\in \Dia$. Also, $\overline \F_I$ has the following cancellation properties:
\begin{enumerate}
\item given a composition $g\circ f\in \D(I)^\due$, if $g\circ f$ and $f\in \overline\E_I$, then $g\in \overline\E_I$;
\item given a composition $g\circ f\in \D(I)^\due$, if $g\circ f$ and $g\in \overline\M_I$, then $f\in \overline\M_I$.
\end{enumerate}
If $\D=y(\C)$ is representable, then $\F$ is a \dfs if and only if each $\F_I$ is an orthogonal factorization system.
\end{lemma}
\begin{proof}
By Lemma \refbf{corth=orth} it is clear that the two classes $\overline\E_I$ and $\overline \M_I$ are weakly orthogonal and, using the essential surjectivity of $\dia_\tre(I)$, it is not difficult to show that any morphism in $\D(I)$ is $\overline\F_I$-crumbled. To show that $(\overline\E_I,\overline \M_I)$ is a weak factorization system it enough to show that $\overline\E_I$ and $\overline \M_I$ are closed under retracts, which is an easy exercise. 

Let us now verify the cancellation properties (1) and (2). As in the proof of Lemma \refbf{factorization_triangle}, consider the unique possible natural transformations $\gamma\colon (0,1)\Rightarrow (1,2)$ and $\beta\colon (0,2)\Rightarrow (1,2)$. Let $Z\in \D^\tre(I)$ be an object such that 
\[
\dia_{\tre}Z\cong[ \cdot \xto{f}\cdot\xto{g}\cdot]\in \D(I)^\tre.
\] 
Suppose that $Z_{(0,2)}$ and $Z_{(0,1)}\in \mathbb E(I)$, and let $Y\in \mathbb M(I)$. Given a morphism $\phi\colon Z_{(1,2)}\to Y$, there exists a unique morphism $d\colon Z_2\to Y_0$ such that $\varphi_{Z_{(0,2)},Y}(d)=\phi\circ \beta^*_Z$. Then, $\varphi_{Z_{0,1}, Y}(d\circ g)=\phi\circ\gamma^*_Z=\varphi_{Z_{0,1}, Y}(\phi_0)$, so that $d\circ g=\phi_0$. This shows that $g=\dia_{\due}Z_{(1,2)}$ is weakly orthogonal to $\dia_\due Y$ for any $Y\in \M_I$, so $g\in \overline \E_I$. This proves (1), the proof of (2) is completely analogous. 

The last statement follows by Lemma \refbf{dpfs=pfs_if_discrete} and the fact that saying that $\Psi_{\F}(I)$ is essentially surjective is equivalent to say that any morphism in $\D(I)$ is $\overline \F_I$-crumbled when $\D$ is represented. 
\end{proof}

Before analyzing \textsc{dfs}s in the context of stable derivators, let us give the following reformulation of their definition:

\begin{lemma}\label{easier_def_dfs}
In Setting \refbf{setting_sec_3}, let $\F=(\mathbb E,\mathbb M)$ be a pair of sub-2-functors of $\D^\due$. Then, $\F$ is a \textsc{dfs} if and only if the following statements hold true for any $I\in \Dia$:
\begin{enumerate}
\item $\mathbb E(I)\corth \mathbb M(I)$ (that is $\mathbb E(I)\subseteq {}^{\corth} \mathbb M(I)$ or, equivalently, $\mathbb M(I)\subseteq  \mathbb E(I)^{\corth}$);
\item $\mathbb E(I)$ and $\mathbb M(I)$ are closed under isomorphisms in $\D^\due(I)$;
\item $\Psi_\F$ is essentially surjective.
\end{enumerate}
\end{lemma}
\begin{proof}
It is trivial to verify that if $\F$ is a \textsc{dfs} then it satisfies (1), (2) and (3). On the other hand, let $I\in \Dia$, suppose (1), (2) and (3) are satisfied and let us prove that $\mathbb E(I)\supseteq {}^{\corth} \mathbb M(I)$ (the proof that $\mathbb M(I)\supseteq  \mathbb E(I)^{\corth}$ is completely analogous). Indeed, let $X\in{}^{\corth} \mathbb M(I)$ and consider an object $\widetilde X\in \D_\F(I)$ such that $\Psi_\F(\widetilde X)\cong X$. By construction, $(0,1)^*\widetilde X\in \mathbb E(I)\subseteq {}^{\corth}\mathbb M(I)$ and $(0,2)^*\widetilde X\cong X\in {}^{\corth}\mathbb M(I)$; by the same argument used in the proof of the cancellation properties in Lemma \refbf{weak_and_cancellation}, one verifies that $(1,2)^*\widetilde X\in {}^{\corth}\mathbb M(I)$, but then $(1,2)^*\widetilde X\in {}^{\corth}\mathbb M(I)\cap \mathbb M(I)$, showing that $\dia_\due((1,2)^*\widetilde X)$ is an isomorphism. As a consequence, $\dia_\due(X)\cong \dia_\due((0,1)^*\widetilde X)$ and, using Setting \refbf{setting_sec_3}(1,2), this implies $X\cong (0,1)^*\widetilde X\in \mathbb E(I)$.
\end{proof}

We close the subsection showing that the bijections of \athm \ref{stable_equv_orth_pre} restrict to {\dfs}s:

\begin{theorem}\label{stable_equv_orth}
Suppose $\D$ is a stable derivator. The following are equivalent for a pair of sub pre-derivators $\F=(\mathbb E,\mathbb M)$  of $\D^{\due}$:
\begin{enumerate}
\item $\F$ is a \dfs;
\item \color{black}$\F$ is a choric \dfs;\color{black}
\item $\Psi_\F\colon \D_\F\to \D^{\due}$ is an equivalence;
\item $\overline\F_I=(\overline\E_I,\overline\M_I)$ is a \hfs in $\D(I)$ for any $I\in\Dia$.
\end{enumerate}
\end{theorem}
\begin{proof}
The equivalence ``(1)$\Leftrightarrow$(2)'' follows by \athm \ref{stable_equv_orth_pre}, while the equivalence ``(2)$\Leftrightarrow$(3)'' easily follows from the definitions. The implication ``(1)$\Rightarrow$(4)'' follows by Lemma \refbf{coherent_orth_is_wobbly} and the fact that $\dia_{I,\due}\colon \D^I(\due)\to \D(I)^{\due}$ is full and essentially surjective for any $I\in \Dia$. Finally, to prove the implication ``(4)$\Rightarrow$(1)'' notice that, by Lemma \refbf{coherent_orth_is_wobbly}, we know that $\F$ is a  \cpfs, let us show that each $\Psi_{\F}(I)$ is essentially surjective. For this, remember that the diagram functor
\[
\dia_{I,\tre}\colon \D(I\times \tre)\to \D(I)^{\tre}
\]
is full and essentially surjective. Let $X\in \D^I(\due)$ and choose two composable morphisms $\bar X_e\in \bar \E_I$ and $\bar X_m\in \bar \M_I$ such that $\dia_{I,\due}X=\bar X_m\circ \bar X_e$. We obtain a morphism $(f_0,f_1,f_2)\colon (X_0\to X_0\to X_1)\to (X_0\to (X_e(1)=X_m(0))\to X_1)$ in $\D(I)^{\tre}$ as in the following diagram:
\[
\xymatrix{
X_0\ar@{=}[d]\ar[r]^(.3){\id_{X_0}}&X_0= X_e(0)\ar[d]^{\bar X_e}\\
X_0\ar[r]^(.3){\bar X_e}\ar[d]_{\dia_{I,\due}X}&X_e(1)=X_m(0)\ar[d]^{\bar X_m}\\
X_1\ar[r]^(.3){\id_{X_1}}&X_1=X_m(1)
}
\]
Since $\dia_{I,\tre}$ is full and essentially surjective, we can lift the above diagram to a morphism $f\colon (1,2)_*X\to F$, where the underlying diagram of $F$ is exactly $(X_0\to (X_e(1)=X_m(0))\to X_1)$, so $F\in \D_\F(I)$. To conclude, one should just prove that $F_{(0,2)}\cong X$, but in fact $f_{(0,2)}\colon (0,2)^*(1,2)_*X(\cong X)\to F_{(0,2)}$ is an isomorphism. To see this just notice that $f_{(0,2)}$ is an isomorphism if and only if $0^*f_{(0,2)}$ and $1^*f_{(0,2)}$ are isomorphisms but, by construction, $0^*f_{(0,2)}=f_0=\id_{X_0}$ and $1^*f_{(0,2)}=f_2=\id_{X_2}$ are clearly isomorphisms.\qedhere
\end{proof}

\subsection{The relation with \textsc{fs} in $\infty$-categories} \label{infty_cat_fs}
Recall that an \emph{$\infty$-category} \cite{HTT} or \emph{quasi-category} \cite{joyal2008notes} is defined as a simplicial set $X$ in which every \emph{inner horn} $\Lambda^k[n] \to X$ has an extension $\Delta^n \to X$ along the inclusion $\Lambda^k[n]\to \Delta^n$. 

These liftings take care of the complicated ladder of coherence conditions for compositions in a $(\infty,1)$-category, as well as of the invertibility of all cells in dimension $k\ge 2$, making quasicategories into a flexible model to re-enact classical categorical constructions inside $X$ (co/limits, Kan extensions) and outside $X$ (adjoints, monads, monoidal structures). We refer to the sources \cite{HTT} or \cite{joyal2008notes} for a general background and the terminology and notation that we borrow.

Let's fix through this subsection an $\infty$-category $\iC$. In this subsection we are going to recall how to associate to $\iC$ a pre-derivator $\D_\iC$ and then, under suitable hypotheses, establish a natural bijection between the family of \textsc{fs}s in $\iC$ (in the sense of \cite{joyal2008notes}) and a family of {\dfs}s on $\D_\C$ called ``maximal {\dfs}s''. 
%
\begin{definition}[squares and fillers]
Any two edges $f,\, g\colon \Delta^{1}\to \iC$ can be considered as elements in $(\iC^{\Delta^{1}})_0$. By the definition of internal hom in a presheaf topos, an element $q\in \iC^{\Delta^{1}}(f,g)$ is a square $q\colon \Delta^{1}\times \Delta^{1} \to \iC$ such that $q\restriction_{\Delta^{\{0\}}\times \Delta^{1}}=f$ and $q\restriction_{\Delta^{\{1\}}\times \Delta^{1}}=g$. We define a \emph{filler} for $q\in \iC^{\Delta^{1}}(f,g)$ to be an element $s\in \iC(f_1,g_0)$ such that 
\begin{enumerate}
\item $q\restriction_{\Delta^{1}\times \Delta^{\{0\}}}$ is homotopic to $d\circ f$;
\item $q\restriction_{\Delta^{1}\times \Delta^{\{1\}}}$ is homotopic to $g\circ d$.
\end{enumerate}
That is, $s$ makes the two triangles in the following diagram commute up to homotopy:
\[
\xymatrix{
\cdot\ar[r]\ar[d]_f&\cdot\ar[d]^g\\
\cdot\ar[r]\ar@{.>}[ur]|{\ s\ }&\cdot
}
\]
\end{definition}
\begin{definition}[orthogonality]
We say that $f$ \emph{left orthogonal} to $g$ (while $g$ is \emph{right orthogonal} to $f$), in symbols $f\perp g$ if, for any $q\in \iC^{\Delta^{1}}(f,g)$, the space of fillers for $q$ is contractible. 
\end{definition}
\begin{definition}[orthogonality of a class]
Given a subclass $\mathcal{X}\subseteq \iC_1$ we let
\begin{gather*}
{}^{\perp}\mathcal{X}\coloneqq \{y\in\iC_1 \mid y\perp X\}\\
\mathcal{X}^{\perp}\coloneqq \{y\in\iC_1 \mid X\perp y\}.
\end{gather*}
A pair $\F=(\E,\M)$ of sub-classes of $\iC_1$ is a \textsc{pfs} provided ${}^{\perp}\M=\E$ and $\E^\perp=\M$. Furthermore, $\F$ is a \textsc{fs} if, for any $f\in \iC_1$, there exist $e\in \E$ and $m\in\M$ such that $f$ is homotopic to $m\circ e$. 
\end{definition}
Now that we have \textsc{fs}s defined in the setting of $\infty$-categories, let us describe their relation to {\dfs}s. For this we will need to associate to our $\infty$-category $\iC$ a derivator $\D_\iC$ of shape $\Dia$, for a suitable category of diagrams $\Dia$. Since we have taken $\iC$ to be finitely bicomplete, the canonical choice for $\Dia$ is the $2$-category of finite directed categories:
\begin{definition}
A category $I$ is a \emph{finite directed category} if it has a finite number of objects and morphisms, and if there are no directed cycles in the quiver whose vertices are the objects of $I$ and the arrows are the non-identity morphisms in $I$. Equivalently, the nerve of $I$  has a finite number of non-degenerate simplices.\\
We denote by $\fincat\subseteq \cat$ the $2$-category of diagrams spanned by the finite directed categories.
\end{definition}
\begin{proposition}
\cite[Remark \textbf{5.3.10}]{riehl2017kan}
Given an $\infty$-category $\iC$, the composition
$$\D_\iC\colon\xymatrix@R=0pt{
\cat^\opp\ar[r]^{N}&\cat_{\infty}^\opp\ar[r]^{\iC^{(-)}}&\Cat_{\infty}\ar[r]^{\ho}&\Cat\\
I\ar@{|->}[r]&N(I)\ar@{|->}[r]&\iC^{N(I)}\ar@{|->}[r]&\ho(\iC^{N(I)})}$$
is a pre-derivator. If $\Dia$ is a category of diagrams such that $\iC$ has all limits and colimits of shape $\Dia$, then $\D_{\iC}\restriction_{\Dia^\opp}$  is a strong derivator and it is pointed if and only if $\iC$ is pointed. If $\iC$ is stable then, by definition, $\iC$ has all limits and colimits of shape $\fincat$, so $\D_\iC\restriction_{\fincat^\opp}$ is a derivator which, moreover, is stable. 
\end{proposition}

In the rest of this subsection we will work in the following setting: we let $\iC$ be an $\infty$-category, and $\D_\iC\colon \fincat^\opp\to \Cat$ be the associated pre-derivator. Let $\F=(\E,\M)$ be a \textsc{fs} in $\iC$. By \cite[§\textbf{24.10}]{joyal2008notes}, for any $I\in \Dia$ we can define a \textsc{fs} $\F_I=(\E_I,\M_I)$ in $\iC^{N(I)}$, point-wise induced by $\F$. Letting $\mathbb E(I) \coloneqq \ho(\E_I)\subseteq \D^\due_\iC(I)$ and $\mathbb M(I) \coloneqq \ho(\M_I)\subseteq \D^\due_\iC(I)$, one can prove that $\mathbb E$ and $\mathbb M$ are sub-$2$-functors of $\D_\iC$ and that the pair $\F_{\D} \coloneqq (\mathbb E,\mathbb M)$ is a \dpfs on $\D_\iC$. In fact, the unique delicate part is to show that the orthogonality of 1-simplices in $\iC$ implies orthogonality of the corresponding objects in $\D_{\iC}(\due)$. One can do that by hand or using \cite[Lem. \textbf{5.2.8.22}]{HTT}. 
Thus we have constructed a map 
\begin{equation}\label{the_map}
\xymatrix@R=0pt{
\left\{\text{\textsc{fs} in $\iC$}\right\}\ar@{->}[rr]&&\left\{\text{\dfs in $\D_\iC$}\right\}\\
(\E,\M)\ar@{|->}[rr]&&(\E,\M)_\D.
}
\end{equation}
What we would like to understand is whether or not any \dfs in $\D_\iC$ arises from a \textsc{fs} in $\iC$. 
%
%
%
%
When $\iC=N(\A)$ is the nerve of a 1-category $\A$, $\D_\iC$ is the derivator represented by $\ho(\iC)\cong\A$, so both the \textsc{fs}s on $\iC$ and the  {\dfs}s on $\D_\iC$ correspond bijectively with the (classical) factorization systems on $\A$. 

For the rest of this subsection we concentrate our efforts to find a similar bijection when $\iC$ is stable. We start with the following lemma:

\begin{lemma}\label{lemma_infinity_orth}
Let $\iC$ be a stable $\infty$-category and let $\F=(\mathbb E,\mathbb M)$ be a \dfs on $\D_\iC$. Let $\mathbf{E}$ and $\mathbf M$ be the simplicial subsets of $\iC^{\Delta^1}$ spanned by $\mathbb E(\uno)$ and $\mathbb M(\uno)\subseteq \iC_1$, respectively. Given $X\colon X_0\to X_1$ in $\iC^{\Delta^1}$, consider the following squares in $\iC^{\Delta^1\times \Delta^1}$
\[
p:\qquad\begin{matrix}\xymatrix{
X_0\ar[r]^e\ar[d]_X&C\ar[d]^m\\
X_1\ar@{=}[r]&X_1
}\end{matrix}
\qquad
\text{and}
\qquad
q:\qquad\begin{matrix}\xymatrix{
X_0\ar@{=}[r]\ar[d]_e&X_0\ar[d]^X\\
C\ar[r]_{m}&X_1
}\end{matrix}
\]
If $e\in \mathbf E_0$ and $m\in \mathbf M_0$, then $p$ is an $\mathbf M$-localization and $q$ is an $\mathbf E$-co-localization of $X$ (in the sense of \cite[\adef\textbf{5.2.7.6}]{HTT}). In particular, $\mathbf M$ is reflective, $\mathbf E$ is coreflective in $\iC^{\Delta^1}$ and any reflection of $X$ in $\mathbf M$ (resp., any co-reflection of $X$ in $\mathbf E$) is equivalent to $p$ (resp., $q$).
\end{lemma}
\begin{proof}
By definition of $\mathbf E$-co-localization, we should verify that $q$ induces, for any $e'\in \mathbf E_0$, a weak homotopy equivalence
$\mathrm{Map}_\iC(e',e) \to \mathrm{Map}_\iC(e',X)$. According to Whitehead's theorem, we need to show that for every $k\leq 0$, the map
\[
\mathrm{Ext_{\iC^{\Delta^1}}^k}(e',e)=\D_\iC(\due)(\Sigma^{-k}e',e)\to \D_\iC(\due)(\Sigma^{-k}e',X)= \mathrm{Ext_{\iC^{\Delta^1}}^k}(e',X)
\] 
is an isomorphism of abelian groups. We prove first the case $k=0$. Indeed, fix a quasi-inverse $\Phi_\F\colon \D_\iC(\due)\to \D_\F(\uno)$ to the functor $\Psi_\F\colon \D_\F(\uno)\to \D_\iC(\due)$ (see \athm\refbf{stable_equv_orth}), then 
\begin{align*}
\D_\iC(\due)(e',X)&\cong \D_\iC(\tre)(\Phi_\F(e'),\Phi_\F(X))\\
&\cong \D_\iC(\tre)((0,1)_!e',\Phi_\F(X))\\
&\cong \D_\iC(\due)(e',(0,1)^*\Phi_\F(X))\\
&\cong \D_\iC(\due)(e',e).
\end{align*}
For $k\leq -1$, complete $q$ to a fiber sequence
\[
\begin{matrix}\xymatrix{
X_0\ar@{=}[r]\ar[d]_e&X_0\ar[d]^X\ar[r]&0\ar[d]\\
c\ar[r]_{m}&X_1\ar[r]&C(m)
}\end{matrix}
\]
and, passing to the associated long exact sequence, we are reduced to prove that, for any $k\leq -1$,
\[
0=\D_\iC(\due)\left(\Sigma^{-k}e',1_!C(m)\right)\cong \D_\iC(\uno)\left(\Sigma^{-k}C(e'),C(m)\right).
\]
Since $\Sigma^{-k}e', \Sigma^{-k+1}e'\in \mathbf E_0$ for any $k\leq -1$,  it is enough to prove that 
\[
\D_\iC(\uno)(C(e''),C(m))=0
\]
for any $e''$ such that $e''$ and $\Sigma e''\in \mathbf E_0$. But this follows by the two conditions $e''\horth m$ and $\Sigma e''\horth m$ in the triangulated category $\D_\iC(\uno)$ (they imply that the map $\D_\iC(\uno)(C(e''),C(m))\to \D_\iC(\uno)(e_1,\Sigma m_0)$ is both injective and trivial).
\end{proof}

Using the above lemma, we can now mimic part of the proof of \cite[\aprop \textbf{5.2.8.17}]{HTT} to verify that a \dfs on $\D_\iC$ induces a \textsc{fs} on the original $\infty$-category $\iC$. 

\begin{theorem}\label{infinity_FS_vs_CFS}
Let $\iC$ be an $\infty$-category which is either stable or the nerve of a 1-category. Denote by $\D_\iC\colon \fincat^\opp\to \Cat$ the induced pre-derivator. Then there is a bijective correspondence
\[
\xymatrix{
\left\{\text{\textsc{fs}s in $\iC$}\right\}\ar@{<->}[rr]&&\left\{\text{{\dfs}s in $\D_\iC$}\right\}.
}
\]
\end{theorem}
\begin{proof}
We have already mentioned that, if $\iC$ is the nerve of a 1-category $\A$, then both the \textsc{fs}s on $\iC$ and the {\dfs}s on $\D_\iC$ correspond bijectively with the classical factorization systems on $\A$. Hence, we concentrate on the case when $\iC$ is stable, so that $\D_\iC$ is a stable derivator. Given a \dfs $\F=(\mathbb E,\mathbb M)$ on $\D_\iC$ and letting $(\E,\M)$ be the two classes of 1-simplices in $\iC$ corresponding respectively to $\mathbb E(\uno)$ and $\mathbb M(\uno)$, we have to show that $(\E,\M)$ is a \textsc{fs} in $\iC$. Our strategy will be the following: we faithfully repeat the argument of the proof of the implication (1)$\Rightarrow$(2) in \cite[\aprop \textbf{5.2.8.17}]{HTT} to show that the restriction map 
\[
p\colon \Fun_{\E|\M}(\Delta^2,\iC)\to \Fun(\Delta^{\{0,2\}},\iC)
\] 
is a trivial Kan fibration, where $\Fun_{\E|\M}(\Delta^2,\iC)$ denotes the full subcategory of $\Fun(\Delta^2,\iC)$ spanned by those diagrams 
corresponding to a composition of an element in $\E$ followed by an element in $\M$. By \cite[\emph{ibi}]{HTT}, the fact that $p$ is a trivial Kan fibration is equivalent to say that $(\E,\M)$ is a \textsc{fs}, thus concluding the proof.

Let us start observing that $p$ is a categorical fibration, so it suffices to verify that it is a  categorical equivalence. We do this in two steps. First we let $\sD$ be the full subcategory of $\Fun(\Delta^1\times \Delta^1,\iC)$ spanned by those diagrams of the form
\[\xymatrix{
X\ar[r]^e\ar[d]_f&Y\ar[d]^m\\
Z\ar[r]^g&Z'
}\]
with $e\in \E$, $m\in \M$ and $g$ an equivalence. The map $p$ factors as a composition
\[
\xymatrix{
\Fun_{\E|\M}(\Delta^2,\iC)\ar[r]^(.7){p'}&\sD\ar[r]^(.3){p''}& \Fun(\Delta^{1},\iC)
}
\] 
where $p'$ carries a diagram
\[
\xymatrix{
X\ar[rr]^f\ar[rd]_e&&Z\\
&Y\ar[ur]_m}
\]
to the partially degenerate square
\[\xymatrix{
X\ar[dr]|f\ar[r]^e\ar[d]_f&Y\ar[d]^m\\
Z\ar@{=}[r]&Z
}\]
and $p''$ is given by restriction to the left vertical edge of the diagram. To complete the proof, it will suffice to show that $p'$ and $p''$ are categorical equivalences. For $p'$, this is a general fact proved in \cite[\emph{ibi}]{HTT} (which does not depend on the properties of the pair $(\E,\M)$) so it remains only to show that $p''$ is a trivial Kan fibration.

Let $\T$ denote the full subcategory of $\Fun(\Delta^1,\iC)\times \Delta^1$ spanned by those pairs $(m, i)$ where either $i = 0$ or $m\in \M$. 

The projection map $r \colon \T\to \Delta^1$ is the cartesian fibration associated to the inclusion 
$\Fun_{|\M}(\Delta^1,\iC) \subseteq \Fun(\Delta^1,\iC)$, where $\Fun_{|\M}(\Delta^1,\iC)$ is the full subcategory spanned by the elements of $\M$.
Using Lemma \refbf{lemma_infinity_orth}, we conclude that $r$ is also a co-Cartesian fibration. Moreover, we can identify
\[\sD \subseteq \Fun(\Delta^1\times \Delta^1, \iC) \cong \mathrm{Map}_{\Delta^1} (\Delta^1, \T)\]
with the full subcategory spanned by the co-Cartesian sections of $r$. In terms of this identification, $p''$ is given by evaluation at the initial vertex $\{0\}\subseteq \Delta^1$ and is therefore a trivial Kan fibration, as desired.

The fact that this correspondence really gives a bijection is a consequence of Corollary \ref{everything_is_maximal}.
\end{proof}

\color{black}
\section{The Rosetta stone for derivators}
We have seen in Theorem \refbf{triang-rosetta} that, given a triangulated category $\cD$, there is a bijection between $t$-structures and normal {\htth}s on $\cD$. In this section we are going to prove a similar bijection for derivators. More precisely, we fix a stable derivator $\D$ of type $\Dia$, and we show in \athm\refbf{derrosetta} that there is a bijection between $t$-structures on the triangulated category $\D(\uno)$ and \emph{normal derivator torsion theories} on $\D$. As a consequence we recover one of the main results of \cite{Fiorenza2014} (see Corollary \refbf{infinity_rosetta}). 
\subsection{Lifting {\hfs}s}
\begin{notat}\label{notation_rosetta}
Given a finite directed category $I$, we define its \emph{length} $\ell(I)\in \N$ as the maximal length of a path of non-identity arrows in $I$. An object $i$ in $I$ is \emph{minimal}, if there is no non-identity morphism starting in $i$. We denote by $\fincat$ the full sub-2-category of $\Cat$ whose $0$-cells are the finite directed categories. If $I\in\fincat$ we denote $I^\circ$ the subcategory of $I$ spanned by the non-minimal objects. Note that for any minimal object $i\in I$ of a finite directed category, if $u\colon I^\circ\hookrightarrow I$ denotes the inclusion, we have the inequalities
\[
\ell(I)>\ell(J)\geq \ell(u/i).
\] 
Finally, we denote $\partial I \coloneqq I\setminus I^\circ$ (\ie the subcategory spanned by all minimal objects).
\end{notat}
\begin{lemma}{\rm \cite{SSV}}\label{lemma_approximation_SSV}
Let $I$ be a category of finite length and let $u\colon I^\circ\hookrightarrow I$. Furthermore, given a stable derivator $\D\colon \Dia^\opp\to \Cat$ and $X\in \D(I)$, there is a distinguished, pointwise split triangle
\[
\bigoplus_{i\in\partial I}i_!i^*u_!u^*X\to \bigoplus_{i\in\partial I}i_!X_i\oplus u_!u^*X\to X\to \Sigma \bigoplus_{i\in\partial I}i_!i^*u_!u^*X
\]
induced by the counit maps $\epsilon_i : i_!i^*\to 1$ and $\epsilon_u : u_!u^* \to 1$.
\end{lemma}
In what follows, we start with a \hfs on the base $\D(\uno)$ of a stable derivator $\D$ and we want to show that this lifts point-wise to a \hfs on $\D(I)$ for all $I\in \fincat$. In the following lemma we start proving that the liftings give homotopy orthogonal classes.
\begin{lemma}\label{lifting_orthogonality}
Let $\D\colon \fincat^{op}\to \CAT$ be a stable derivator and let $\bar\F=(\bar\E,\bar\M)$ be a \phfs on the triangulated category $\D(\uno)$.  Let $\F=(\E,\M)$ be the pair of full subcategories of $\D(\due)$ of those objects whose underlying diagrams are, respectively, in $\bar \E$ and $\bar \M$. 
Given $I\in \fincat$, define
\[
\E_I \coloneqq \{X\in \D^{\due}(I):X_i\in \E, \, \forall i\in I\} \qquad \M_I \coloneqq \{Y\in \D^{\due}(I):Y_i\in \M, \, \forall i\in I\}.
\] 
Then, $\bar\E_I\horth\bar\M_I$ in $\D(I)$.
\end{lemma}
\begin{proof}
We proceed by induction on $\ell(I)$. The case $\ell(I)=0$ being trivial (in that case $\D(I)$ is a finite product of copies of $\D(\uno)$), let us suppose $\ell (I)\geq 1$ and that our statement is verified for any finite directed category of shorter length. 
Let $X\in \E_I$ and $Y\in \M_I$. 
By Lemma \refbf{lemma_approximation_SSV} we have a morphism of triangles as follows:
\[
\xymatrix@C=15pt{
\bigoplus_{i\in\partial I}i_!i^*u_!u^*X\ar[r]\ar[d]& \bigoplus_{i\in\partial I}i_!X_i\oplus u_!u^*X\ar[r]\ar[d]& X\ar[rr]^(.8)+\ar[d]&&\\
\pt^*\pt_!\bigoplus_{i\in\partial I}i_!i^*u_!u^*X\ar[r]& \pt^*\pt_!\left(\bigoplus_{i\in\partial I}i_!X_i\oplus u_!u^*X\right)\ar[r]& \pt^*\pt_!X\ar[rr]^(.8)+&&
}
\]
Applying the contravariant functor $\D^{\due}(I)(-,Y)$ to the above diagram we obtain
{\footnotesize \[
\xymatrix{
\underset{i\in\partial I}\prod \D(\due)((u_!u^*X)_i, Y_i) & \underset{i\in\partial I}\prod\D(\uno)(\pt_!(u_!u^*X)_i, \pt_*Y_i)\ar[l]_\cong\\
\ar[u]\underset{i\in\partial I}\prod\D(\due)(X_i, Y_i)\times \D^\due(\partial I)(u^*X, u^*Y) & \underset{i\in\partial I}\prod\D(\due)((X_i)_1, (Y_i)_0)\times \D^\due(\partial I)((u^*X)_1, (u^*Y)_0)\ar[u]\ar[l]_(.55)\cong\\
\ar[u]\D^\due(I)(X,Y) & \D^\due(I)(X_1, Y_0)\ar[u]\ar[l]\\
\ar[u]\underset{i\in\partial I}\prod\D(\due)((\Sigma u_!u^*X)_i, Y_i) & \underset{i\in\partial I}\prod\D(\uno)(\pt_!(\Sigma u_!u^*X)_i, \pt_* Y_i)\ar[u]\ar[l]_\cong\\
\vdots \ar[u]\ar@{<-}[r]^{\cong}& \vdots\ar[u] 
}
\]}\noindent
(we used that left Kan extensions commute with left Kan extensions and that functors of the form $v^*$ commute with both left and right Kan extensions). The isomorphisms in the above diagram come from our inductive hypothesis and so, by the 5-lemma, we can conclude that $\D^{\due}(I)(X_1,Y_0)\cong \D^{\due}(I)(X,Y)$, which means that $X\corth Y$.
\end{proof}

It remains to verify that, for a given $I\in\fincat$, the pointwise \phfs we found in the above lemma is also a \hfs. In the following lemma we reformulate this requirement in a way that will be easier to verify via Lemma \refbf{lemma_approximation_SSV}:
\begin{lemma}\label{fs_are_corefs}
For all $I\in \Dia$, there is a bijection between the following classes:
\begin{enumerate}
\item {\hfs}s in the triangulated category $\D(\uno)$;
\item co-reflections $S\dashv R: \E\leftrightarrows\D(\due)$ with counit $\rho\colon SR\to \id_{\D(\due)}$, such that $\rho_0\colon 0^*SR\to 0^*$ is a natural isomorphism.
\end{enumerate}
Given a {\hfs} $\F=(\E,\M)$ on $\D(\uno)$, let $S\dashv R$ be the associated co-reflection. For a morphism $\varphi\colon SE\to X$ in $\D(\due)$, with $E\in \widetilde\E$,  the map
\[
\D(\due)(SE',\varphi)\colon \D(\due)(SE',X)\to \D(\due)(SE',SE)\cong\E(E',E) 
\]
is an isomorphism if and only if $\varphi_0$ is an iso and $\varphi_1\in \M$.
\end{lemma}
\begin{proof}
Given a \hfs $\F=( \E, \M)$ in $\D(\uno)$, we have shown in Lemma \refbf{lifting_orthogonality} that $\E{\corth}\M$. Hence, the functor $\Psi\colon \D_{\F}(\uno)\to \D(\due)$ is an equivalence (see Lemmas \refbf{coherent_orth_is_wobbly} and \refbf{level_wise_ff}), so we can choose a quasi-inverse $\Phi\colon \D(\due)\to \D_\F(\uno)$. The desired co-reflection is constructed as the following composition:
\[
R_\F:=(0,1)^*\Phi\colon \D(\due)\to \D_\F(\uno)\to \E.
\]
This is clearly a right adjoint to the inclusion $S_\F\colon \E\to \D(\due)$, and it is not difficult to verify that the pair $(S_\F,R_\F)$ has the desired properties. 

On the other hand, given a reflection $R: \D(\due)\rightleftarrows\E :S$ as in the statement, we define 
\begin{align*}
\E_{(S,R)}\subseteq \D(\due), \quad \E_{(S,R)}&\coloneqq \{E\in \D(\due) \mid \rho_E \text{ is invertible}\}\\
\M_{(S,R)}\subseteq \D(\due), \quad \M_{(S,R)}&\coloneqq \{E\in \D(\due) \mid \dia_\due(RM) \text{ is invertible}\}
\end{align*}
We shall prove that $\F_{(S,R)}:=(\bar\E_{(S,R)},\bar\M_{(S,R)})$ is a \hfs in $\D(\uno)$. 
Let $E\in \E_{(S,R)}$ and $M\in \M_{(S,R)}$, then
\[
\D(\due)(E,M)\cong \D(\due)(E,SRM)\overset{(*)}{\cong} \D(\due)(E,1_*M_0)\cong \D(\uno)(E_1,M_0)
\]
where the isomorphism marked by $(*)$ is true since $RM$ is an iso, so $SRM\cong 1_*(SRM)_0$ and $(\rho_M)_0\cong(SRM)_0\to M_0$ is an iso. By the above isomorphisms one deduces that $E\corth M$, so that $\E_{(S,R)}\corth \M_{(S,R)}$, which is equivalent to say that $\bar \E_{(S,R)}\horth \bar \M_{(S,R)}$. These classes are also closed under taking isomorphisms, so it is enough to show that any morphism in $\D(\uno)$ is  $\F_{(S,R)}$-crumbled. Indeed, let $X\in \D(\due)$ and consider $\rho_X\colon SRX\to X$. Since $(\rho_X)_0$ is an iso, $\dia_\due(X)\cong (\rho_X)_1\circ \dia_\due(SRX)$ and notice that $(\rho_X)_1\in \bar\M_{(S,R)}$ while $\dia_\due(SRX)\in \bar\E_{(S,R)}$.

For the last part of the statement, consider a morphism $\rho\colon E\to X$ with $E\in \E$ and $X\in \D(\due)$, such that $\rho_0$ is an iso and $\rho_1\in\bar\M$. Then, for any given $E'\in\E$
\begin{align*}
\D(\due)(E',X)&\cong \D(\tre)(\Phi E',\Phi X)\\
&\cong \D(\tre)((0,1)_!E',\Phi X)\\
&\cong \D(\due)(E',(0,1)^*\Phi X)\cong \E(E',E).
\end{align*}
where the second isomorphism is true since $\Phi(E')\cong (0,1)_!E'$, that is, the factorization of a coherent morphism $E'$ in $\E$ is given by $E'$ followed by an isomorphism. Furthermore, the last isomorphism is true since our hypotheses on $\rho$ imply that $\dia_\due(X)\cong \rho_1\dia_\due(E)$, where $\rho_1\in \bar \M$, so this factorization is the $\F$-factorization of $\dia_\due(X)$.
\end{proof}

\begin{theorem}\label{everything_is_maximal}
Let $\D\colon \fincat^\opp\to \Cat$ be a stable derivator and let $\bar\F=(\bar\E,\bar\M)$ be {\hfs} on $\D(\uno)$. Then there is a pointwise induced \dfs $\F=(\mathbb E,\mathbb M)$ on $\D$ such that $(\bar\E,\bar\M)=(\bar{\mathbb E}(\uno),\bar{\mathbb M}(\uno))$.
\end{theorem}
\begin{proof}
Let $\E$ and $\M$ be the full subcategories of $\D(\due)$ of those objects whose underlying diagrams are, respectively, in $\bar \E$ and $\bar \M$. For any $I\in \fincat$, let
\begin{align*}
\E_I & \coloneqq \{E\in\D^\due(I):E_i\in\E,\, \forall i\in I\}\\ 
\M_I & \coloneqq \{M\in\D^\due(I):M_i\in\M,\, \forall i\in I\}.
\end{align*}
We want to show that $\F:=(\mathbb E,\mathbb M)$, where $\mathbb E(I):=\E_I$ and $\mathbb M(I):=\M_I$, is a {\dfs}.
By Lemmas \refbf{lifting_orthogonality} and \refbf{level_wise_ff}, the functor
\[
\Psi_\F\colon \D_\F\to \D^\due
\] 
is fully faithful and so, by \athm \refbf{stable_equv_orth}, it is enough to verify that $\Psi_\F$ is essentially surjective. By Lemma \refbf{fs_are_corefs}, we know that there is a co-reflection
\[
S: \E\leftrightarrows \D(\due): R
\]
where $S$ is the inclusion, and the co-unit $\rho_X\colon RX\to X$ has the property that $(\rho_X)_0$ is an isomorphism for all $X$.  It is enough to verify that, for any $I\in \fincat$, the inclusion $S_I\colon \E_I\to \D^\due(I)$ is co-reflective, with co-reflection $R_I\colon \D^\due(I)\to \E_I$, and that the co-unit $\rho^I\colon R^I\to \id_{\D^\due(I)}$ is such that $(\rho^I)_0$ is an isomorphism. 
We proceed by induction on $\ell(I)$. If $\ell(I)=1$, then $I$ is a disjoint union of copies of $ \uno$ and there is nothing to prove.

By the inductive assumption, and given the inequalities noted in \refbf{notation_rosetta} there is an adjunction $S_J: \E_J\leftrightarrows \D^\due(J):R_J$, with co-unit $\rho^J\colon R_J\to \id_{\D^\due(J)}$ such that $1^*\rho^J$ is an iso. Given an object $X\in \D(I)$, let us construct a coreflection $\rho\colon E\to X$. We start considering the following triangle, constructed in Lemma \refbf{lemma_approximation_SSV}:
\[
\bigoplus_{i\in\partial I}i_!i^*u_!u^*{X}\to \left(\bigoplus_{i\in\partial I}i_!{X}_i\right)\oplus u_!u^*{X}\to {X}\to \Sigma \bigoplus_{i\in\partial I}i_!i^*u_!u^*{X}.
\]
For any minimal object $i\in I$, we consider the following commutative squares:
\[
\xymatrix@C=18pt{
i_!i^*u_!R_J(u^*X)\ar@{}[drr]|{(\mathrm{a})}\ar[d]\ar[rr]^{\varepsilon_i}&&u_!R_J(u^*X)\ar[d]&&i_!i^*u_!R_J(u^*X)\ar@{}[drr]|{(\mathrm{b})}\ar[d]\ar[rr]&&i_!R(X_i)\ar[d]\\
i_!i^*u_!u^*X\ar[rr]_{\varepsilon_i}&&u_!u^*X&&i_!i^*u_!u^*X\ar[rr]&&i_!i^*X
}
\]
where these squares are constructed as follows:
\begin{enumerate}
\item[(a)] the first square is the easiest: we start with the co-unit $\rho^J_{u^*X}\colon R_J(u^*X)\to u^*X$, then the left column is $i_!i^*u_!(\rho^J_{u^*X})$, while the right column is $u_!(\rho^J_{u^*X})$. The horizontal maps are the appropriate components of the co-unit $\varepsilon_i$ of the adjunction $(i_!,i^*)$. Hence, the square commutes by naturality of $\varepsilon_i$;
\item[(b)] as for the second square, we construct first a commutative square
\[
\xymatrix{
i^*u_!R_J(u^*X)\ar[d]_{i^*u_!(\rho^J_{u^*X})}\ar[rr]&&R(X_i)\ar[d]^{\rho_{X_i}}\\
i^*u_!u^*X\ar[rr]_{i^*(\varepsilon_u)_{X}}&&X_i
}
\]
and then apply $i_!$. To construct this square, take $\rho_{X_i}\colon R(X_i)\to X_i$ as the vertical map on the right. The horizontal map at the base of the square is $i^*(\varepsilon_u)_{X}$, where $\varepsilon_u$ is the co-unit of the adjunction $(u_!,u^*)$. The vertical map on the left is $i^*u_!(\rho^J_{u^*X})$, as in the first square. Notice that $i^*u_!R_J(u^*X)\in \E$, in fact, $i^*u_!R_J(u^*X)\cong\hocolim_{u/i}\mathrm{pr}_i^* R_J(u^*X)$ and this belongs in $\E$ since the left class of a \hfs is always closed under taking homotopy colimits (this is a consequence of our Lemma \refbf{extension} and \cite[Theorem 7.1]{Ponto-Schulman}). To conclude, notice that, by adjointness, there is a unique morphism $i^*u_!R_J(u^*X)\to R(X_i)$ that makes the above square commutative. 
\end{enumerate}
Putting together the above squares, with $i$ varying in $i\in\partial I$, we get a commutative square as on the left of
\[
\xymatrix{
\underset{i\in\partial I}{\bigoplus}i_!i^*u_!R_J(u^*X)\ar[rr]\ar[d]&&\left(\underset{i\in\partial I}{\bigoplus}i_!R(X_i)\right)\oplus u_!R_J(u^*X)\ar[d] \ar[r]& E\ar@{.>}[d]\ar[r]^(.7){+}&\\ 
\underset{i\in\partial I}{\bigoplus}i_!i^*u_!u^*X\ar[rr]&&\left(\underset{i\in\partial I}{\bigoplus}i_!X_i\right)\oplus u_!u^*X \ar[r]& X\ar[r]^(.7){+}&.
}
\]
Of course the rows are triangles and, letting $E$ be the cone of the first row, we obtain a morphism $\rho\colon E\to X$. To conclude we have to show that $E\in \E_I$ and that $\rho$ is the co-reflection of $X$.  In fact, it is easy to show that the objects in the first row 
belong to $\E_I$, so $E$, which is the cone of a map between these two objects, still belongs to $\E_I$. Furthermore, to show that $\rho\colon E\to X$ it is enough to show that $\rho_1\in\bar\M_I$ and $\rho_0$ is an iso. Consider the following diagram where all the rows and columns are triangles and everything commutes:
\[
\xymatrix@R=15pt@C=10pt{
\underset{i\in\partial I}{\bigoplus}i_!i^*u_!R_J(u^*X)\ar[rr]\ar[d]&&\left(\underset{i\in\partial I}{\bigoplus}i_!R(X_i)\right)\oplus u_!R_J(u^*X)\ar[d]\ar[rr]&&E\ar[rr]^(.7){+}\ar[d]&&\\ 
\underset{i\in\partial I}{\bigoplus}i_!i^*u_!u^*X\ar[rr]\ar[d]&&\left(\underset{i\in\partial I}{\bigoplus}i_!X_i\right)\oplus u_!u^*X\ar[d]\ar[rr]&&X\ar[rr]^(.7){+}\ar[d]&&\\
A\ar[d]^(.7){+}\ar[rr]&&B\ar[rr]\ar[d]^(.7){+}&&M\ar[rr]^(.7){+}\ar[d]^(.7){+}&&\\
&&&&&}
\]
We will have concluded if we can prove that $E\in\E_I$ and that $M\cong 1_*M_1\in \M_I$, that is, $k^*E\in \E$ and $k^*M\cong 1_*(k^*M)_1\in \M$ for any $k\in I$. We start from the case when $k\in I^\circ$, and apply $k^*$  to the above $3\times 3$ diagram, obtaining the following commutative diagram in $\D(\uno)$:
\[
\xymatrix{
0\ar[rr]\ar[d]&&0\oplus k^*u_!R_J(u^*X)\ar[d]\ar[rr]&&E_k\ar[d]\ar[r]^(.65)+&\\ 
0\ar[rr]\ar[d]&&0\oplus X_k\ar[rr]\ar[d]&&X_k\ar[d]\ar[r]^(.65)+&\\
A_k\ar[d]^(.65)+\ar[rr]&&B_k\ar[d]^(.65)+\ar[rr]&&M_k\ar[d]^(.65)+\ar[r]^(.65)+&\\
&&&&&
}
\]
%
%
Using the fact that $k^*$ sends triangles to triangles, we get that $A_k=0$ and that $k^*u_!R_J(u^*X)\to X_k$ is the co-reflection of $X_k$ onto $\E$. As a consequence $M_k\in \M$ and $(M_k)_0=0$.

On the other hand, if $k$ is a minimal object, applying $k^*$ to the above $3\times 3$ diagram we get the following commutative diagram in $\D(\uno)$, where all the rows and columns are distinguished triangles:
\[
\xymatrix{
k^*u_!R_J(u^*X)\ar[rr]\ar[d]&&R(X_k)\oplus k^*u_!R_J(u^*X)\ar[d]\ar[rr]&&R(X_k)\ar[d]\ar[r]^(.65)+&\\ 
k^*u_!u^*X\ar[rr]\ar[d]&&X_k\oplus k^*u_!u^*X\ar[rr]\ar[d]&&X_k\ar[d]\ar[r]^(.65)+&\\
A_k\ar[d]^(.65)+\ar[rr]&&B_k\ar[d]^(.65)+\ar[rr]&&M_k\ar[d]^(.65)+\ar[r]^(.65)+&\\
&&&&&
}
\]
The maps $k^*u_!R_J(u^*X)\to k^*u_!R_J(u^*X)$ and $k^*u_!u^*X\to k^*u_!u^*X$ in the first two rows are clearly isomorphisms by the construction of square (a).
In particular, the first two rows are split triangles and the first arrow in the triangle 
\[
R(X_k)\to X_k\to M_k\to \Sigma R(X_k)
\] 
is the cokernel of the first maps in the first two columns. Since $R(X_k)\to X_k$ is the co-reflection of $X_k$ in $\E$ by the construction of square~(b), we get that $M_k\in \M$  and $(M_k)_0\cong 0$ as desired.
\end{proof}

As a consequence of the above theorem we can deduce that a \dfs on a stable derivator of type $\fincat$ is completely determined by the \hfs it induces on the base:

\begin{corollary}\label{everything_is_maximal}
Let $\D\colon \fincat^\opp\to \Cat$ be a stable derivator and let $\F=(\mathbb E,\mathbb M)$ be a \dfs on $\D$. Given $I\in\fincat$, an object $X\in \D^\due(I)$ belongs to $\mathbb E(I)$ (resp., $\mathbb M(I)$) if and only if $X_i\in \mathbb E(\uno)$ (resp., $\mathbb M(\uno)$), for all $i\in I$. 
\end{corollary}
\begin{proof}
Let $\bar\E:=\bar {\mathbb E}(\uno)$, $\bar\M:=\bar {\mathbb M}(\uno)$, and $\bar \F=(\bar \E,\bar \M)$. Then, $\bar\F$ is a \hfs on $\D(\uno)$ and, by the above theorem there is a second \dfs $\F'=(\mathbb E',\mathbb M')$, where $\mathbb E'(I)$ and $\mathbb M'(I)\subseteq \D^\due(I)$ are the full subcategories of those objects that are pointwise in $\mathbb E(\uno)$ and $\mathbb M(\uno)$, respectively. But now $\mathbb E\subseteq \mathbb E'$ and $\mathbb M\subseteq \mathbb M'$, and these two inclusions imply that $\F=\F'$.
\end{proof}

\color{black}

\subsection{The Rosetta stone theorem}\label{higher_rosetta_subs}
To prove that $t$-structures on $\D(\uno)$ correspond bijectively to normal derivator torsion theories on $\D$,  we should say what it means for a \dfs $\F=(\mathbb E,\mathbb M)$ on $\D$ to be a normal derivator torsion theory. One easy way to say this is to ask that, for any $I\in \fincat$, the {\hfs} $\overline\F_I=(\overline{\mathbb E(I)},\overline{\mathbb M(I)})$ is a normal \htth. In the following definition we give a different (but equivalent) formulation that better fits into the language of derivators. Notice that the following definition makes sense in any pointed derivator $\D$, not just for stable ones.

\begin{definition}[normal derivator torsion theories]
A sub pre-derivator $\mathbb X$ of $\D$ is said to have the $3$-for-$2$ property if, for any $I\in \fincat$, given $X\in \D^\tre(I)$ such that $2$ objects in the set $\{X_{(0,1)},X_{0,2},X_{1,2}\}$ belong to $\mathbb X^{\due}(I)$, so does the third. A \dfs $\F=(\mathbb E,\mathbb M)$ on $\D$ is said to be
\begin{enumerate}
\item a \emph{derivator torsion theory} (for short, \textsc{dtth}) provided $\mathbb E$ and $\mathbb M$ have the $3$-for-$2$ property;
\item \emph{left normal} if, given $I\in \fincat$, $X\in \D(I)$ and $F\in \D^\tre({I})$ such that $F_{(0,2)}\cong 0_*X$, then $0_*0^!F_{(0,1)}\in \mathbb E(I)$;
\item \emph{right normal} if, given $I\in \fincat$, $X\in \D(I)$ and $F\in \D^\tre({I})$ such that $F_{(0,2)}\cong 1_!X$, then $1_!1^?F_{(1,2)}\in \mathbb E(I)$;
\item \emph{normal} if it is both left and right normal.
\end{enumerate}
\end{definition}

Using the stability of our derivator $\D$, it is not difficult to verify that a \dfs is left normal if and only if it is right normal, if and only if it is normal.

\begin{theorem}\label{derrosetta}
Let $\D\colon \fincat^\opp\to \Cat$ be a stable derivator. There is a bijection between the following classes:
\begin{enumerate}
\item $t$-structures in $\D(\uno)$; 
\item normal \textsc{dtth}s on $\D$.
\end{enumerate}
\end{theorem}
\begin{proof}
By Theorem \refbf{triang-rosetta} there is a bijective correspondence between $t$-structures and normal {\htth}s in $\D(\uno)$. Using Corollary \refbf{everything_is_maximal} it is not difficult to show that normal {\htth}s in $\D(\uno)$ correspond bijectively to normal \textsc{dtth}s on $\D$.
\end{proof}

Let now $\iC$ be a stable $\infty$-category. We have seen in Theorem \refbf{infinity_FS_vs_CFS} that \textsc{fs}s on $\iC$ correspond bijectively with maximal {\dfs}s on the associated derivator $\D_\iC\colon \fincat^\opp\to \Cat$ and, by Corollary \refbf{everything_is_maximal}, any \dfs on $\D_\iC$ belongs to this family. Furthermore, it is not difficult to verify that a \textsc{fs} on $\iC$ is a normal torsion theory (see \cite{Fiorenza2014} for the exact definition) if and only if the associated \dfs on $\D_\iC$ is a normal \textsc{dtth}. As a consequence we obtain the following corollary:

\begin{corollary}{\rm \cite{Fiorenza2014}}\label{infinity_rosetta}
Let $\iC$ be a stable $\infty$-category. There is a bijection between the following classes:
\begin{enumerate}
\item $t$-structures in $\ho(\iC)$; 
\item normal torsion theories in $\iC$.
\end{enumerate}
\end{corollary}

\section{Functoriality of factorizations }\label{sec:thmA}
\setlength{\epigraphwidth}{.2\textwidth}
\epigraph{
	\begin{CJK}{UTF8}{min}道生一，\end{CJK}\\
	\begin{CJK}{UTF8}{min}一生二，\end{CJK}\\
	\begin{CJK}{UTF8}{min}二生三，\end{CJK}\\
	\begin{CJK}{UTF8}{min}三生万物。\end{CJK}
}{Laozi 42}

Let $\Dia$ be a $2$-category of diagrams (see the beginning of §\refbf{sec:squaring}) and denote by $\PDer$ the 2-category of pre-derivators of type $\Dia$. Recall from \cite[\textbf{2.1}.(ii)]{Moritz} that there is a strict 2-functor, called \emph{shift functor}, defined as
\[
\textsf{sh}(-,-)\colon \Dia^\opp\times \Der\longrightarrow \Der\colon(J,\D)\mapsto \D^{J}.
\]
Here $\D^J$ is the pre-derivator such that $\D^J(I) \coloneqq \D(J\times I)$. Given a functor $u\colon I\to J$ in $\Dia$ the action of $\D^J$ is described by the formula $\D^K(u) \coloneqq \D(\id_K\times u)$ for all $K\in\Dia$, and similar formulas hold for the action on $\D^J$ on natural transformations. 

It is convenient for us to introduce the following notation: \cite{Moritz} blurs the distinction between the functor $u^*\colon \D(J)\to \D(I)$ image of $u\colon I\to J$ under a derivator $\D$ and the pseudo\hyp{}natural transformation $u^\circledast\colon \D^{J}\to \D^{I}$ induced by the same $u$. However, this clash of notation can be harmful to our discussion, since we will mainly consider instances of the second map, while needing a reference to its action on morphisms of $\Dia$.
\begin{notat}
Given a pre-derivator $\D$ and a functor $u\colon I\to J$ in $\Dia$, we let 
\[
u^\circledast  \coloneqq \textsf{sh}(u,\D)\colon \D^{J}\to \D^{I}.
\]
If $\D$ is a derivator, then $u^\circledast$ has a left and a right adjoint as a 1-cell in the 2-category $\PDer$, that we denote by
\[
u_\circledbang\dashv u^\circledast\dashv u_\circledast.
\]
Given $K\in \Dia$, the components $u_\circledbang(K),u_\circledast(K)\colon \D^{J}(K)\to \D^{I}(K)$ are given by
$u_\circledbang(K)=(u\times \id_K)_!$ and $u_\circledast(K)=(u\times \id_K)_*$.
\end{notat}
\subsection{Factorization pre-algebras}\label{subs_facprealg}
Following \cite{RW}, given a category $\C$, a functor $F\colon \C^\due\to \C$ such that $F(\id_c)=c$ for each $c\in \C$, and not only a coherent isomorphism $F(\id_c)\cong c$, is said to be a \emph{normal} factorization pre\hyp{}algebra. In this subsection we introduce a similar notion in the context of pre-derivators and we describe some of its elementary properties.
\begin{remark}
We explicitly remark that there is no connection between the normality of a factorization pre\hyp{}algebra and the normality of a (homotopy) torsion theory defined in \refbf{hontt}; the coincidence of the two terms is only an unfortunate clash of terminology of the two sources from which we are extracting our main theorems.
\end{remark}
\begin{definition}[normal factorization pre\hyp{}algebra]\label{def_factorization_prealgebra}
A morphism $F\colon \D^\due\to \D$ in $\PDer$
is said to be a \emph{factorization pre\hyp{}algebra} provided  there exists an isomorphism $\gamma\colon F\circ \pt^\circledast\to \id_{\D}$. A factorization pre\hyp{}algebra is \emph{normal} provided $F\circ \pt^\circledast=\id_{\D}$. 
\end{definition}
The reason to call these functors ``pre-algebras'' will be clarified in Section \refbf{sec:thmB}: factorization pre\hyp{}algebras are just algebras over the squaring monad deprived of their extended associator. On the other hand, the use of the term ``factorization'' is justified by the validity of the following lemma in the context of pre-derivators: we recall the adjunctions $1^\circledast \adjunct{\eta}{} \pt^\circledast  \adjunct{}{\epsilon}0^\circledast$ and define
\begin{lemma}
Let $\D$ be a pre-derivator and $F\colon \D^\due\to \D$ a normal factorization pre\hyp{}algebra. Then $F$ induces a factorization
\[
\xymatrix{
0^\circledast\ar[r]^{e_{\firstblank}}&F\ar[r]^{m_{\firstblank}}&1^\circledast
}
\]
 of the 2-cell $\kappa \colon 0^\circledast \to 1^\circledast\colon \D^\due \to \D$ introduced in \refbf{la-kappa}, where $e,m$ are obtained whiskering $F$ with the unit and counit above:
 \begin{align}
 e_{\firstblank} &\colon 0^\circledast=F \circ \pt^\circledast\circ  0^\circledast \xto{F * \epsilon} F\notag\\
 m_{\firstblank} &\colon F\xto{F * \eta} F \circ \pt^\circledast\circ 1^\circledast=1^\circledast.
 \end{align}
Conversely, for each pair of natural transformations $\ee \colon 0^\circledast \to F$ and $\mm\colon F \to 1^\circledast$ that factor $\kappa$ via a 1-cell $F$ and such that $\ee\circ \pt^\circledast \cong \id$, $\mm \circ \pt^\circledast\cong \id$, one has $\ee = F * \epsilon$ and $\mm = F * \eta$.
\end{lemma}
\begin{proof}
By the very definition of $\kappa$, the whiskering $\pt^\circledast * \kappa$ coincides with the counit\hyp{}unit composition $\pt^\circledast \circ 0^\circledast \xto{\epsilon} \id \xto{\eta} \pt^\circledast \circ 1^\circledast$. Thus, $m_{\firstblank}\circ e_{\firstblank}=(F * \eta)\circ (F * \epsilon)=F\circ \pt^\circledast*\kappa=\kappa$. The last statement is a simple formal consequence of the zig\hyp{}zag identities for the adjunctions $1^\circledast \adjunct{\eta}{} \pt^\circledast  \adjunct{}{\epsilon}0^\circledast$.
\end{proof}
\begin{remark}
Spelled out more explicitly, the above lemma shows that a normal factorization pre\hyp{}algebra $F$ functorially associates to a given $X\in \D^\due(I)$ a factorization 
\begin{equation}\label{func_fact}
\xymatrix{X_0 \ar@/^18pt/[rr] \ar[r]^{e_{X}}&FX\ar[r]^{m_{X}}&X_1}
\end{equation}
of the underlying diagram $X_0\to X_1$ of $X$. 
\end{remark}
In the following discussion we are going to show that such factorization can be made ``coherent'' via a morphism
\[
\Phi_F\colon \D^\due\longrightarrow \D^\tre
\]
such that $(0,2)^\circledast \circ \Phi_F=\id_{\D^\due}$, $1^\circledast \circ \Phi_F =F$ and $\dia_{\tre}(\Phi_FX)$ is the diagram in (\refbf{func_fact}).
\begin{notat}
Before giving the construction of $\Phi_F$, let us introduce a few more notation, this time regarding the category of functors $[\due\times\due,\due]$, that comes equipped with a commutative square of natural transformations as in the left diagram below (see also \cite[§\textbf{2.4}]{RW}), whose components are depicted on the right.
\[
\label{diamond_def}
\vcenter{
	\xymatrix{r \ar@{=>}[r]^\mu\ar@{=>}[d]^{\mu'}& h  \ar@{=>}[d]^\nu \\ v\ar@{=>}[r]_{\nu'} & l}
}\qquad\qquad
\vcenter{\xymatrix{
r(i,j)= i\land j \ar[d]^{\mu'_{ij}}\ar[r]^{\mu_{ij}} &  h(i,j)=j \ar[d]^{\nu_{ij}}\\
v(i,j)=i \ar[r]_{\nu'_{ij}} & l(i,j)=i\lor j 
}}
\]
where $\mu$, $\mu'$, $\nu$ and $\nu'$ obviously represent the two chains $i\land j \le i\le i\lor j$ and $i\land j \le j\le i\lor j$. Consider also the ``slit'' functor $d\colon \tre\times \due\to \due$ defined as follows:
\[
\begin{kodi}[xscale=2,yscale=1.2]
\foreach \x in {0,1,2} {
\foreach \y in {0,1}
  \node at (\y,-\x) (\x\y) {$(\x,\y)$};
}
\node at (2,-1) (label) {$d$};
\draw[densely dotted] (00.north west) -- 
					  (10.south west) -- 
					  (10.south east) -- 
					  (00.south east) -- 
					  (01.south east) -- 
					  (01.north east) -- cycle;
\draw[densely dotted] (11.north west) -- 
					  (11.north east) -- 
					  (21.south east) -- 
					  (20.south west) -- 
					  (20.north west) -- 
					  (21.north west) -- cycle;
\node at (3,0) (zero) {$0$};
\node at (3,-2) (uno) {$1$};
\draw[->] (01) to[bend left] (zero);
\draw[->] (21) to[bend right] (uno);
\mor 00 -> 01 -> 11 -> 21;
\mor * -> 10 -> 20 -> *;
\mor 10 -> 11;
\mor zero -> uno;
\end{kodi}
\]
\end{notat}
Having established this notation, we gather in the following statement some elementary facts about the above objects and arrows:
\begin{lemma}\label{properties_of_functors_among_ord}
In the above notation,
\begin{enumerate}[label=($\roman*$)]
\item there are adjunctions $l\adjunct{\varphi}{} \Delta\adjunct{}{\psi} r\colon \due\to \due\times\due$ 
(the counit of $l\dashv \Delta$ and the unit of $\Delta\dashv r$  are identities);
\item $v=\id_{\due}\times\pt\colon \due\times \due\to \due\times \uno=\due$;
\item $h=\pt\times\id_\due\colon \due\times \due\to \uno\times \due=\due$;
\item The identities $v\circ\Delta = h\circ\Delta =\id_\due$ hold;
\item $d\circ((0,1)\times \id_\due)=r$ and $d\circ((1,2)\times \id_\due)=l$;
\item $d\circ ((0,2)\times \id_{\due})=\id_{\due}\times\pt$;
\item $d\circ (1\times \id_{\due})=\id_{\due}$.
\end{enumerate}
\end{lemma}
\begin{definition}[coherent factorization and its pieces]\label{func_fact_def}
Given a pre-derivator $\D\colon \Dia^{op}\to \Cat$ and a normal factorization pre\hyp{}algebra $F\colon \D^\due\longrightarrow \D$, we define the following morphisms of derivators:
\begin{itemize}
\item $\Phi_F \coloneqq  F^\tre \circ d^\circledast\colon \D^\due\to \D^\tre$;
\item $F_l \coloneqq  F^\due \circ l^\circledast\colon \D^\due\to \D^\due$;
\item $F_r \coloneqq  F^\due \circ r^\circledast\colon \D^\due\to \D^\due$,
\end{itemize}
and the following natural transformations:
\begin{itemize}
\item $m_{e_{\firstblank}}\colon F\circ F_r=F\circ F^\due \circ r^\circledast\xto{FF^\due * \mu^\circledast} F\circ F^\due\circ  v^\circledast=F$;
\item $e_{m_{\firstblank}}\colon F=F\circ F^\due\circ  v^\circledast\xto{FF^\due * \nu^\circledast} F\circ F_l=F\circ F^\due\circ l^\circledast$.
\end{itemize}
\end{definition}
We are now ready to prove the announced properties of $\Phi_F$:
\begin{lemma}\label{prealgebra_induces_factorization}
In the above notation, the following statements hold true:
\begin{enumerate}
\item $(1,2)^\circledast \circ \Phi_F\cong F_l$ and $(0,1)^\circledast \circ\Phi_F\cong F_r$;
\item $(0,2)^\circledast \circ \Phi_F\cong \id_{\D^\due}$;
\item $1^\circledast \circ \Phi_F \cong F$.
\end{enumerate}
\end{lemma}
\begin{proof}
Since $\Phi_F$ is a morphism in $\PDer$, $(1,2)^\circledast\circ \Phi_F \cong F_l$ as a consequence of the chain of isomorphisms
\begin{align*}
(1,2)^\circledast\circ \Phi_F & =(1,2)^\circledast\circ F^\tre d^\circledast\\
& \cong F^\due \circ ((1,2)\times \id_\due)^\circledast \circ d^\circledast\\
&=F^\due \circ (d \circ((1,2)\times \id_\due))^\circledast\\
&=F^\due \circ l^\circledast=F_l
\end{align*}
where we used Lemma \refbf{properties_of_functors_among_ord}.($v$). This proves the first half of (1), the second half is completely analogous. Also, parts (2) and (3) follow similarly, using part ($vi$) and ($vii$) of \refbf{properties_of_functors_among_ord}, respectively.
\end{proof}
\begin{remark}
It is easy, though not needed in the following discussion, to define factorizations 
\[
\xymatrix{
	0^\circledast \ar[dr]\ar[rr]^e&& F\ar[rr]^m \ar[dr]|{FF^\due* \nu^\circledast}&& 1^\circledast \\
	& FF^\due r^\circledast \ar[ur]|{FF^\due * \mu^\circledast}&& FF^\due l^\circledast\ar[ur]
}
\]
and to show that these two triangles are, respectively, $\Phi_F(F_rX)$ and $\Phi_F(F_lX)$.
\end{remark}
We conclude the discussion with the following remark that shows how working with factorization pre\hyp{}algebras which are normal is not restrictive (this is completely analogous to \cite[§\textbf{2.2}]{Korostenski199357}):
\begin{remark}[normalization lemma]\label{normal_not_restrictive}
Given a factorization pre\hyp{}algebra $F\colon \D^\due\to \D$ with a fixed isomorphism $\gamma\colon F\circ \pt^\circledast \to \id_\D$ we can find another morphism $F'\colon \D^\due\to \D$ such that $F'\circ \pt^\circledast =\id_\D$ and $F'\cong F$. Indeed, given $I\in \Dia$, one defines $F'_I\colon \D^\due(I)\to \D(I)$ as follows: for an object $X\in \D^\due(I)$
\[
F'_I(X) \coloneqq \begin{cases}
Y&\text{if $X=\pt^*(Y)$;}\\
F_I(X)&\text{otherwise;}
\end{cases}
\]
while for a morphism $\phi\colon X\to X'$ in $\D^\due(I)$, 
\[
F'_I(\phi) \coloneqq \delta_{X'}\circ F_I(\phi)\circ \delta_{X}^{-1}\qquad\text{where}\qquad
\delta_X \coloneqq \begin{cases}
\gamma_Y&\text{if $X=\pt^*(Y)$;}\\
\id_{F_I(X)}&\text{otherwise.}
\end{cases}
\]
\end{remark}
\subsection{Eilenberg-Moore factorizations}\label{EM_subs}
In this subsection we are going to prove that, under very mild assumptions, a normal factorization pre\hyp{}algebra $F\colon \D^\due\to \D$ induces a \dfs $\F=(\mathbb E_F,\mathbb M_F)$ such that the functor $\Phi_F$ of \adef\refbf{func_fact_def} provides an inverse to the functor $\Psi_\F$ of \adef\refbf{def_hfs}.

Let us start defining the pre-derivators $\mathbb E_F$ and $\mathbb M_F$:

\begin{definition} \label{def_EF_MF}
In the same setting and with the same notations of \adef\refbf{func_fact_def}, we define two sub pre-derivators $\mathbb E_F$ and $\mathbb M_F\subseteq \D^{\due}$ where, for any $I\in \Dia$, 
\begin{align}
\mathbb{E}_F(I)&=\{X\in \D^{\due}(I)\mid F_{l}X\text{ is an iso}\}\notag\\ 
\mathbb{M}_F(I)&=\{Y\in \D^{\due}(I)\mid F_{r}Y\text{ is an iso}\}.
\end{align}
\end{definition}
What allows us to prove that the pair $(\mathbb E_F,\mathbb M_F)$ is a \dfs is the rephrasing of the Eilenberg\hyp{}Moore condition.
\begin{definition}[eilenberg\hyp{}moore factorization]\label{def_EM_algebra}
A normal factorization pre\hyp{}algebra $F\colon \D^\due\to \D$ is said to be a \emph{Eilenberg\hyp{}Moore} (\emph{\textsc{em}}, for short) \emph{factorization} provided $F_r(X)\in \mathbb E_F(I)$ and  $F_l(X)\in \mathbb M_F(I)$, for any $I\in\Dia$ and $X\in \D^\due(I)$.
\end{definition}
We can now prove our awaited Theorem \textbf{III}:
\begin{proof}[Proof of Theorem \textbf{III}]
By Lemma \refbf{prealgebra_induces_factorization}, $\Phi_F$ takes values in $\D_\F\subseteq \D^\tre$ and $\Psi_\F\Phi_F=(0,2)^\circledast\Phi_F\cong \id_{\D^\due}$. This shows that $\Psi_\F$ is essentially surjective and full. Consider now $X\in \mathbb E_FI$ and $Y\in \mathbb M_FI$ and let us show that the map 
\[
\varphi_{X,Y}\colon\D^I(\uno)(X_1,Y_0)\longrightarrow \D^I(\due)(X,Y)
\] 
is an isomorphism. Indeed, $\Phi_F X\cong (0,1)_!X$ and $\Phi_F Y\cong (1,2)_*Y$, so that
\begin{align*}
\D(\uno)(X_1,Y_0) &\cong \D(\due)(X,1_*Y_0)\\
& \cong \D(\tre)((0,1)_!X,(1,2)_*Y)\\
& \cong \D(\tre)(\Phi_F X,\Phi_F Y)\\
& \twoheadrightarrow\D(\due)(\Psi_\F\Phi_F X,\Psi_\F\Phi_F Y)\\
& \cong \D(\due)(X,Y)
\end{align*}
showing that $\varphi_{X,Y}$ is surjective; it remains to show injectivity. Indeed, consider two morphisms $a,\, b\colon X_1\to Y_0$, such that 
$\varphi_{X,Y}(a)=\varphi_{X,Y}(b)$. This means that $\psi_Y\pt_\due^*a\varphi_X=\psi_Y\pt_\due^*b\varphi_X$ and so, in particular,
\[
F(\psi_Y)F(\pt^*a)F(\varphi_X)=F(\psi_Y)F(\pt^*b)F(\varphi_X).
\]
Now, $\psi_Y=\dia_{\due}(l^*Y)$ so $F(\psi_Y)=\dia_{\due}(F^{\due}l^*Y)=\dia_{\due}(F_lY)$ is an iso and, similarly, $F(\varphi_X)$ is an iso. 
Hence, we obtain that $a= F(\pt^*a)=F(\pt^* b)=b$. This proves conditions (1) and (3) of Lemma \refbf{easier_def_dfs}, while condition (2) easily follows by construction, thus $\F$ is a \textsc{dfs}.

On the other hand, Let $\F'$ be a \dfs and suppose that $\D$ is represented or that it is a stable derivator. In both settings we known that $\Psi_{\F'}\colon \D_{\F'}\to \D^{\due}$ is an equivalence. Fix a quasi-inverse $\Phi_{\F'}\colon \D^{\due}\to \D_{\F'}$ to $\Psi_{\F'}$ and let $F' \coloneqq 1^\circledast\circ\Phi_{\F'}$. One can show that $F'$ is an \textsc{em} factorization and that $\F'=(\mathbb E_{F'},\mathbb M_{F'})$.
\end{proof}

\section{Coherence of factorization algebras}
\label{sec:thmB}

This last section of the paper is devoted to introduce all the background needed to discuss, precisely state and, finally, prove \athm \textbf{IV} of the Introduction. First of all, in §\refbf{monads_subs}, we recall from \cite[§\textbf{1}]{lack2002codescent} the relevant definitions of 2-monads and pseudo-algebras. After that, in §\refbf{squaring_subs}, we specialize these general definitions to the so-called squaring monad on $\PDer$. \athm\@\textbf{IV} is then proved at the end of §\refbf{higher_coherence_sub}.
 
\subsection{2-monads}\label{monads_subs}
One of the most annoying features of higher dimensional monad theory is in how many place the coherence conditions can hide: the category $\cate{K}$ where the monad is defined, the monad $T$ itself, the naturality for multiplication and unit, and their associativity and unitality constraints, as well as the compatibility conditions for a $T$\hyp{}algebra, can all give rise to some diagrams that commute only up to a (invertible or non\hyp{}invertible) 2-cell.

Of course, some of these combinations of laxity are quite uncommon: 2-dimensional monad theory often copes with \emph{strict} 2-monads, or with strong monads that can be suitably ``strictified''. According to the existing zoology, here we need \emph{lax algebras for a strict pseudo-monad on a strict 2-category}. However, having no interest in different flavours, we simply call it the category of ``algebras for a 2-monad $T$''. We start with the definition of 2-monad, from \cite{lack2002codescent}.
\begin{definition}[2-monad]\label{two-monad}
Let $\cate{K}$ be a strict 2-category. A \emph{2-monad} on $\cate{K}$ consists of a tuple $\TT=(T,\mu,\eta,\ass, \uni)$ where $T$ is a strict endofunctor $T\colon\cate{K}\to\cate{K}$ endowed with a pair $(\mu,\eta)$ of 2-cells $\mu\colon T\circ T \Rightarrow T$, $\eta \colon \id_{\cate{K}}\Rightarrow T$ subject to the following relations:
\begin{itemize}
\item[(\textsc{mn})] the components of $\mu$ and $\eta$ fit into pseudo-commutative diagrams
\[
\vcenter{\xymatrix{
T^2K\ar[r]^{\mu_K}\ar@{}[dr]|{\Swarrow\me_f}\ar[d]_{T^2f} & TK\ar[d]^{Tf} & K\ar[r]^{\eta_K}\ar@{}[dr]|{\Swarrow\yu_f}\ar[d]_f & TK\ar[d]^{Tf}\\
T^2K' \ar[r]_{\mu_{K'}} & TK' & K' \ar[r]_{\eta_{K'}}& TK'
}}
\]
for 2-cells $\yu_f$ and $\me_f$ subject to the obvious conditions with respect to composition and identity 1-cells
 (these are nothing more than pseudo\hyp{}naturality conditions that can be be applied to any 2-cell).
\item[(\textsc{ma})] $\mu$ is associative, in that the diagram
\[
\xymatrix@R=1.5cm@C=1.5cm{
T^3K \ar[r]^{\mu_{TK}}\ar@{}[dr]|{\Swarrow\ass_K}\ar[d]_{T\mu_K} & T^2K \ar[d]^{\mu_K} \\
T^2K \ar[r]_{\mu_K}& TK
}
\]
commutes when filled by an invertible 2-cell $\ass_K \colon \mu_K\circ (\mu * T)_K \Rightarrow \mu_K \circ (T * \mu)_K$, which can be regarded as the $K$-component of an invertible 3-cell $\ass \colon \mu\circ (\mu *T) \Rrightarrow \mu\circ (T *\mu)$.
\item[(\textsc{mu})] $\eta$ is unital, in that the diagram
\[
\xymatrix@R=1.5cm@C=1.5cm{
TK \ar[r]^{\eta_{TK}}\ar@{-}[dr] \ar[d]_{T\eta_K}& T^2K\ar[d]^{\mu_K} \\
T^2K \ar[r]_{\mu_K} \ar@{}[ur]|(.35){\Swarrow\uni_{\textsc{l},K}}\ar@{}[ur]|(.65){\uni_{\textsc{r},K}\Nearrow}& TK
}
\]
commutes when filled with invertible 2-cells $\uni_{\textsc{r},K}\colon \id_{TK}\Rightarrow \mu_K\circ (\eta *T )_K$ and $\uni_{\textsc{l},K}\colon \id_{TK}\Rightarrow \mu_K\circ ( T * \eta )_K$, which can be regarded as the $K$-components of invertible 3-cells $\uni_\textsc{l} \colon \id_{T}\Rrightarrow \mu\circ ( T * \eta )$ and $\uni_\textsc{r} \colon \id_{T}\Rrightarrow \mu\circ ( \eta *T )$.
\end{itemize}
\end{definition} 
\begin{definition}[pseudo-algebras for a 2-monad]\label{two-algebras}
Let $\TT=(T,\mu,\eta,\ass,\uni)$ be a 2-monad on $\cate{K}$. A \emph{2-algebra} for $\TT$, or a $\TT$-algebra for short, consists of a tuple $\underline{A}=(a, \alpha_m, \alpha_u)$ where $a\colon TA\to A$ is a 1-cell of $\cate{K}$, and $\alpha_m,\alpha_u$ are invertible 2-cells called respectively the \emph{extended associator} and the \emph{normalizer} of the algebra structure, such that the following diagrams of 2-cells commute:
\begin{gather}
\vcenter{\xymatrix@C=.4cm{
&T^2A \ar[rr]^{Ta} \ar@{}[dd]|{\Downarrow\me_a}\ar[dr]_{\mu_A}&& TA\ar[dr]^a \ar@{}[dl]|{\Swarrow\alpha_m}\\
T^3A \ar[ur]^{T^2a}\ar[dr]_{\mu_{TA}}&& TA \ar@{}[dr]|{\Searrow\alpha_m} \ar[rr]_a&& A\\
& T^2A \ar[ur]^{Ta}\ar[rr]_{\mu_A}&& TA\ar[ur]_a
}}
\quad
{\Huge =}
\quad 
\vcenter{\xymatrix@C=.3cm{
&T^2A \ar@{}[dr]|{T\alpha_m\Searrow}\ar[rr]^{Ta}&& TA\ar@{}[dd]|{\Downarrow\alpha_m}\ar[dr]^{a} \\
T^3A \ar[rr]^{T\mu_A}\ar[ur]^{T^2a}\ar[dr]_{\mu_{TA}}&& T^2A\ar@{}[dl]|{\ass_A\Swarrow}\ar[ur]^{Ta}\ar[dr]_{\mu_A} && A\\
& T^2A \ar[rr]_{\mu_A}&& TA\ar[ur]_a
}}\notag\\[5mm]
\vcenter{\xymatrix@C=.4cm{
&A \ar@{}[dd]|{\Downarrow\yu_m}\ar@{-}@/^1.5pc/[drrr]\ar[dr]_{\eta_A}&&\\
TA \ar[ur]^a\ar[dr]_{\eta_{TA}}&& TA\ar@{}[ur]|{\Swarrow\alpha_u}\ar@{}[dr]|{\Searrow\alpha_m} \ar[rr]_a&& A\\
& T^2A \ar[ur]^{Ta}\ar[rr]_{\mu_A}&& TA\ar[ur]_a
}}
\quad
{\Huge =}
\quad 
\vcenter{\xymatrix@C=.4cm{
&A \ar@{}[rrrdd]|{=}\ar@{-}@/^1.5pc/[drrr]&&\\
TA  \ar@{-}@/^1.5pc/[drrr] \ar[dr]_{\eta_{TA}}\ar[ur]^a&& && A\\
&T^2A \ar[rr]_{\mu_A} \ar@{}[ur]|{\Swarrow\uni_{\textsc{r},A}}&& TA\ar[ur]_a &
}}\notag\\[5mm]
\vcenter{\xymatrix@C=.3cm{
&& & TA\ar[dr]^a\ar@{}[dd]|{\Downarrow\alpha_m}\\
TA\ar[rr]^{T\eta_A} \ar@{-}@/^1.5pc/[urrr]\ar@{-}@/_1.5pc/[drrr]&& T^2A \ar[ur]^{Ta}\ar[dr]_{\mu_A}
\ar@{}[ul]|{\Searrow T\alpha_u}\ar@{}[dl]|{\Nearrow \uni_{\textsc{l},A}}&& A\\
&& & TA\ar[ur]_a
}}
\quad
{\Huge =}
\quad 
\vcenter{\xymatrix@C=.4cm{
&& & TA\ar[dr]^a\\
TA \ar@{}[rrrr]|{||} \ar@/^1.5pc/[urrr]\ar@/_1.5pc/[drrr] && && A\\
&& & TA\ar[ur]_a
}}
\end{gather}
\end{definition}
\begin{remark}
We often stick to denote a pseudo-algebra for a monad $\TT$ simply as a \emph{$T$-algebra}; we also call \emph{normal} a $T$-algebra for which the normalizer $\alpha_u$ is the identity map (so $a\circ \eta_A = \id_A$ and the coherence diagrams above obviously simplify). We will be mainly interested in normal $T$-algebras; this is not restrictive, as shown in \refbf{normal_not_restrictive}.
\end{remark}
\subsection{The squaring monad and its algebras}\label{squaring_subs}
\begin{definition}[the squaring monad on $\PDer$]\label{def_squaring_monad}
Let $\Dia$ be a $2$-category of diagrams and denote by $\PDer$ the 2-category of pre-derivators of type $\Dia$. The \emph{squaring monad} on $\PDer$ is the triple $((-)^\due, \Delta^\circledast,\pt^\circledast)$, where $(-)^\due \coloneqq \textsf{sh}(\due,-)$, while $\Delta$ and $\pt$ are defined in Section \refbf{comonoid_due_subs}. 
\end{definition}
Using the shift functor $\textsf{sh}(-,-)$ to transport the comonoid structure on $\due$ described in §\refbf{comonoid_due_subs}, one can see that the squaring monad is a 2-monad in the sense of \cite{lack2002codescent} but in a very strict sense, in that we have equalities (and not mere natural isomorphisms) in the following expressions
\[
\begin{cases}
\Delta^\circledast\circ(\pt\times \id_{\due})^\circledast=\id_{\D^\due}=\Delta^\circledast\circ(\id_{\due}\times \pt)^\circledast\notag\\
\Delta^\circledast\circ(\Delta\times\id_{\due})^\circledast=\Delta^\circledast\circ(\id_{\due}\times \Delta)^\circledast.
\end{cases}
\]
The strictness of the squaring monad can be used to greatly simplify the definition of pseudo-algebras given in \cite{lack2002codescent}:
\begin{definition}\label{def_factorization_alg}
In the same setting of \adef\refbf{def_squaring_monad}, let $\D$ be a pre-derivator of type $\Dia$. A \emph{normal pseudo-algebra} for the squaring monad is a morphism $F\colon \D^\due\to \D$ such that $F\circ \pt^\circledast=\id_\D$, together with a natural isomorphism $\gamma\colon FF^\due\xto{\cong} F\Delta^\circledast$ that satisfies the following properties:
\begin{enumerate}
\item $\gamma *(\id_\due\times\pt)^\circledast=\id_F$;
\item $\gamma*(\pt\times \id_\due)^\circledast=\id_F$;
\item  $(\gamma*(\Delta\times\id_\due)^\circledast) \circ (\gamma* F^{\due\times \due})=(\gamma*(\id_\due\times\Delta)^\circledast) \circ  (F*\gamma^\due)$.
\end{enumerate}
We will refer to a normal pseudo-algebra over the squaring monad simply as a \emph{normal $(\firstblank)^\due$\hyp{}algebra}.
\end{definition}
\subsection{Coherence for factorization algebras}\label{higher_coherence_sub}
First of all, we are going to re-enact some technical results of \cite{RW}, in preparation for the proof of \athm\textbf{IV}; these are simply the result of having adapted the most relevant results in \cite[§\textbf{2}]{RW} from the 2-category $\Cat$ to the 2-category $\PDer$. 

\begin{lemma}\label{a_straightforward_lemma}\label{nuff_to_determine}
Let $F\colon \D^\due\to\D$ be a normal factorization pre\hyp{}algebra; then precomposition with $l^{\circledast}$ induces a bijection
\[
\PDer(\D^{\due\times\due}, \D)(FF^\due, F\Delta^\circledast) \xto{\quad \firstblank * l^{\circledast}\quad} \PDer(\D^\due, \D)(FF^\due l^{\circledast}, F).
\]
\end{lemma}
\begin{proof}
We start noticing that, in any 2-category, given a diagram of 2-cells
\[
\xymatrix@C=2cm{
	A \ar@{}[r]|{\Downarrow\sigma}\ar@/^1pc/[r]^F\ar@/_1pc/[r]_G& B \ar@{}[r]|{\Downarrow\tau}\ar@/^1pc/[r]^S\ar@/_1pc/[r]_T & C
}
\]
if $T * \sigma$ is invertible, then $\tau  *F$ is determined by $\tau *G$, in the sense that $\tau * F = (T * \sigma)^{-1}\circ (\tau *G)\circ (S* \sigma)$. Similarly, if $S * \sigma$ is invertible then $\tau *G = (T *\sigma)\circ (\tau *F )\circ (S * \sigma)^{-1}$ (specialized to the 2-category $\Cat$, this is \cite[Lemma 2.2]{RW}).

For any $\tau\colon FF^\due\to F\Delta^\circledast$ there is a diagram of 2-cells in $\PDer$:
\[
\xymatrix@C=2cm{
	\D^{\due\times\due} \ar@{}[r]|{\Downarrow\eta_l^\circledast}\ar@/^1pc/[r]^{\id_{\D^{\due\times \due}}}\ar@/_1pc/[r]_{l^\circledast\Delta^{\circledast}}& \D^{\due\times\due} \ar@{}[r]|{\Downarrow\tau}\ar@/^1pc/[r]^{FF^\due}\ar@/_1pc/[r]_{F\Delta^{\circledast}} & \D
}
\]
Using the above general fact, one can easily prove that there is a bijection 
\[
\PDer(\D^{\due\times\due}, \D)(FF^\due, F\Delta^\circledast) \xto{\quad \firstblank * l^{\circledast}\Delta^{\circledast}\quad} \PDer(\D^{\due\times \due}, \D)(FF^\due l^{\circledast}\Delta^{\circledast}, F\Delta^{\circledast}),
\]
which induces the desired bijection since $\Delta^{\circledast}$ is co-fully faithful (for more details see \cite[§\textbf{2}]{RW}).
\end{proof}

\begin{lemma}
\label{iso_implies_EM}
Let $F\colon \D^\due\to \D$ be  a normal factorization pre\hyp{}algebra and suppose that there is an isomorphism $\alpha\colon FF^\due\to F\Delta^\circledast$. Then all the 2-cells $m_{e_{\firstblank}}$, $e_{m_{\firstblank}}$, $FF^\due*(\mu')^\circledast$ and $FF^\due*(\nu')^\circledast$ are invertible. As a consequence, $\F \coloneqq (\mathbb E_F,\mathbb M_F)$ is an Eilenberg\hyp{}Moore factorization system.
\end{lemma}
\begin{proof}
Consider the following commutative diagram, obtained applying $\alpha$ to the first diagram in \refbf{diamond_def}
\begin{equation}
\label{two_diamonds}
\xymatrix{
FF_rX\ar[dr]|{\alpha * r^\circledast}\ar[rr]^{m_{e_X}}\ar[dd]_{FF^\due * (\mu')^\circledast} && FX\ar[dd]^{e_{m_X}} \ar[dl]|{\alpha * v^\circledast}\\
& FX & \\
FX \ar[rr]_{FF^\due *(\nu')^\circledast}\ar[ur]|{\alpha * h^\circledast}&& FF_l X\ar[ul]|{\alpha * l^\circledast}
}
\end{equation}
where the central object results as a square of identities $\id_{FX}$ obtained from \refbf{properties_of_functors_among_ord}.($iv$). It is then clear that, being $\alpha$ invertible, the four arrows that point to $FX$ are all invertible, and so are the components of $m_{e_{\firstblank}}$, $e_{m_{\firstblank}}$, and $FF^\due*(\mu')^\circledast, FF^\due*(\nu')^\circledast$. 

For the last statement we should verify that $F_rX\in \mathbb E_F$ and $F_lX\in \mathbb M_F$, for any $X\in \D^\due(I)$, equivalently, one should verify that $F_lF_rX$ and $F_rF_lX$ are isomorphisms. But this is clear since the underlying diagram of $F_lF_rX$ is exactly $m_{e_X}$, while the underlying diagram of $F_rF_lX$ is $e_{m_X}$. 
\end{proof}

The above two lemmas were general facts about $\PDer$. From now on, we will need to work under much stronger hypotheses, indeed, we will need to assume that our pre-derivator $\D$ is either representable or that it is a stable derivator. In fact, the unique point in which we will actively use these hypotheses is in the following lemma, which is the counterpart of \cite[\acor\textbf{2.9}]{RW}.

\begin{lemma}
\label{tfae_for_gamma}
Suppose $\D$ is either a representable pre-derivator or a stable derivator. 
Let $F\colon \D^\due\to \D$ be  a normal factorization pre\hyp{}algebra and suppose that there is an isomorphism $\alpha\colon FF^\due\to F\Delta^\circledast$. Then,  $m_{e_{\firstblank}} =FF^\due*(\mu')^\circledast$ and $e_{m_{\firstblank}} =FF^\due*(\nu')^\circledast$. Furthermore, the following conditions are equivalent
\begin{enumerate}[label=($\roman*$)]
\item $\alpha*v^\circledast=\id_{F}$;
\item $\alpha*l^\circledast=(e_{m_{\firstblank}} )^{-1}$;
\item $\alpha*h^\circledast=\id_{F}$;
\item $\alpha*r^\circledast=m_{e_{\firstblank}}$.
\end{enumerate}
\end{lemma}
\begin{proof}
By Lemma \refbf{iso_implies_EM}, $F$ is an \textsc{em} factorization and so, by our hypotheses on $\D$ and the results in §\refbf{sec:squaring}, $\Psi_\F$ is fully faithful. Now consider the following functors (compare with \eqref{diamond_def})
\[
\begin{matrix}\xymatrix{
&V\ar@{=>}[dr]^{\underline\nu}\\
R\ar@{=>}[dr]_{\underline\mu'}\ar@{=>}[ur]^{\underline\mu}&&L\\
&H\ar@{=>}[ur]_{\underline\nu'}
}\end{matrix}
\colon \tre\times \due\times \due\longrightarrow  \tre\times\due
\]
where
\[
R(a,b,c) \coloneqq \begin{cases}
(0,0)&\text{if $a=0$;}\\
(1,r(b,c))&\text{if $a=1$;}\\
(2,b)&\text{if $a=2$;}
\end{cases}
\qquad
 L(a,b,c) \coloneqq \begin{cases}
(0,b)&\text{if $a=0$;}\\
(1,l(b,c))&\text{if $a=1$;}\\
(2,1)&\text{if $a=2$;}
\end{cases}
\]
\[
V(a,b,c) \coloneqq \begin{cases}
(0,0)&\text{if $a=0$;}\\
(1,v(b,c))&\text{if $a=1$;}\\
(2,1)&\text{if $a=2$;}
\end{cases}
\qquad
 H(a,b,c) \coloneqq \begin{cases}
(0,0)&\text{if $a=0$;}\\
(1,h(b,c))&\text{if $a=1$;}\\
(2,1)&\text{if $a=2$;}
\end{cases}
\]
and where $\underline\mu$, $\underline\mu'$, $\underline\nu$ and $\underline\nu'$ are defined in the unique possible way. Also notice that $V$ and $H$ are constructed in such a way that $V\circ ((0,2)\times \id_{\due\times \due})=H\circ ((0,2)\times \id_{\due\times \due})$ and
\[
\underline\mu*((0,2)\times \id_{\due\times \due})=\underline\mu'*((0,2)\times \id_{\due\times \due}).
\]
Given $X\in \D^\due(I)$, let $Y \coloneqq (0\times\id_\due)_!X\in \D^{\tre\times \due}(I)$ and notice that
\begin{align*}
\Psi^\due_\F F^{\tre\times\due}*\underline \mu^\circledast_Y&=((0,2)\times\id_\due)^\circledast F^{\tre\times\due}*\underline \mu^\circledast_Y\\
&=F^{\due\times\due}((0,2)^\due\times\id_{\due\times\due})^\circledast*\underline \mu^\circledast_Y\\
&=F^{\due\times\due}((0,2)^\due\times\id_{\due\times\due})^\circledast*(\underline \mu')^\circledast_Y\\
&=((0,2)\times\id_\due)^\circledast F^{\tre\times\due}*(\underline \mu')^\circledast_Y=\Psi^\due_\F F^{\tre\times\due}*(\underline \mu')^\circledast_Y.
\end{align*}
Since $\Psi_\F$ is faithful, we get that $F^{\tre\times\due}*\underline \mu^\circledast_Y=F^{\tre\times\due}*(\underline \mu')^\circledast_Y$. So in particular,
\begin{align*}
F^\due*\mu^\circledast_X&=F^{\due}(1\times \id_{\due\times\due})^\circledast *\underline \mu^\circledast_Y=(1\times \id_\due)^\circledast F^{\tre\times\due}*\underline \mu^\circledast_Y\\
&=F^{\tre\times\due}*(\underline \mu')^\circledast_Y=(1\times \id_\due)^\circledast F^{\tre\times\due}*(\underline \mu')^\circledast_Y\\
&=F^{\due}(1\times \id_{\due\times\due})^\circledast *(\underline \mu')^\circledast_Y=F^\due*(\mu')^\circledast_X.
\end{align*}
A similar argument shows that $F^\due*\nu^\circledast =F^\due*(\nu')^\circledast$. With these equalities, it is not difficult to derive the equivalence among (i), (ii), (iii) and (iv) by just looking at the commutative diagram in the proof of Lemma \refbf{iso_implies_EM}.
\end{proof}

At this point we can finally prove our \athm\textbf{IV}. The arguments used in the proof are analogous to those of  \cite[\athm\textbf{2.10} and \athm\textbf{2.11}]{RW}.

\begin{proof}[Proof of \athm\textbf{IV}]
The implication (1)$\Rightarrow$(2) is trivial, while (2)$\Rightarrow$(3) follows by Lemma \refbf{iso_implies_EM}. Hence, we assume (3)  and we show how to construct an extended associator 
\[
\alpha\colon FF^\due \xto{\sim} F\Delta^\circledast
\] 
that satisfies conditions \refbf{def_factorization_alg}.(1--3). Consider the following commutative diagrams
\[
\vcenter{
\xymatrix{
& \Delta r \ar[dr]^{\Delta* \mu}\ar[dl]_{\epsilon_r}& \\
\id_{\due\times \due}\ar[dr]_{ \eta_l}&& \Delta h\ar[dl]^{ \Delta*\nu}\\
& \Delta l}
}
\qquad\qquad\vcenter
{\xymatrix{
& FF^\due (\Delta r)^\circledast \ar[dr]^{FF^\due * \mu^\circledast*\Delta^{\circledast}}\ar[dl]_{FF^\due * \epsilon_r^{\circledast}}& \\
FF^\due \ar[dr]_{FF^\due * \eta_l^\circledast}&& F\Delta^\circledast\ar[dl]^{FF^\due * \nu^\circledast*\Delta^{\circledast}}\\
& FF^\due (\Delta l)^\circledast
}}
\]
where the second one is obtained from the first one composing with $FF^\due$, after having applied $\D$. The \textsc{em} condition tells us the $FF^\due*\mu^\circledast$ and $FF^\due*\nu^\circledast$ are invertible. We claim that also the two remaining arrows in the above diagram are invertible. Let us give an argument just for $FF^\due * \epsilon_r^{\circledast}$ as $FF^\due * \eta_l^{\circledast}$ is an iso for formally dual reasons. By \refbf{ortho_to_self}, it is enough to show that $FF^\due * \epsilon_r^{\circledast}\in \E_F\cap \M_F$. Consider the following commutative diagrams:
\[
\xymatrix{
(0,0)\circ\pt\circ(\id_\due\times \pt)\ar[r]\ar[d]&((0,0),(0,1))\circ(\id_\due\times \pt)\ar[d]\\
\Delta\circ r\ar[r]&\id_{\due\times \due}}
\]
\[
\xymatrix{
(0,0)^\circledast\ar[r]^(.39){\star}\ar[d]_{\star}&F ((0,0),(0,1))^\circledast\ar[d]^{\star}\\
FF_r\Delta^\circledast\ar[r]&FF^2}
\]
where $(0,0)\colon \uno\to \due\times \due$ and $((0,0),(0,1))\colon \due\to\due\times \due$ select respectively the upper left corner and the left horizontal arrow in $\due\times \due$; furthermore, the arrows in the first diagram are the unique possible, while the second diagram is obtained from the first one composing with $FF^\due$, after having applied $\D$. Using the \textsc{em} condition, one can show that the arrows marked by ($\star$) are in $\E_F$ (since they are instances of the natural transformation $e_{-}$) and so, by the closure and cancellation properties of this class, also the remaining arrow, which is $FF^\due * \epsilon_r^{\circledast}$, does belong to $\E_F$. A similar argument but starting with the following diagram 
\[
\xymatrix{
\Delta\circ r\ar[r]\ar[d]&\id_{\due\times \due}\ar[d]\\
\Delta\circ (\id_\due\times \pt)\ar[r]&((1,0),(1,1))\circ(\id_\due\times \pt)}
\]
shows that $FF^\due * \epsilon_r^{\circledast}$ does belong to $\M_F$. We can now define 
\[
\alpha  \coloneqq  (FF^\due * \mu^\circledast*\Delta^\circledast) \circ (FF^\due * \epsilon^\circledast_r)^{-1}
\] 
and verify that it is the extended associator we are looking for. 
In fact, it is  easy to show that $\alpha$ satisfies one of the equivalent conditions of \refbf{tfae_for_gamma}, so that it satisfies conditions (1) and (2) of \refbf{def_factorization_alg}, as they are exactly (i) and (iii) of \refbf{tfae_for_gamma}.  
%
%
It remains to check condition \refbf{def_factorization_alg}.(3). 
For this, consider the following diagram of 2-cells
\begin{equation}\label{diff_coh_dia}
\xymatrix@C=2cm{
\ar[d]_{F*\alpha^\due} \ar[r]^{\alpha * F^{\due\times\due}}FF^\due F^{\due\times\due} & **[r] F\Delta^\circledast_\D F^{\due\times\due} \overset{\star}= FF^\due \Delta^\circledast_{\D^\due} \ar[d]^{\alpha  * \Delta^\circledast_{\D^\due}}\\
FF^\due (\Delta^\circledast_\D)^\due \ar[r]_{\alpha * (\Delta^\circledast_\D)^\due} & **[r] F\Delta^\circledast_\D \Delta^\circledast_{\D^\due} \overset{\star\star}= F\Delta^\circledast_\D(\Delta^\circledast_\D)^\due
}
\end{equation}
(the equality ($\star$) follows from naturality, and the equality ($\star\star$) follows from associativity). Notice that \refbf{def_factorization_alg}.(3) expresses exactly the commutativity of the above square. Thanks to \refbf{nuff_to_determine} (applied twice), it is enough to check that it commutes when composed in the north\hyp{}west corner with $l_{\D^\due}^{\circledast}l^{\circledast}$. Notice that, after composing with $l_{\D^\due}^{\circledast}$, the paths going from the north\hyp{}west to the south\hyp{}east corner are both 2-cells going from $FF^\due F^{\due\times \due}$ to $F\Delta_\D^\circledast \Delta^\circledast_{\D^\due}$. Let us fix an arbitrary 2-cell 
\[
\theta\colon FF^\due F^{\due\times \due}\to F\Delta_\D^\circledast \Delta^\circledast_{\D^\due}
\]
and let us record some of its general properties. We first consider 
\[
\xymatrix@C=2.5cm{
	\D^\due 
	\ar@/^1pc/[r]^{v^\circledast}
	\ar@/_1pc/[r]_{l^\circledast}
	\ar@{}[r]|{\Downarrow\nu^\circledast}
	& \D^{\due\times\due} 
	\ar@/^1pc/[r]^{FF^\due F^{\due\times\due} l_{\D^\due}^{\circledast}}
	\ar@/_1pc/[r]_{F\Delta^\circledast_\D}
	\ar@{}[r]|{\Downarrow\theta *l^{\circledast}_{\D^\due}}
	& \D
}
\]
We claim that the whiskering $FF^\due F^{\due\times\due} l^\circledast_{\D^\due} * \nu^{\circledast}$ is invertible, in fact,
\begin{align*}
FF^\due F^{\due\times\due} l^\circledast_{\D^\due}  * \nu^\circledast &\cong FF^\due l^\circledast  F^\due * \nu^\circledast\\
&= FF_lF^\due * \nu^\circledast \\
&\cong  FF^\due * \nu^\circledast
\end{align*}
and we have already noticed that this last 2-cell is invertible. Hence,
\[
\theta *l^{\circledast}_{\D^\due}l^\circledast=(F\Delta^\circledast_\D*\nu^\circledast)\circ (\theta *l^{\circledast}_{\D^\due}v^{\circledast})\circ(FF^\due F^{\due\times\due} l_{\D^\due}^{\circledast}*\nu^\circledast)^{-1}
\]
This shows that, given $\theta,\, \theta'\colon FF^\due F^{\due\times \due}\to F\Delta_\D^\circledast \Delta^\circledast_{\D^\due}$, then $\theta *l^{\circledast}_{\D^\due}l^\circledast=\theta' *l^{\circledast}_{\D^\due}l^\circledast$ if and only if $\theta *l^{\circledast}_{\D^\due}v^{\circledast}=\theta' *l^{\circledast}_{\D^\due}v^{\circledast}$. This reduces the verification of the commutativity of \eqref{diff_coh_dia}, to the proof of the following equality:
\[
((\alpha * \Delta_{\D^\due}^\circledast)\circ (\alpha * F^{\due\times\due})) *l^{\circledast}_{\D^\due}v^{\circledast}=((\alpha * (\Delta_\D)^\due)\circ (F * \alpha^\due))*l^{\circledast}_{\D^\due}v^{\circledast}
\]
To conclude notice that, by \refbf{tfae_for_gamma}.(i), both sides are equal to $\id_{F}$.
\end{proof}

\bibliographystyle{amsalpha}
\bibliography{refs}

\end{document}